\documentclass[12pt]{amsart}

\usepackage{amssymb, amscd, amsmath, float, graphicx, longtable,
  color, enumerate, bm, todonotes}
\usepackage{tikz-cd}
\usepackage[centertableaux,boxsize=1em]{ytableau}
\usepackage{fouriernc}\DeclareMathAlphabet{\mathcal}{OMS}{cmsy}{m}{n}
\usepackage{array}

\newcommand{\hide}[1]{}

\DeclareMathOperator{\charac}{char}

\DeclareMathOperator{\coh}{{{coh}}}

\DeclareMathOperator{\cok}{cok}

\DeclareMathOperator{\diag}{diag}
\DeclareMathOperator{\End}{End}
\DeclareMathOperator{\Ext}{Ext}

\DeclareMathOperator{\FM}{\mathsf{FM}}
\DeclareMathOperator{\GL}{GL}

\DeclareMathOperator{\Grass}{Grass}

\DeclareMathOperator{\Hom}{Hom}

\DeclareMathOperator{\id}{id}

\DeclareMathOperator{\im}{image}

\DeclareMathOperator{\Ind}{Ind}

\DeclareMathOperator{\rank}{rank}

\DeclareMathOperator{\Res}{Res}
\DeclareMathOperator{\Spec}{Spec}

\DeclareMathOperator{\Sym}{Sym}

\DeclareMathOperator{\Tr}{Tr}

\renewcommand{\phi}{\varphi}

\renewcommand{\bar}{\overline}
\renewcommand{\tilde}{\widetilde}

\renewcommand{\mod}{\operatorname{mod}}
\renewcommand{\leq}{\leqslant}
\renewcommand{\le}{\leqslant}
\renewcommand{\geq}{\geqslant}
\renewcommand{\ge}{\geqslant}

\renewcommand{\to}{\longrightarrow}

\newcommand\into{{\;\hookrightarrow\;}}
\newcommand{\xto}{\xrightarrow}
\newcommand\onto{\twoheadrightarrow}

\newcommand{\Wedge}{{\textstyle\bigwedge}}

\newcommand{\svee}{{\scriptscriptstyle \vee}}
\renewcommand{\hat}{\widehat}
\newcommand{\ihat}{{\hat\imath}}

\newcommand{\x}{{\underline x}}

\newcommand{\DD}{{\mathbb D}}

\newcommand{\LL}{{\mathbb L}}

\newcommand{\PP}{{\mathbb P}}
\newcommand{\GG}{{\mathbb G}}
\newcommand{\QQ}{{\mathbb Q}}

\newcommand{\SSS}{{\mathbb {S}}}

\newcommand{\VV}{{\mathbb V}}
\newcommand{\WW}{{\mathbb W}}

\newcommand{\YY}{{\mathbb Y}}

\newcommand{\calc}{{\mathcal C}}
\newcommand{\cald}{{\mathcal D}}
\newcommand{\cale}{{\mathcal E}}
\newcommand{\calf}{{\mathcal F}}

\newcommand{\calm}{{\mathcal M}}
\newcommand{\caln}{{\mathcal N}}
\newcommand{\calo}{{\mathcal O}}

\newcommand{\calq}{{\mathcal Q}}
\newcommand{\calr}{{\mathcal R}}
\newcommand{\cals}{{\mathcal S}}
\newcommand{\calt}{{\mathcal T}}
\newcommand{\calu}{{\mathcal U}}
\newcommand{\calv}{{\mathcal V}}

\newcommand{\caly}{{\mathcal Y}}
\newcommand{\calz}{{\mathcal Z}}

\newcommand{\calog}{{\calo_{\GG}}}

\newcommand{\caloz}{{\calo_\calz}}

\def\cEnd{{\mathcal E}\!\mathit{nd}}

\def\cHom{{\mathcal H}\!\mathit{om}}
\newcommand{\rHom}[1][{}]{{\mathbf R}^{#1}\!\Hom}
\newcommand{\R}{{\mathbf R}}

\newcommand{\bL}{{\mathbf {L}}}
\newcommand{\lotimes}{\mathbin{\overset{\bL}{\otimes}}}

\newcommand{\Ya}{{\ytableausetup{boxsize=0.5em}\ydiagram{1}\ytableausetup{boxsize=normal}}}
\newcommand{\Yab}{{\ytableausetup{boxsize=0.5em}\ydiagram{1,1}\ytableausetup{boxsize=normal}}}
\newcommand{\Yaa}{{\ytableausetup{boxsize=0.5em}\ydiagram{2}\ytableausetup{boxsize=normal}}}
\newcommand{\Yaab}{{\ytableausetup{boxsize=0.5em}\ydiagram{2,1}\ytableausetup{boxsize=normal}}}
\newcommand{\Yaabb}{{\ytableausetup{boxsize=0.5em}\ydiagram{2,2}\ytableausetup{boxsize=normal}}}

\theoremstyle{plain}
\newtheorem{theorem}{Theorem}

\newtheorem{prop}[theorem]{Proposition}
\newtheorem{proposition}[theorem]{Proposition}

\newtheorem{lemma}[theorem]{Lemma}

\newtheorem{cor}[theorem]{Corollary}
\newtheorem{corollary}[theorem]{Corollary}

\newtheorem*{theorem*}{Theorem}
\newtheorem*{prop*}{Proposition}

\theoremstyle{definition}

\newtheorem{definition}[theorem]{Definition}

\newtheorem{remark}[theorem]{Remark}

\newtheorem{example}[theorem]{Example}

\newtheorem{sit}[theorem]{}
\newenvironment{nsit}{\begin{sit}\textit}{\end{sit}}
\numberwithin{theorem}{section}
\numberwithin{equation}{theorem}
\numberwithin{figure}{section}

\begin{document}

\date{\today}

\title[Desingularization of Determinantal Varieties II]%
{Non-commutative Desingularization of Determinantal Varieties, II: Arbitrary minors}

\author[R.-O. Buchweitz]{Ragnar-Olaf Buchweitz}
\address{Dept.\ of Computer and Mathematical Sciences, University of
Toronto Scarborough, Toronto, Ont.\ M1C 1A4, Canada}
\email{\href{mailto:ragnar@utsc.utoronto.ca}{ragnar@utsc.utoronto.ca}}

\author[G.J. Leuschke]{Graham J. Leuschke}
\address{Dept.\ of Mathematics, Syracuse University,
Syracuse NY 13244, USA}
\email{\href{mailto:gjleusch@math.syr.edu}{gjleusch@math.syr.edu}}
\urladdr{\url{http://www.leuschke.org/}}

\author[M. Van den Bergh]{Michel Van den Bergh}
\address{Departement WNI, Universiteit Hasselt, 3590
  Diepenbeek, Belgium} 
\email{\href{mailto:michel.vandenbergh@uhasselt.be}{michel.vandenbergh@uhasselt.be}}

\thanks{The first author was partly supported by NSERC grant
  3-642-114-80. The second author was partly supported by NSF grant
  DMS~0902119.  The third author is director of research at the FWO\@.
  Part of this research was supported through the programme ``Research
  in Pairs'' by the Mathematisches Forschungsinstitut Oberwolfach in
  2012, and part was carried out at the Mathematical Sciences Research
  Institute in 2013 with the support of the National Science
  Foundation under Grant No.\ 0932078 000. We thank both institutes for
  their hospitality}


\begin{abstract}
  In our paper ``Non-commutative desingularization of determinantal
  varieties I'' we constructed and studied non-comm\-u\-ta\-tive resolutions
  of determinantal varieties defined by maximal minors. At the end of
  the introduction we asserted that the results could be generalized
  to determinantal varieties defined by non-maximal minors, at least
  in characteristic zero.  In this paper we prove the \emph{existence}
  of non-commutative resolutions in the general case in a manner which
  is still characteristic free, and carry out the explicit description
  by generators and relations in characteristic zero.  As an
  application of our results we prove that there is a fully faithful
  embedding between the bounded derived categories of the two
  canonical (commutative) resolutions of a determinantal variety,
  confirming a well-known conjecture of Bondal and Orlov in this
  special case.
\end{abstract}
\maketitle

\setcounter{tocdepth}{1}
\noindent\makebox[\textwidth][c]{%
\begin{minipage}{.75\linewidth}
\footnotesize
\tableofcontents
\normalsize
\end{minipage}
}

\section{Introduction}

Let $K$ be a field and let $F$, $G$ be two $K$-vector spaces of ranks
$m$ and~$n$ respectively. We take unadorned tensor products over $K$
and denote by $(-)^{\svee}$ the $K$-dual. Put $H=\Hom_K(G,F)$, viewed
as the affine variety of $K$-rational points of $\Spec S$, where $S =
\Sym_K(H^{\svee})$ is isomorphic to a polynomial
ring in $m n$ indeterminates.  The \emph{generic $S$-linear map} $\phi
\colon G \otimes S \to F \otimes S$ corresponds to multiplication by
the generic $(m \times n)$-matrix comprising those indeterminates.

Fix a non-negative integer $l<\min(m,n)$, and let $\Spec R$ be the
locus in $\Spec S$ where $\Wedge^{l+1}\phi=0$.  Then $R$ is the
quotient of $S$ by the ideal of $(l+1)$-minors of the generic $(m \times
n)$-matrix.  It is a classical result that~$R$ is Cohen-Macaulay of
codimension $(n-l)(m-l)$, with singular locus defined by the
$l$-minors of the generic matrix; in particular $R$ is smooth in
codimension $2$.

In this paper we consider some natural $R$-modules.  For a partition
$\alpha = (\alpha_1, \ldots, \alpha_r)$ and a vector
space $V$, write 
\[
\Wedge^\alpha V = \Wedge^{\alpha_1} V \otimes \cdots \otimes
\Wedge^{\alpha_r} V\,.
\]
Let $\alpha'$ denote the conjugate partition 
of $\alpha$, and
$\Wedge^{\alpha'} \phi^\svee \colon \Wedge^{\alpha'} F^\svee\otimes S
\to \Wedge^{\alpha'}G^\svee \otimes S$ the natural map induced by
$\phi$.  Define
\begin{equation*}
  T_\alpha = \im \left(\Wedge^{\alpha'} F^\svee \otimes R
    \xto{\left(\Wedge^{\alpha'}\phi^\svee\right) \otimes R\ }
    \Wedge^{\alpha'}G^\svee 
    \otimes R\right)\,.
\end{equation*}
Let $B_{u,v}$ be the set
of all partitions with at most $u$ rows and at most $v$ columns and
set
\begin{equation*}
T = \bigoplus_{\alpha \in B_{l,m-l}} T_\alpha \qquad \text{and} \qquad
E = \End_R(T)\,.
\end{equation*}

Our first main result generalizes the case $l=m-1$~\cite[Theorem
A]{Buchweitz-Leuschke-VandenBergh:2010}, and shows that general
determinantal varieties admit a \emph{non-commutative
  desingularization} in the following sense. 
{\def\thetheorem{A}
\begin{theorem}
  \label{thm:desing}
  For $m \leq n$, the endomorphism ring $E= \End_R(T)$ is
  maximal Cohen-Macaulay as an $R$-module, and has moreover finite
  global dimension.
  In particular $T_\alpha$ is a maximal Cohen-Macaulay $R$-module for
  each $\alpha \in B_{l,m-l}$.
\end{theorem}
}

If $m=n$ then $R$ is Gorenstein; in this case $E$ is an example of a
\emph{non-commutative crepant resolution} as defined
in~\cite{VandenBergh:crepant}. 

The $R$-module $T_\alpha$ is in general far from indecomposable.
Assume for a moment that $K$ has characteristic zero and denote by
$L^\alpha V$ the irreducible $\GL(V)$-module with highest weight
$\alpha$ (a.k.a.\ Schur module~\cite{Weyman:book}). It then follows
from the Pieri rule that $\Wedge^{\alpha'}V = L^\alpha V \oplus W$,
where $W$ is a direct sum of certain $L^\beta V$ with $\beta < \alpha$
for the natural order on partitions. Hence if we put
\begin{equation*}
N_\alpha=\im\left(L^\alpha (F^\svee)\otimes
  R\xrightarrow{\ \left(L^\alpha (\phi^\svee)\right) \otimes R\ }
  L^{\alpha}(G^\svee)\otimes  R\right) 
\end{equation*}
then in characteristic zero $T_\alpha$ is a direct sum of $N_\beta$
for $\beta\le\alpha$ with $N_\alpha$ appearing with multiplicity one.
In particular we obtain that $N_\alpha$ is maximal Cohen-Macaulay.
This is false in small characteristic; see
Remark~\ref{rmk:Nalpha-Weyman} below where we make the connection with
the work of Weyman~\cite[\S6]{Weyman:book}.

If we set $N = \bigoplus_{\alpha \in B_{l,m-l}} N_\alpha$ and
$A= \End_R(N)$, then $A$ is Morita equivalent to
$E=\End_R(T)$. Clearly Theorem~\ref{thm:desing} remains valid in
characteristic zero if we replace $E$ by $A$.  Furthermore, we have
the following description by generators and relations of the
non-commutative desingularization $A$. 
Write $\alpha \nearrow \beta$ if $\beta$ is obtained by adding a box
to $\alpha$.

{\def\thetheorem{B}
\begin{theorem}[Theorem~\ref{thm:quiveriso}]
  \label{thm:quiverintro}
  Assume that $K$ has characteristic zero and $m-l>1$.  As a
  $K$-algebra, $A$ is isomorphic to the bound path algebra of the
  truncated Young quiver (Fig.~\ref{fig:Young}) having vertices
  $\alpha \in B_{l,m-l}$ and arrows $\alpha \to \beta$ indexed by
  bases for
  \[
  \begin{cases}
    F^\svee & \text{if
      $\alpha\nearrow \beta$, and}\\
    G & \text{if $\beta \nearrow
      \alpha$,}
  \end{cases}
  \]  
   with vector spaces of relations between two vertices $\alpha,
   \gamma \in B_{l,m-l}$ given by
  \begin{equation*}
  \begin{cases}
    \Sym_2 F^\svee & \text{if } \gamma \nearrow \nearrow \alpha, \text{ two boxes in a column}\\
    \Wedge^2 F^\svee &  \text{if }\gamma \nearrow \nearrow \alpha, \text{ two boxes in a row} \\
    \Sym_2 F^\svee \oplus \Wedge^2 F^\svee \cong F^\svee \otimes
      F^\svee &  \text{if }\gamma \nearrow \nearrow \alpha , \text{ two  disconnected boxes} \\
    F^\svee\otimes G &  \begin{minipage}{2.5in} if $\alpha \neq \gamma$  and $\alpha \nearrow
      \beta$,\ $\gamma \nearrow \beta$ \\ \phantom{asdf}for some $\beta$ with 
      $\beta_1\leq m-l$\end{minipage}\\
    (F^\svee\otimes G)^{\oplus(t(\alpha)-1)} &  \text{if }\alpha =\gamma\\
    \Sym_2 G &  \text{if }   \alpha \nearrow \nearrow \gamma, \text{ two
      boxes in a column} \\
    \Wedge^2 G &  \text{if }   \alpha \nearrow \nearrow \gamma, \text{ two
      boxes in a row}\\
    \Sym_2 G \oplus \Wedge^2 G \cong G \otimes G &  \text{if }   \alpha \nearrow \nearrow \gamma, \text{
      two disconnected boxes}
  \end{cases}
  \end{equation*}
  where $t(\alpha)$ is the number of ways to add a box to $\alpha$
  without making any row longer than $m-l$.  
\end{theorem}
}
   \begin{figure}[htbp]
     \includegraphics{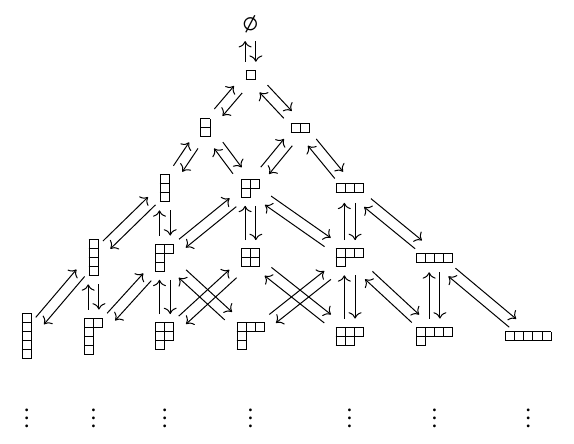}
     \caption{
     }
     \label{fig:Young}
   \end{figure}
Note that the representations defining the relations listed above,
for example $F^\svee\otimes F^\svee \subset (F^\svee \otimes
F^\svee)^{\oplus 2}$, are
\emph{not} induced by the obvious diagonal inclusions; there are some
non-trivial scalars appearing. A precise description of the relations
with explicit scalars is given in Theorem~\ref{thm:explicitrels}.

Now let $K$ be general again.  We have taken care to state
Theorems~\ref{thm:desing} and~\ref{thm:quiverintro} in algebraic
language but as in~\cite{Buchweitz-Leuschke-VandenBergh:2010} the
proofs proceed by invoking algebraic geometry, i.e.\ by constructing a
suitable tilting bundle on the Springer resolution of $\Spec R$.

Write $\GG = \Grass(l,F) \cong \Grass(l,m)$ for the Grassmannian
variety of $l$-dimensional subspaces of $F$, and let $\pi \colon \GG
\to \Spec K$ be the structure morphism to the base scheme $\Spec K$.  On
$\GG$ we have a tautological exact sequence of vector bundles
\begin{equation}
  \label{eq:taut}
  0 \to \calr \to \pi^* F^\svee \to \calq \to 0
\end{equation}
whose fiber above a point $(V \subset F) \in \GG$ is the short exact
sequence $0 \to (F/V)^\svee \to F^\svee \to V^\svee \to
0$. In~\cite{Buchweitz-Leuschke-VandenBergh:Grass} we proved that the
$\calog$-module
\[
\calt_0 = \bigoplus_{\alpha \in B_{l,m-l}} \Wedge^{\alpha'}\calq
\]
is a tilting bundle on $\GG$.  From this we derive our main geometric
result as follows.
 Set $\caly = \GG
\times_{\Spec K} H$, with the canonical projections $p\colon \caly \to
\GG$ and $q\colon \caly \to H$.  Define the \emph{incidence variety}
\[
\calz = \left\lbrace \,(V,\theta) \in \GG\times_{\Spec K}
H\;\middle\vert\;\im \theta\subset V\,\right\rbrace  \subseteq \caly
\]
and denote by $j$ the natural inclusion $\calz \to \caly$. The
composition $q' = q j\colon \calz \to H$ is then a birational
isomorphism from $\calz$ onto its image $q'(\calz) = \Spec R$, while
$p' =p j\colon \calz\to \GG$ is a vector bundle (with zero section
$\theta=0$). 
\begin{figure}[htbp]
\[
\begin{tikzcd}
\calz  \arrow[hook]{dr} {j} \arrow[bend left=15]{drr}{p'} \arrow{dd}[swap]{q'}\\
& \caly = H\times \GG \arrow{r}{p} \arrow{d}[swap]{q} & \GG=\Grass(l,F)\arrow{d}{\pi}\\
\Spec R \arrow[hook]{r} & H = \Hom_K(G,F) \arrow{r} & 
 \Spec K
\end{tikzcd}
\]
\caption{
}
\label{fig:Springer}
\end{figure}
 Figure~\ref{fig:Springer} summarizes the schemes and
maps we have defined.  We call $\calz$ the \emph{Springer resolution}
of $\Spec R$.

{\def\thetheorem{C}
\begin{theorem}
  \label{thm:tiltZ}
  The $\caloz$-module
  \[ 
  \calt ={p'}^{\ast}\left(\bigoplus_{\alpha\in B_{l,m-l}}
    \Wedge^{\alpha'}\calq\right)
  \]
  is a classical tilting bundle on $\calz$, i.e.\
  \begin{enumerate}[\quad(i)]
  \item $\calt$ classically generates the derived category
    $\cald^b(\coh \calz)$, in that the smallest thick subcategory of
    $\cald^b(\coh \calz)$ containing $\calt$ is $\cald^b(\coh \calz)$, and
  \item $\Hom_{\cald^b(\coh \calz)}(\calt,\ \calt[i]) = 0$ for $i \neq 0$.
  \end{enumerate}
  Furthermore we have
  \begin{enumerate}[\quad(i)]\setcounter{enumi}{2}
  \item $T_\alpha \cong \R q'_* \Wedge^{\alpha'} \calq$ for each
    $\alpha \in B_{l,m-l}$, so that $T\cong \R q'_\ast \calt$, and 
  \item $E\cong \End_{\calo_\calz}(\calt)$.
  \end{enumerate} 
\end{theorem}
}

The proofs of Theorems~\ref{thm:desing} and~\ref{thm:tiltZ} are
substantially simpler than the corresponding ones
in~\cite{Buchweitz-Leuschke-VandenBergh:2010}, even in the case
treated there of maximal minors.

\medskip

As $H=\Hom_K(G,F)$ is canonically isomorphic to
$\Hom_K(F^\svee,G^\svee)$ we obtain a second Springer resolution
$q'_2\colon \calz_2 \to \Spec R$ by replacing $(F,G)$ with
$(G^\svee,F^\svee)$. Put $\hat\calz=\calz\times_H\calz_2$.  As an
application of Theorem~\ref{thm:tiltZ}, we prove the following.  {%
  \def\thetheorem{D}
  \begin{theorem} \label{thm:BO} If $m\le n$ then the Fourier-Mukai
    transform with kernel $\calo_{\hat\calz}$ induces a fully faithful
    embedding $\FM\colon\cald^b(\coh \calz )\into \cald^b(\coh \calz_2
    )$.  If $m=n$ then $\FM$ is an equivalence.
\end{theorem}
}
A general conjecture by Bondal and Orlov~\cite{Bondal-Orlov:2002}
asserts that a flip between algebraic varieties induces a fully
faithful embedding between their derived categories.  It is not hard
to see that the birational map $\calz_2 \dashrightarrow \calz$ is a flip, so we
obtain a confirmation of the Bondal-Orlov conjecture in this special
case.

The first half of the paper is characteristic-free.  We include a
short section recalling the results we need
from~\cite{Buchweitz-Leuschke-VandenBergh:Grass}, as well as some
background on characteristic-free versions of the Cauchy formula and
Littlewood-Richardson rule.  These are used to prove
Theorem~\ref{thm:tiltZ}, and as a consequence
Theorem~\ref{thm:desing}, in the third section.  The fourth section
contains the proof of Theorem~\ref{thm:BO}.

In the second half, we specialize to characteristic zero.
Section~\ref{sect:simples} contains the calculation of the $\Ext$
groups between the simple $A$-modules, which will be used in
Section~\ref{sect:Young-Pieri} to construct an isomorphism between $A$
and the path algebra with relations of the truncated Young quiver
$\YY_{l,m-l}$ in Theorem~\ref{thm:quiverintro}.  The relations on the
path algebra of $\YY_{l,m-l}$ are induced by relations between certain
maps occurring in Pieri's formula.  Such maps were first considered by
Olver~\cite{Olver}; we give an independent analysis in
Section~\ref{sect:pierisystems} and show how to compute the relevant
relations and thereby the scalars appearing in
Theorem~\ref{thm:quiverintro}.  The first non-trivial example $(m,n,l)
= (4,4,2)$ is worked out in Section~\ref{sect:G24}.

We include an Appendix giving an alternative description of the
non-commutative desingularization as a ``quiverized Clifford algebra''
as in our earlier paper~\cite{Buchweitz-Leuschke-VandenBergh:2010}.

\medskip 

Since the original version of this article was posted on the arXiv,
similar results have been obtained by other
authors~\cite{Weyman-Zhao, Donovan:2011, Donovan-Segal:2012,
  Ballard-Favero-Katzarkov:2013}.

\medskip

We are grateful to Vincent Franjou, Catharina Stroppel, and Antoine
Touz\'e for help with references, and to Jerzy Weyman, Steven V Sam,
and Gufang Zhao for interesting conversations.

\section{Preliminaries}\label{sect:notation}

We recall two results from~\cite{Buchweitz-Leuschke-VandenBergh:Grass}.  Recall that we write
$L^\alpha$ for the Schur functors; our conventions are that $L^{(t)} V
= \Sym_t V$ and $L^{(1^t)}V = \Wedge^t V$.\footnote{This convention
  differs from that in~\cite{Weyman:book}.  Our indexing is such that
  $L^\alpha$ has highest weight $\alpha$.}  The functors $L^\alpha$
are defined for all \emph{dominant weights,} that is, weakly
decreasing sequences of integers.  A \emph{partition} is a dominant weight
with non-negative entries.

\begin{theorem}[{\cite[Theorem 1.2]{Buchweitz-Leuschke-VandenBergh:Grass}}]
  \label{thm:tiltG}
  The $\calog$-module
  \[
  \calt_0 = \bigoplus_{\alpha \in B_{l,m-l}} \Wedge^{\alpha'} \calq
  \]
  is a classical tilting bundle on $\GG$, i.e.\
  \begin{enumerate}[\quad (i)]
  \item $\calt_0$ classically generates the derived category
    $\cald^b(\coh \GG)$, in that $\cald^b(\coh \GG)$ is the smallest
    thick subcategory of itself containing $\calt_0$, and
  \item $\Hom_{\cald^b(\coh \GG)}(\calt_0,\calt_0[i]) = 0$ for $i \neq
    0$.\qed
  \end{enumerate}
\end{theorem}

\begin{prop}[{\cite[Prop. 1.3]{Buchweitz-Leuschke-VandenBergh:Grass}}]
  \label{prop:work}
  Let $\alpha \in B_{l,m-l}$ and let $\delta$ be any partition.  Then
  for all $i >0$ one has
  \[
  H^i\left(\GG, \left(\Wedge^{\alpha'} \calq\right)^\svee
    \otimes_\calog L^\delta \calq\right) = 0\,.\qed
  \]
\end{prop}

We also state for easy reference the following characteristic-free
versions of the Cauchy formula and the Littlewood-Richardson rule.
See~\cite[(2.3.2), (2.3.4)]{Weyman:book}. 

\begin{theorem}[Boffi~\cite{Boffi:1988},
  Doubilet-Rota-Stein~\cite{Doubilet-Rota-Stein:1974}]
  \label{thm:cauchyLR}
  Let $V$ and $W$ be $K$-vector spaces and let $\alpha$ and $\beta$ be
  dominant weights.
  \begin{enumerate}[\quad(i)]
  \item \label{item:cauchy} There is a natural filtration on $\Sym_t(V
    \otimes W)$ whose associated graded object is a direct sum with
    summands tensor products $L^\gamma V \otimes L^{\gamma'} W$ of Schur
    functors.
  \item There is a natural filtration on $\Wedge^t(V \otimes W)$ whose
    associated graded object is a direct sum with summands tensor
    products $L^\gamma V \otimes (L^{\gamma'} W^\svee)^\svee$ of Schur functors.
  \item \label{item:LR} There is a natural filtration on $L^\alpha V
    \otimes L^\beta V$ whose associated graded object is a direct sum
    of Schur functors $L^\gamma V$. The $\gamma$ that appear, and
    their multiplicities, can be computed using the usual
    Littlewood-Richardson rule.
  \end{enumerate}
  If $\charac K=0$ then the filtrations above degenerate to direct
  sums.  Note that in characteristic zero $(L^{\gamma'} W^\svee)^\svee
  \cong L^{\gamma'}W$.
\end{theorem}

\section{A tilting bundle on the resolution}\label{sect:proofsAC}

To prove Theorem~\ref{thm:tiltZ}, keep all the notation introduced
there.  One easily verifies that
\[
\calz=\underline{\Spec} \left(\Sym_{\GG}(G\otimes \calq)\right)\,;
\]
indeed, a closed point of the right-hand side consists of a pair $(V
\subset F, \theta)$, where $(V \subset F) \in \GG$ and $\theta$ is an
element of the fiber of $(G \otimes \calq)^\svee$ over the point $(V
\subset F)$.  That fiber is $(G \otimes V^\svee)^\svee = \Hom_K(G,V)
\subset \Hom_K(G,F)$, so the pair $(V, \theta)$ is precisely a point of
$\calz$. 

We have  $\calt_0 = \bigoplus_{\alpha\in B_{l,m-l}}
\Wedge^{\alpha'} \calq$, a tilting bundle on $\GG$ by
Theorem~\ref{thm:tiltG}.  Set $\calt = {p'}^{\ast} \calt_0$, a vector
bundle on $\calz$.

\begin{prop}
  \label{prop:tiltZ}
  The $\caloz$-module $\calt = {p'}^{\ast}\calt_0$ is a tilting
  bundle on $\calz$.
\end{prop}

\begin{proof}
  Since $\calt_0$ classically generates $\cald^b(\coh \GG)$ and $p'$
  is an affine morphism, it is easy to see that $\calt$ classically
  generates $\cald^b(\coh \calz)$, so it remains to prove
  $\Ext$-vanishing.  We have
  \[
  \Ext^i_{\caloz}(\calt,\calt)=H^i(\GG,\Sym_{\GG}(G\otimes
  \calq)\otimes_{\calog} \cEnd_{\calog} (\calt_0))
  \]
  and hence we need to prove that
  \begin{equation}
    \label{eq:tiltZ-van}
    \Sym_{\GG}(G\otimes \calq)\otimes_{\calog} \cHom_{\calog}
    (\Wedge^{\alpha'} \calq,\Wedge^{\beta'} \calq)
  \end{equation}
has vanishing higher cohomology for $\alpha,\beta\in B_{l,m-l}$.

Using Theorem~\ref{thm:cauchyLR} we find that \eqref{eq:tiltZ-van} has
a filtration whose associated graded object is a direct sum of vector
bundles of the form
\begin{equation}
  \label{eq:somewhatarbitrary}
  (\Wedge^{\alpha'}\calq)^\svee\otimes_{\calog} L^\delta\calq
\end{equation}
where $\alpha\in B_{l,m-l}$ and $\delta$ is some partition containing
$\beta$.  It now suffices to invoke Proposition~\ref{prop:work}.
\end{proof}

To prove the rest of Theorem~\ref{thm:tiltZ}, we shall show that
$\End_R(\R q'_* \calt) = \R q'_* \cEnd_{\caloz}(\calt)$,
and that the latter is MCM and has finite global dimension.  Put 
\[
\cale = \cEnd_{\caloz}(\calt)\,,
\]
and let $\omega_\calz$ be the dualizing sheaf of $\calz$.

\begin{lemma}
  \label{lem:calE-MCM}
  Assume $m \leq n$.  Then $\Ext_{\caloz}^i(\cale, \omega_\calz) =0$
  for all $i > 0$.
\end{lemma}

\begin{proof}
  We have $\cale = {p'}^{*} \cale_0$, with $\cale_0 =
  \cHom_\calog(\calt_0,\calt_0)$.  Substituting this and using the
  fact that $\cale_0$ is self-dual, we find
  \begin{align*}
    \Ext_\caloz^i(\cale,\omega_\calz) 
    &= \Ext_\caloz^i({p'}^{*}\cale_0,\omega_\calz) \\
    &= \Ext_\calog^i(\cale_0,p'_* \omega_\calz)\\
    &= H^i(\GG, \cale_0\otimes_\calog p'_* \omega_\calz)\,.
  \end{align*}
  Hence to continue we must be able to compute $p'_*\omega_\calz$.
  Since $\calz = \underline{\Spec} \left(\Sym_{\GG}(G\otimes
    \calq)\right)$, the standard expression, see e.g.\ \cite[Exercise
  III.8.4]{Hartshorne}, for the dualizing sheaf of a symmetric algebra
  gives
  \[
  p'_* \omega_\calz = 
  \omega_\GG \otimes_\caloz \Wedge^{l n} (G \otimes \calq) 
  \otimes_\caloz \Sym_\GG(G \otimes \calq)\,.
  \]
  Furthermore the sheaf $\Omega_\GG$ of differential forms on $\GG$ is
  known to be given by $\Omega_\GG = \calq^\svee \otimes_\calog
  \calr$, where $\calr$ is the tautological sub-bundle of $\pi^*
  F^\svee$ as in~\eqref{eq:taut}.  Hence $\omega_\GG = \Wedge^{l n}
  (\calq^\svee \otimes_\calog \calr)$ and so
  \[
  p'_* \omega_\calz = \Wedge^{l n} (\calq^\svee \otimes_\calog \calr)
  \otimes_\calog \Wedge^{l n} (G \otimes \calq) \otimes_\calog
  \Sym_\GG(G \otimes \calq)\,.
  \]
  Rewriting all the exterior powers in terms of $\calq$, we find
  \begin{equation*}
    \begin{split}
      &\Wedge^{l n} (\calq^\svee \otimes \calr ) \otimes
      \Wedge^{l n} (G \otimes \calq)\\
      &\qquad\qquad= \left(\Wedge^l \calq\right) ^{-m+l} \otimes
      \left(\Wedge^{m-l}\calr\right)^l \otimes \left(\Wedge^n G\right)^l
      \otimes \left(\Wedge^l \calq\right)^n \\
      &\qquad\qquad= \left(\Wedge^l \calq\right) ^{-m+l} \otimes
      \left(\Wedge^{m}F\right)^{-l} \otimes \left(\Wedge^l \calq\right)^{-l}
      \otimes \left(\Wedge^n G\right)^l
      \otimes \left(\Wedge^l \calq\right)^n \\
      &\qquad\qquad= \left(\Wedge^l \calq\right) ^{n-m} \otimes
      \left(\Wedge^{m} F\right)^{-l} \otimes \left(\Wedge^n G\right)^l\,.
    \end{split}
  \end{equation*}
  So finally 
  \[
  \cale_0 \otimes_\calog p'_* \omega_\calz = 
  \left(\Wedge^{m} F\right)^{-l} \otimes \left(\Wedge^n G\right)^l
  \otimes \cale_0 \otimes_\calog \left(\Wedge^l \calq \right)^{n-m}
  \otimes_\calog \Sym_\GG(G \otimes \calq)\,.
  \]
  Discarding the copies of the vector spaces $\Wedge^m F$ and
  $\Wedge^n G$, we find a direct sum of vector bundles of the form
  \[
  \Wedge^{\alpha'} \calq^\svee \otimes_\calog \Wedge^\beta \calq
  \otimes_\calog \left(\Wedge^l \calq \right)^{n-m} \otimes_\calog
  \Sym_\GG (G \otimes \calq)\,,
  \]
  which (since $m \leq n$) are the subject of
  Proposition~\ref{prop:work}. 
\end{proof}

Next we verify Theorem~\ref{thm:tiltZ} for 
\[
\bar E = \End_{\caloz}(\calt) =\Gamma(\calz,\cale)
\qquad\text{and}\qquad 
\bar T = \Gamma(\calz,\calt)\,.
\]
Recall the following consequence of tilting (see
e.g.~\cite{Hille-VdB:2007}).

\begin{proposition}
  \label{prop:HVdB}
  Assume that $\calt$ is a tilting bundle on a smooth variety
  $X$. Then $\rHom_{\calo_X}(\calt,-)$ defines an equivalence of
  derived categories $\cald^b(\coh X )\cong \cald^b(\mod E)$ where
  $E=\End_{\calo_X}(\calt)$.  If $X$ is projective over an
  affine variety then $E$ is finite over its center and has finite
  global dimension.\qed
\end{proposition}

\begin{proposition}
  \label{prop:thmA}
  Assume $m \leq n$.  Then
  \begin{enumerate}[\quad(i)]
  \item $\bar E \cong \End_R(\bar T)$;
  \item $\bar E$ and $\bar T$ are MCM $R$-modules; and
  \item $\bar E$ has finite global dimension.
  \end{enumerate}
\end{proposition}

\begin{proof} 
  That $\bar{E}$ has finite global dimension follows from
  Propositions~\ref{prop:tiltZ} and~\ref{prop:HVdB}.  Since
  $\Ext_\caloz^i(\calt,\calt)=0$ for $i>0$ by
  Proposition~\ref{prop:tiltZ}, the higher direct images of $\cale$
  vanish, i.e.
  \begin{equation*}
    \R q'_\ast\cale = q'_\ast \cale = \bar E\,.
  \end{equation*}
  To prove that $\bar E$ is MCM we must show that $\Ext_R^i(\bar
  E,\omega_R)=0$ for $i > 0$, where $\omega_R$ is the dualizing module
  for $R$.  Replacing $\bar E$ by $\R q'_*\cale$ and using duality for
  the proper morphism $q'$~\cite[1.2.22]{Weyman:book}, we see that
  this is equivalent to showing $\Ext_\caloz^i(\cale, {q'}^!\omega_R)=0$
  for $i>0$.  But ${q'}^!\omega_R = \omega_\calz$ is the dualizing
  sheaf for $\calz$, so Lemma~\ref{lem:calE-MCM} implies that $\bar E$
  is MCM.

  As $\caloz$ is a direct summand of $\calt$ we see that $\bar{T}$ is
  a summand of $\bar{E}$, whence $\bar{T}$ is Cohen-Macaulay as
  well. Furthermore we have an obvious homomorphism
  $i\colon\End_{\caloz}(\calt)\to \End_R(\bar{T})$ between reflexive
  $R$-modules, which is an isomorphism on the locus where $q' \colon
  \calz\to \Spec R$ is an isomorphism. The complement of this locus is
  given by the matrices which have rank $<l$, a subvariety of
  $\Spec R$ of codimension $\ge 2$.  Hence $i$ is an isomorphism.
\end{proof}

Propositions~\ref{prop:tiltZ} and~\ref{prop:thmA} imply
Theorems~\ref{thm:desing} and~\ref{thm:tiltZ} provided we can show
$T \cong \bar T$. We do this next. Recall that for a partition
$\alpha$ we denote
\[
N_\alpha = \im \left(L^\alpha (F^\svee) \otimes R
  \xto{\left(L^\alpha(\phi^\svee)\right) \otimes R} L^\alpha
  (G^\svee)\otimes R\right)\,.
\]
Set $\caln_\alpha = {p'}^*L^\alpha \calq$.

\begin{proposition}
  \label{prop:image-sections}
  With notation as above, we have 
  \[
  N_\alpha \cong \Gamma (\calz,\ \caln_\alpha)\,.
  \]
\end{proposition}

\begin{proof}
  With $\phi \colon G \otimes S \to F \otimes S$ the generic map
  defined over $S$, let $\psi = j^*q^* \phi$ be the map induced over
  $\calz$.  Then the fiber of $\psi^\svee$ over a point $(V, \theta)$
  factors as
  \[
  F^\svee\to V^\svee\to G^\svee 
  \]
  where the first map is the dual of the given inclusion $V \into F$.
  Thus $\psi^\svee$ factors as
  \[
  {p'}^\ast\pi^\ast F^\svee 
  \to 
  {p'}^{\ast}\calq
  \to
  {p'}^{\ast}\pi^\ast G^\svee\,.
  \]
  The first map is obviously surjective. The second map is injective
  since it is a map between vector bundles which is generically
  injective.  Schur functors preserve epimorphisms and monomorphisms
  of vector bundles~\cite[\S8.1]{Fulton:1997}, so we get an epi-mono
  factorization
  \[
  L^\alpha(\psi^\svee) \colon 
  L^\alpha ({p'}^{\ast}\pi^\ast F^\svee)
  \to 
  L^\alpha {p'}^{\ast}\calq
  \to
  L^\alpha ({p'}^{\ast}\pi^\ast G^\svee)\,.
  \]
  To prove the claim it is clearly sufficient to show that the first
  map remains an epimorphism after applying $q'_*$, i.e.\ that the
  epimorphism 
  \[
  \pi^* L^\alpha (F^\svee) \otimes_\calog \Sym_\GG (G \otimes
  \calq)
  \to 
  L^\alpha \calq \otimes_\calog \Sym_\GG (G \otimes \calq)
  \]
  remains an epimorphism upon applying $\Gamma(\GG,-)$.
  In fact it suffices to show that 
  \[
  \pi^* \left(L^\alpha (F^\svee) \otimes_\calog \Sym_\GG (G \otimes
  F^\svee)\right)
  \to 
  L^\alpha \calq \otimes_\calog \Sym_\GG (G \otimes \calq)
  \]
  remains an epimorphism upon applying $\Gamma(\GG,-)$.  By
  Theorem~\ref{thm:cauchyLR}, source and target are filtered by Schur
  functors, so it is enough to show that for any partition $\delta$
  the canonical map
  \[
  \pi^* L^\delta (F^\svee) \to L^\delta \calq
  \]
  remains an epimorphism upon applying $\Gamma(\GG,-)$.  But taking
  global sections of this map gives
  \[
  L^\delta(F^\svee) \to \Gamma(\GG, L^\delta\calq)
  \]
  which is even an isomorphism by the definition of Schur modules.
  Hence we are done.
\end{proof}

Set ${\bar T}_\alpha = \Gamma(\calz, \calt_\alpha)$, where
$\calt_\alpha = p'^*(\Wedge^{\alpha'} \calq)$ as in
Theorem~\ref{thm:tiltG}, and recall
\[
T_\alpha = \im\left(\Wedge^{\alpha'} (F^\svee) \otimes R
  \xto{(\Wedge^{\alpha'}\phi^\svee)\otimes R}
  \Wedge^{\alpha'}(G^\svee) \otimes R\right)\,.
\]
Filtering everything by Schur functors and applying
Proposition~\ref{prop:image-sections}, we see that these coincide:

\begin{corollary}\label{cor:T-Tbar}
  We have $T_\alpha \cong {\bar T}_\alpha$ for each $\alpha \in B_{l,
    m-l}$.  In particular $T \cong \bar T$ is a maximal Cohen-Macaulay
  $R$-module.\qed
\end{corollary}

Assembling the pieces, we obtain Theorem~\ref{thm:tiltZ} and, as a
consequence, Theorem~\ref{thm:desing}.\qed

\begin{remark}
  \label{rmk:Nalpha-Weyman}
  It follows from Proposition~\ref{prop:image-sections} that
  $N_\alpha=M(\alpha,0)$ in the notation of~\cite[\S
  6]{Weyman:book}. In particular the very general result~\cite[Cor
  (6.5.17)]{Weyman:book} gives an alternative way to see that
  $N_\alpha$ is Cohen-Macaulay in characteristic
  zero. Furthermore~\cite[Example (6.5.18)]{Weyman:book} shows that
  $N_{2}$ is not Cohen-Macaulay in characteristic~$2$.
\end{remark}

\begin{example}\label{eg:BLV10}
  Assume that $m-l=1$ with $m\le n$. Then we have $\GG=\PP^{m-1}$. Set
  $\PP = \PP^{m-1}$, so that $\calq=\Omega^\svee_{\PP}(-1)$, and let
  $\alpha = 1^a$ for some $a$, $0 \leq a \leq m-1$.  We find
  \begin{align*}
    \calt_{\alpha} &= {p'}^{\ast} \left(\Wedge^a\,
      \Omega^\svee_{\PP}(-a)\right)\\
    &={p'}^{\ast} \left(\Wedge^{m-1-a}\,
      \Omega_{\PP}\otimes_{\calo_{\PP}} \omega^{-1}_\PP(-a)\right)\\
    &={p'}^{\ast}\left(\Wedge^{m-1-a}\, \Omega_{\PP}(m-a)\right)\,.
  \end{align*}
  Thus in the notation of~\cite{Buchweitz-Leuschke-VandenBergh:2010}
  we have $T_\alpha=M_{m-a} = \cok \Wedge^{m-a} X$.
\end{example}

\section{Proof of Theorem~\ref{thm:BO}}
We now need to refer to the two resolutions of $\Spec R$ in a uniform
way, so we introduce appropriate symmetrical notation for this section
only.  We start by putting $G_1=F^\svee$ and $G_2=G$ so that
\[
H=\Sym_K(G_1\otimes G_2)\,.
\]
We also put $n_i=\rank_K G_i$ and $\GG_i=\Grass(n_i-l,G_i)$.  Thus $n_1=m$,
$n_2=n$, and we have canonically $\GG_1\cong\GG$.

For symmetry we also put $\calz_1=\calz$. In general we will decorate 
the notations in the diagram~\eqref{fig:Springer} by a ``$1$'' or a
``$2$'' depending on whether they refer to $\calz_1$ or $\calz_2$. 

We now explain how we prove Theorem~\ref{thm:BO}.  In
Proposition~\ref{prop:tiltZ} we have constructed tilting bundles
$\calt_{1}$, $\calt_2$ on $\calz_1$, $\calz_2$. For our purposes it
turns out to be technically more convenient to use the tilting bundle
$\calt^\svee_1$ on $\calz_1$ rather than $\calt_1$.  With $E_{1}'$,
$E_2$ the endomorphism rings of $\calt_1^\svee$ and $\calt_2$
respectively, it turns out that if $n_1\le n_2$ then $E_1'\cong e
E_2e$ for a suitable idempotent $e\in E_2$.  Thus we immediately
obtain a fully faithful embedding $D^b(\coh \calz_1)\hookrightarrow
D^b(\coh \calz_2)$.  We then show that this embedding coincides with
the indicated Fourier-Mukai transform.

Now we proceed with the actual proof. On $\GG_i$ we have the
tautological exact sequence
\[
0\to \calr_i\to \pi_i^\ast G_i\to \calq_i\to 0\,.
\]
We also define
\[
\hat\calz=\calz_1\times_H \calz_2\,.
\]
There are projection maps $r_1\colon\hat \calz\to\calz_1$,
$r_2\colon\hat\calz\to\calz_2$. These fit together in the following
commutative diagram.
\[
\begin{tikzcd}
  {} &{} & \hat\calz\arrow{dl}{r_1} \arrow{dr}[swap]{r_2} \arrow[bend right]{ddll}[swap]{p''_1}
  \arrow[bend left]{ddrr}{p''_2}&&\\ 
  {} &\calz_1\arrow{dr}[swap]{q'_1} \arrow{dl}{p'_1} &&
  \calz_2 \arrow{dl}{q'_2} \arrow{dr}[swap]{p'_2}&\\ 
  \GG_1&& \Spec R &&\GG_2
\end{tikzcd}
\]
Let $H_0\subset \Spec R$ be the (open) locus of tensors of rank
exactly $l$, so that the maps $q_i'$ and $r_i$, for $i=1,2$, are all
isomorphisms above $H_0$.  Let $\hat\calz_0$ be the inverse image of
$H_0$ in $\hat \calz$.

Let $\alpha$ be a partition and set $\calt_{\alpha,i} =
{p_i'}^*\left(\Wedge^{\alpha'}\calq_i\right)$ for $i=1,2$.  Further
set $B_i = B_{l,n_i-l}$, 
\[
\calt_i=\bigoplus_{\alpha\in B_i} \calt_{\alpha,i}
\qquad\text{and}\qquad 
E_i =\End_{\calo_{\calz_i}}(\calt_i)\,.
\]
By Theorem~\ref{thm:tiltZ}, $\calt_i$ is a tilting bundle on $\calz_i$
and hence $\cald^b(\coh \calz_i )\cong \cald^b(\mod E_i)$.  

Here is an asymmetrical piece of notation.  Assume that  $n_1 \leq 
n_2$. Then $B_1\subseteq B_2$. Set
\begin{equation}
  \label{eq:T2prime}
  \calt_2' = \bigoplus_{\alpha \in B_1} \calt_{\alpha,2} \subseteq\bigoplus_{\alpha \in B_2} \calt_{\alpha,2}=\calt_2
  \qquad\text{and}\qquad
  E_2' = \End_{\calo_{\calz_2}}(\calt_2')\,.
\end{equation}
As $\calt_2'$ is a direct summand of $\calt_2$, we have $E_2' = e E_2e$
for a suitable idempotent $e \in E_2$.  Hence there is a fully
faithful embedding
\begin{equation}
\label{ff}
\tilde e \colon \cald^b(\mod E_2') \into \cald^b(\mod E_2)
\end{equation}
given by $\tilde e(\calm) = (E_2)e   \otimes_{E_2'} \calm$. 

Put $E_1' = \End_{\calo_{\calz_1}}(\calt_1^\svee)$.  Note that
it follows easily from Grothendieck duality that $\calt_1^\svee$ is
also a tilting bundle on $\calz_1$.

Finally set
\[
T_{\alpha,i} = {q_i'}_* \calt_{\alpha,i}\,,
\qquad
T_i = {q_i'}_* \calt_i\,,
\]
and $T_2' = {q_2'}_* \calt_2'$.  By Theorem~\ref{thm:tiltZ}, we have
$E_i = \End_R(T_i)$, $E_1' = \End_R(T_1^\svee)$, and $E_2'
= \End_R(T_2')$. 

\begin{lemma} 
  One has
  \(
  \hat\calz=\underline{\Spec}
  \left(\Sym_{{\GG_1\times\GG_2}}(\calq_1\boxtimes \calq_2)\right)\,.
  \)
\end{lemma}

\begin{proof} This is a straightforward computation.  
  \begin{align*}
    \calz_1\times_H\calz_2&=
    \calz_1\times_{\GG_1\times H} (\GG_1\times H)\times_H(\GG_2\times H)\times_{\GG_2\times H}\calz_2\\
    &=\calz_1\times_{\GG_1\times H} (\GG_1\times \GG_2\times H)\times_{\GG_2\times H}\calz_2\\
    &=(\calz_1\times \GG_2)\times_{\hat\GG\times H}(\calz_2\times \GG_1)\\
    &=\underline{\Spec}\left(
      \Sym_{{\GG_1\times\GG_2}}(\calq_1\boxtimes \pi^\ast_2 G_2)
      \otimes_{\Sym_{{\GG_1\times\GG_2}}(\pi_1^\ast G_1\boxtimes \pi_2^\ast G_2)}
      \Sym_{{\GG_1\times\GG_2}}(\pi^\ast_1 G_1\boxtimes \calq_2)\right)\\
    &=\underline{\Spec}\left(\Sym_{{\GG_1\times\GG_2}}(\calq_1\boxtimes \calq_2)\right)\qedhere
  \end{align*}
\end{proof} 

\begin{proposition}
  \label{prop:TTprimeEEprime}
  Assume $n_1 \leq n_2$.  Then $T_2' \cong T_1^\svee$. In 
  particular $E_2' \cong E_1'$, and there  is a fully faithful
  embedding $\cald^b(\mod E_1') \into \cald^b(\mod E_2)$ (using
  \eqref{ff}).  If $n_1=n_2$ then the embedding is an equivalence. 
\end{proposition}

\begin{proof}
  Since $\hat \calz = \underline{\Spec}
  \left(\Sym_{{\GG_1\times\GG_2}}(\calq_1\boxtimes \calq_2)\right)$,
  we have a canonical map
  \[
  u \colon (p_2'')^* \calq_2 \to (p_1'')^* \calq_1^\svee
  \]
  which is an isomorphism on $\hat\calz_0$.  Apply
  $\Wedge^{\alpha'}(-)$ for a partition $\alpha$ to obtain a map 
  \begin{equation}\label{eq:wedge-u}
  \Wedge^{\alpha'} u \colon r_2^* \calt_{\alpha,2} \to r_1^*
  (\calt_{\alpha,1})^\svee
  \end{equation}
  and push down with $(q_1' r_1)_* = (q_2' r_2)_*$ to get a
  homomorphism of $R$-modules
  \begin{equation}\label{eq:taualpha}
  \tau_\alpha\colon T_{\alpha,2} \to T_{\alpha,1}^\svee
  \end{equation}
  which is an isomorphism on $H_0$.  Letting $\alpha$ run over
  partitions in $B_1$, we find a homomorphism
  \(
  \tau\colon T_2' \to T_1^\svee
  \)
  which is also an isomorphism on $H_0$.  Since the exceptional loci
  for the $q_i'$ in $\calz_i$ have codimension at least $2$, the
  modules $T_1$ and $T_2'$ are reflexive by~\cite[Lemma
  4.2.1]{VandenBergh:flops}.  (In fact we know already that $T_1$ is
  Cohen-Macaulay.)   Hence $\tau \colon T_2' \to T_1^\svee$ is an
  isomorphism.

  In particular $\tau$ induces an isomorphism $\tilde \tau \colon E_1'
  \to E_2'$.
\end{proof}

The birational map $\calz_2 \to \calz_1$ is easily seen to be a
\emph{flip}, and, if $n_1=n_2$, even a \emph{flop}. Our final result
thus verifies, in this special case, a general conjecture of Bondal
and Orlov~\cite{Bondal-Orlov:2002}.

\begin{theorem}
  \label{prop:ffembE1E2}
  Assume $n_1 \leq n_2$.  Then there is a fully faithful embedding
  \[
  \calf \colon \cald^b(\coh \calz_1) \to \cald^b(\coh \calz_2)
  \]
  given by 
  \[
  \calf (\calm) = 
  \calt_2'  \lotimes_{E_1'} \rHom_{\calo_{\calz_1}}(\calt_1^\svee, \calm)
  \]
  where $E_1'= \End_R(\calt_1^\svee)$ acts on $\calt_2'$ via the
  isomorphism $E_1' \cong \End_{\calo_{\calz_2}}(\calt_2')$ of
  Proposition~\ref{prop:TTprimeEEprime}.  If $n_1 = n_2$ then $\calf$
  is an equivalence.
\end{theorem}

\begin{proof}
  Since $\calt_1^\svee$ and $\calt_2$ are tilting on $\calz_1$ and
  $\calz_2$, respectively, we have equivalences 
  \[
  \rHom_{\calo_{\calz_1}}(\calt_1^\svee, -) \colon \cald^b(\coh
  \calz_1) \to \cald^b(\mod E_1')
  \]
  and
  \[
  \calt_2 \lotimes_{E_2}  - \colon \cald^b(\mod E_2) \to \cald^b(\coh \calz_2)\,.
  \]
  Putting these together with the isomorphism $E_1' \cong E_2'$ and
  the fully faithful embedding $\tilde e \colon \cald^b(\mod E_2') \to
  \cald^b(\mod E_2)$, we find the composition
  \[
  \calf\colon 
  \cald^b(\coh \calz_1) 
   \xto{\ \cong \ } 
  \cald^b(\mod E_1') 
   \xto{\ \cong \ }
  \cald^b(\mod E_2')
  \ {\into}\  
  \cald^b(\mod E_2)
   \xto{\ \cong \ }
  \cald^b(\coh \calz_2)\,,
  \]
  of the form asserted.
\end{proof}

\begin{theorem} 
  \label{thm:FM}
  Assume that $n_1\le n_2$. Then the Fourier-Mukai transform $\FM = \R
  {r_2}_* \bL r_1^*$ with kernel $(r_1,r_2)_\ast
  \calo_{\hat\calz}$ defines a fully faithful embedding
  \[
  \FM\colon  \cald^b(\coh \calz_1 ) \into \cald^b(\coh \calz_2) 
  \]
  which is an equivalence if $n_1=n_2$.
  There is a natural isomorphism between $\FM$ and the functor $\calf
  = \calt_2' \lotimes_{E_1'} \rHom_{\calo_{\calz_1}}(\calt_1^\svee,
  -)$ introduced in Proposition~\ref{prop:ffembE1E2}.  In particular
  $\FM$ is fully faithful.
\end{theorem}

\begin{proof}
  For a partition $\alpha \in B_1$, the map
  $\Wedge^{\alpha'}u \colon r_2^* \calt_{\alpha,2} \to r_1^*
  (\calt_{\alpha,1})^\svee$ constructed in~\eqref{eq:wedge-u} gives by
  adjointness a homomorphism on $\calz_2$
  \[
  \sigma\colon \calt_{\alpha,2} \to \R {r_2}_* r_1^*
  (\calt_{\alpha,1})^\svee\,.
  \]
  We claim that $\sigma$ is an isomorphism.  In particular we must
  show $\R^i {r_2}_* r_1^*(\calt_{\alpha,1})^\svee=0$ for $i > 0$.  To
  this latter end it is sufficient to show that for all $y \in \GG_2$
  and all $i>0$ we have
  \[
  H^i(\GG_1, \Wedge^{\alpha'}\calq_1^\svee \otimes_{\calo_{\GG_1}}
  \Sym_{{\GG_1}}(\calq_1 \otimes(\calq_2)_y)) = 0\,.
  \]
  This follows again from the Cauchy formula together with
  Proposition~\ref{prop:work}.  

  Now we can see that $\sigma \colon \calt_{\alpha,2} \to {r_2}_*
  r_1^* (\calt_{\alpha,1})^\svee$ is an isomorphism.  The source
  is reflexive, the target is torsion-free, and over $\hat \calz_0$
  the map $\sigma$ coincides with $(q_2')^*\tau_\alpha$, where
  $\tau_\alpha \colon T_{\alpha,2} \to T_{\alpha,1}^\svee$ as
  in~\eqref{eq:taualpha}.  Since each $\tau_\alpha$ is an isomorphism,
  so is $\sigma$.

  In particular we obtain an isomorphism $\tilde \sigma
  \colon \calt_2' \to \R {r_2}_* \bL r_1^* \calt_1^\svee$ by summing
  over $\alpha \in B_1$.

  To define the desired natural transformation $\eta \colon \calf \to
  \FM$, we must construct a morphism
  \[
  \eta(\calm) \colon 
  \calt_2' \lotimes_{E_1'}    \rHom_{\calo_{\calz_1}}(\calt_1^\svee,\calm)
  \to 
  \R {r_2}_* r_1^* \calm
  \]
  for every $\calm$ in $\cald^b(\coh \calz_1)$.  The desired map is
  the composition of
  \[
  \begin{tikzcd}[row sep = normal]
    \calt'_2 \lotimes_{E'_1}
    \rHom_{\calo_{\calz_1}}(\calt^\svee_1,\calm) 
    \arrow{d}{\tilde\sigma \otimes_{E'_1}  \R r_{2\ast} \bL r^\ast_1}\\
    \R r_{2\ast}\bL r_1^\ast\calt_1^\svee  \lotimes_{E'_1}
    \rHom_{\calo_{\calz_2}}(\R r_{2\ast} \bL r_1^\ast\calt^\svee_1, \R
    r_{2\ast}\bL r_1^\ast\calm) 
  \end{tikzcd}
  \]
  and the evaluation map from the derived tensor product to $\R
  r_{2\ast}\bL r_1^\ast \calm$.  To show that $\eta$ is an isomorphism,
  it suffices, since $\calt_1^\svee$ generates, to prove that
  $\eta(\calt_1^\svee)$ is an isomorphism.  In this case, we have
  \[
 \calt_2'  \lotimes_{E_1'} 
  \rHom_{\calo_{\calz_1}}(\calt_1^\svee, \calt_1^\svee)
\cong
\calt_2' \lotimes_{E_1'}   E_1' 
\cong
  \calt_2' \cong \R {r_2}_* r_1^* \calt_1^\svee\,,
  \]
  an isomorphism by construction.
\end{proof}

\begin{remark}
  \label{rmk:E1E1prime}
  Though we did not use it, in fact we have $E_1' \cong E_1$.  Indeed,
  for $\alpha = (\alpha_1, \dots, \alpha_l) \in B_i$, define
  \[
  \alpha^! = (n_i-l-\alpha_l, \dots, n_i-l-\alpha_1)\,.
  \]
  Then 
  \[
  \Wedge^{\alpha'} \calq_i^\svee \cong \left(\Wedge^l
    \calq_i\right)^{-(n_i-l)} \otimes_{\calo_{\GG_i}}
  \Wedge^{(\alpha^!)'} \calq_i\,.
  \]
  Thus 
  \[
  (\calt_{\alpha,i})^\svee \cong {p_i'}^*
  \left(\Wedge^l\calq_i\right)^{-(n_i-l)} \otimes_{\calo_{\calz_i}}
  \calt_{\alpha^!,i}
  \]
  and hence
  \[
  \calt_i^\svee \cong {p_i'}^*\left(\Wedge^l \calq\right)^{-(n_i-l)}
  \otimes_{\calo_{\calz_i}} \calt_i\,.
  \]
  It follows that $\End_{\calo_{\calz_i}}(\calt_i^\svee)
  \cong \End_{\calo_{\calz_i}}(\calt_i)$. 
\end{remark}

\section{Presentations of the simples}\label{sect:simples}

Throughout this section we assume that the characteristic of our
ground field is zero.  We give an algorithm, based on Bott's theorem
and the Littlewood-Richardson rule, for determining the $\Ext$-groups
between the simple modules over the non-commutative desingularization.
We work out explicitly the representations appearing in the
$\Ext$-groups of low degree, for later use in the proof of
Theorem~\ref{thm:quiverintro}. The method is a direct generalization
of that used in~\cite{Buchweitz-Leuschke-VandenBergh:2010} for the
case of maximal minors, and was independently established in a more
general form by Weyman and Zhao~\cite{Weyman-Zhao}. It was known to
the authors how to extend our methods to arbitrary minors, but after
seeing~\cite{Weyman-Zhao} we realized we could simplify the part of
the argument involving Bott's theorem.  In particular
Lemma~\ref{lem:WZ} is contained in~\cite[Corollary
3.6]{Weyman-Zhao}. We provide a proof for the convenience of the
reader.

Since we work in characteristic zero, we consider the tilting bundle
$\caln = \bigoplus_\alpha \caln_\alpha =\bigoplus_\alpha {p'}^*
L^\alpha\calq$ (cf.\ Proposition~\ref{prop:image-sections}) on the
desingularization $\calz$ and its endomorphism ring $A
= \End_\caloz(\caln)$. Then $A$ is Morita equivalent to $E
= \End_R(T)$ of Theorem~\ref{thm:desing}.

For $\alpha \in B_{l,m-l}$ let $P_\alpha =
\Hom_\caloz(\caln,\;\caln_\alpha)$ be the projective left $A$-module
corresponding to $\alpha$, and let $S_\alpha$ be its associated simple
module.  As in~\cite{Buchweitz-Leuschke-VandenBergh:2010}, we have the
following identification of $S_\alpha$.

\begin{lemma}
  Let $u \colon \GG \to \calz$ be the zero section of the vector
  bundle $p'\colon \calz \to \GG$.  Then the object in $\cald^b(\coh
  \calz)$ corresponding to the simple module $S_\alpha$ is $u_*
  L^{\alpha'} \calr[|\alpha|]$.
\end{lemma}

\begin{proof}
  By~\cite{Kapranov:1988}, the bundles
  $\left\{L^{\alpha'}\calr[|\alpha|]\right\}_{\alpha \in B_{l,m-l}}$ form a dual
  exceptional collection to the full strong exceptional collection
  $\left\{L^\alpha \calq\right\}_{\alpha\in B_{l,m-l}}$, that is,
  \[
  \Ext_{\calog}^t(L^\alpha \calq,\ L^{\beta'} \calr [|\beta|]) = 
  \begin{cases}
    K & \text{if $t=0$ and $\alpha = \beta$, and}\\
    0 & \text{otherwise.}
  \end{cases}
  \]
  This $\Ext$ group is by adjunction isomorphic to
  $\Ext_\caloz^t({p'}^*L^\alpha \calq,\ u_* L^{\beta'} \calr
  [|\beta|])$.  Since ${p'}^*L^\alpha \calq$ corresponds to the
  projective $P_\alpha$ over $A$, this gives the desired statement.
\end{proof}

To compute the extensions between the simple objects, we use the
following proposition~\cite[Proposition
10.6]{Buchweitz-Leuschke-VandenBergh:2010}.  The proof given in loc.\
cit.\ is over $\PP$, but is equally valid over $\GG$.

\begin{lemma}\label{lem:BLV-10.6}
  Let $\calu, \calv$ be objects in $\cald^b(\coh \GG)$.  Then
  \[
  \Ext_{\caloz}^t (u_*\calu,\ u_* \calv) = 
  \bigoplus_s \Ext_\calog^{t-s}(\Wedge^s (\calq \otimes G) \otimes_\calog
  \calu,\ \calv)\,.\qed
  \]
\end{lemma}

\begin{theorem}\label{thm:simple-exts}
  Let $\alpha, \beta \in B_{l,m-l}$. For the
  simple left $A$-modules $S_\alpha$ and $S_\beta$ we have
  \[
  \Ext_A^t(S_\beta, S_\alpha) = 
  \bigoplus_{\lambda} H^{t-|\lambda|+|\alpha|-|\beta|}(\GG,\ L^\lambda
  \calq^\svee \otimes L^{\alpha'}\calr \otimes L^{\beta'}\calr^\svee  )
  \otimes L^{\lambda'}G^\svee\,,
  \]
  where the sum is over all $\lambda \in B_{l,n}$.
\end{theorem}

Observe that the $\lambda$ appearing in the formula have the same
bound on the number of rows as $\alpha$ and $\beta$, but the
constraint on their widths depends on $G$.

\begin{proof}
  This is a direct calculation using Lemma~\ref{lem:BLV-10.6} and the
  Cauchy decomposition from Proposition~\ref{thm:cauchyLR}:
  \begin{align*}
    \Ext_A^t(S_\beta, S_\alpha) &=
    \Ext_\caloz^t(u_*L^{\beta'}\calr[|\beta|],\ u_*L^{\alpha'}\calr[|\alpha|])\\
    &=
    \Ext_\caloz^{t+|\alpha|-|\beta|}(u_*L^{\beta'}\calr,\ u_*L^{\alpha'}\calr)\\
    &= \bigoplus_s
    \Ext_\calog^{t-s+|\alpha|-|\beta|}(\Wedge^s(\calq\otimes
    G)\otimes_\calog L^{\beta'}\calr,\ L^{\alpha'}\calr)\\
    &= \bigoplus_s H^{t-s+|\alpha|-|\beta|}(\GG,\
    \Wedge^s(\calq\otimes
    G)^\svee \otimes_\calog (L^{\beta'}\calr)^\svee\otimes L^{\alpha'}\calr)\\
    &= \bigoplus_s \bigoplus_{|\lambda|=s}
    H^{t-s+|\alpha|-|\beta|}(\GG,\ L^\lambda\calq ^\svee
    \otimes_\calog L^{\alpha'}\calr \otimes
    L^{\beta'}\calr^\svee)\otimes L^{\lambda'}G^\svee
  \end{align*}
  which is equal to the desired sum since $\rank \calq =l$ and $\rank
  G = n$.
\end{proof}

For any given $t$, computing the cohomology indicated in the theorem
is algorithmic, though a complete combinatorial description of exactly
which representations appear remains open.  We can evaluate the sum
for small values of $t$ using the Littlewood-Richardson rule and
Bott's theorem~\cite{Weyman:book}.  Recall the algorithm of Bott: a
bundle of the form $L^\lambda \calq^\svee \otimes L^\gamma
\calr^\svee$, for dominant weights $\lambda$ and $\gamma$, has at most
one non-vanishing cohomology group, and the index $k$ for which
$H^k(\GG,\ L^\lambda \calq^\svee \otimes L^\gamma \calr^\svee) \neq 0$
is computed by flattening the weight $(\gamma,\delta)$ using the
twisted action of the symmetric group $S_m$.\footnote{Technically we
  must flatten $(\lambda^*,\gamma^*)$, where $\lambda^* = -w_0\lambda$
  and $w_0$ is the long word in $S_m$.  However it is easy to see that
  the result is the same, since passing to the dual Grassmannian replaces
  $(\lambda^*, \gamma^*)$ with $(\gamma, \lambda)$.} This means that the
adjacent transpositions $\sigma_i = (i,i+1)$ act on a weight $\alpha =
(\alpha_1, \dots, \alpha_m)$ by $\sigma_i \cdot \alpha = (\alpha_1,
\dots, \alpha_{i+1}+1, \alpha_i -1, \dots, \alpha_m)$. If there exists
a permutation $\tau$ such that $\tau\cdot (\gamma, \lambda)$ is
dominant (that is, weakly decreasing), then the only non-vanishing
cohomology is
\[
H^{l(\tau)}(\GG,\ L^\lambda \calq^\svee \otimes L^\gamma \calr^\svee) =
L^{\tau\cdot (\gamma,\lambda)} F\,,
\]
where $l(\tau)$ is the length of $\tau$'s expansion in adjacent
transpositions. If there exists no such $\tau$, or equivalently
$\tau\cdot (\gamma, \lambda)= (\gamma, \lambda)$ for some non-trivial
$\tau \in S_m$, then all cohomology of $L^\lambda \calq^\svee \otimes
L^\gamma \calr^\svee$ vanishes.

We can describe the algorithm equivalently by defining the action of
$S_m$ via $\sigma_i \cdot \alpha = \sigma_i (\alpha+\rho)-\rho$, where
$\rho = (m-1, m-2, \dots, 1, 0)$.  If $\alpha+ \rho$ contains a
repetition, there is no cohomology.

Note that in this procedure $\gamma$ and $\lambda$ are not assumed to
have non-negative entries.  We write $\alpha = \alpha_+ + \alpha_-$ for
the decomposition of a weight $\alpha$ into positive and negative
parts, and $|\alpha| = |\alpha_+| + |\alpha_-|$ for the signed area of
$\alpha$.

We need a combinatorial lemma.\footnote{A similar argument
  in~\cite{Weyman-Zhao} allowed us to simplify our original argument
  significantly.}  The $L^\gamma \calr^\svee$ appearing in the
Littlewood-Richardson decomposition of $L^{\alpha'}\calr \otimes
L^{\beta'}\calr^\svee$, for $\alpha,\; \beta \in B_{l,m-l}$, will have
$\gamma_i \geq -l$ for all $i$.
\begin{lemma}\label{lem:WZ}
  Let $\gamma = (\gamma_1, \dots, \gamma_{m-l})$ and $\lambda =
  (\lambda_1, \dots, \lambda_l)$ be dominant weights.  Assume that
  $\gamma_i \geq -l$ for all $i$. If $H^k(\GG,\ L^\lambda \calq^\svee
  \otimes L^\gamma \calr^\svee) \neq 0$ for some $k$, then $-\gamma_-
  \subseteq \lambda'$ and $k \geq - |\gamma_-|$.  In particular, if
  $H^{t-|\lambda|+|\gamma|}(\GG,\ L^\lambda \calq^\svee \otimes
  L^\gamma \calr^\svee) \neq 0$ for some $t$, then $t-|\lambda| \geq
  |\gamma_+|$. 
\end{lemma}

\begin{proof}
  We have to show that the negative part of $\gamma$ is contained in
  the first columns of $\lambda$.  If $\gamma$ has no negative entries
  we are of course done.  Set $s= - \gamma_{m-l} \leq l$ and assume
  $s>0$.  Then $\lambda$ can have at most $l-s$ zero entries, for
  otherwise $(\gamma,\lambda)+\rho = (\gamma_1 + m-1, \dots,
  \gamma_{m-l-1} + l+1, l-s, \lambda_1 + l-1, \dots, \lambda_l)$ would
  have a repetition of $l-s$ and all cohomology would vanish.  The
  result of partially flattening $\gamma_{m-l}$ is therefore
  $(\gamma_1, \dots, \gamma_{m-l-1}, \lambda_1 - 1, \dots,
  \lambda_{s} - 1, 0, \lambda_{s-1}, \dots, \lambda_l)$ and
  $\lambda_{s} - 1\geq 0$.  Since $\gamma_{m-l-1} \geq -s$ we may
  repeat the argument with the weight $(\gamma_1, \dots,
  \gamma_{m-l-1}, \lambda_1 - 1, \dots, \lambda_{s} - 1)$ to see that
  $(\lambda_1 - 1, \dots, \lambda_{s} - 1)$ can have at most
  $s-\gamma_{m-l-1}$ zero entries.  Iterate.  The last sentence is
  clear from $|\gamma| = |\gamma_+| + |\gamma_-|$.
\end{proof}

Recall that we use the notation $\alpha \nearrow \beta$ to indicate that $\beta$
is obtained from $\alpha$ by adding a single box.

\begin{nsit}{Computation of $\Ext^t$ for  $t=0,\ 1,\ 2$.}\label{eg:smallt}
  We apply Bott's algorithm first with $t=0$ to compute
  $\Hom_A(S_\beta,\;S_\alpha)$ as a sanity check.
  Theorem~\ref{thm:simple-exts} asks us to compute
  \[
  \bigoplus_{\lambda \in B_{l,n}} H^{-|\lambda|+|\gamma|}(\GG,\
  L^{\lambda} \calq^\svee \otimes L^\gamma \calr^\svee) \otimes
  L^{\lambda'}G^\svee
  \]
  for all $\gamma$ such that $L^\gamma \calr^\svee$ appears in
  $L^{\alpha'}\calr \otimes L^{\beta'}\calr^\svee$.  By the lemma, if
  this cohomology is non-zero then we must have $-|\lambda| \geq
  |\gamma_+|$, which since $\lambda$ is non-negative forces $\lambda=
  (0, \dots,0)$ and $\gamma_+ = (0,\dots, 0)$. The lemma furthermore
  implies $- \gamma_- \subseteq \lambda'$, so $\gamma$ is also the zero
  partition.  This occurs only when $\alpha=\beta$, and we
  obtain
  \[
  \Hom_A(S_\beta,\ S_\alpha) = 
  \begin{cases}
    K & \text{if $\alpha=\beta$, and}\\
    0 & \text{otherwise,}
  \end{cases}
  \]
  as expected.

  For $t=1$, $1-|\lambda| \geq |\gamma_+|$ implies either $\lambda =
  (0,\dots, 0)$ or $\lambda = (1, \dots, 0)$.  In the first case, we
  find $\gamma_-=0$ and $\gamma_+$ can be either $(0,\dots, 0)$ or
  $(1,\dots, 0)$.  The first choice for $\gamma$ leads to $H^1(\GG,\
  \calog) =0$, and the second to $H^0(\GG,\ \calr^\svee) = F$.  In the
  second case we have $\gamma_+=0$ and $\gamma_- = (0,\dots, 0)$ or
  $(0, \dots, -1)$ since $-\gamma_- \subseteq \lambda'$.  Here the
  first choice gives no cohomology and the second contributes
  $H^1(\GG,\ \calq^\svee \otimes \calr) = K$.

  A direct summand of the form $L^{(1,0,\dots,0)} \calr^\svee$
  appearing in $L^{\alpha'}\calr \otimes L^{\beta'}\calr^\svee$
  implies that $\alpha' \subseteq \beta'$ and $\beta'$ differs from
  $\alpha'$ in exactly one entry, where $\beta'_i = \alpha'_i+1$, so
  $\alpha \nearrow \beta$.  Similarly the appearance of
  $L^{(0,\dots,0,-1)} \calr^\svee$ indicates that $\beta$ is the
  result of removing a box from $\alpha$.  Thus
  \[
  \Ext_A^1(S_\beta,\;S_\alpha) = 
  \begin{cases}
    F & \text{if $\alpha \nearrow \beta$, }\\
    G^\svee & \text{if $\beta \nearrow \alpha$, and }\\
    0 & \text{otherwise.}
  \end{cases}
  \]

  The case $t=2$ requires considering several cases corresponding to
  $|\lambda| = 0,\ 1,\ 2$.  If $\lambda$ is the zero partition, then
  $\gamma_- =0$ and $\gamma_+$ is one of $(0, \dots, 0)$, $(1,\dots,
  0)$, $(2,\dots, 0)$ or $(1,1, 0,\dots,0)$.  These are all already
  dominant, so contribute only $H^0$, so we obtain $H^0(\GG,\ \Sym_2
  \calr^\svee) = \Sym_2 F$ and $H^0(\GG,\ \Wedge^2\calr^\svee) =
  \Wedge^2 F$.  These $\gamma$ correspond to obtaining $\beta'$ by
  adding to $\alpha'$, respectively, two boxes not in the same column
  and two boxes not in the same row.

  In case $\lambda = (1,0,\dots,0)$ then $|\gamma_+|\leq 1$ and
  $-|\gamma_-|\leq 1$. The case $\gamma = (0,\dots,0)$ gives no
  cohomology.  If $\gamma = (0, \dots, 0,-1)$ then one swap gives the
  zero partition so we find a contribution to $H^1$ but none to $H^2$.
  If $\gamma = (1,0,\dots, 0,-1)$ then we obtain $H^1(\GG,\
  \calq^\svee\otimes L^{(1,0,\dots,0,-1)}\calr^\svee) = F$.  These
  $\gamma$ correspond to the $\beta'$ obtained by adding one box to
  $\alpha'$ and removing one box.  Finally, if $\lambda =
  (1,0,\dots,0)$ and $\gamma=(1,0,\dots,0)$ then we again have no
  cohomology, unless $\gamma$ has just one entry, in which case
  $m-l=1$ and we get $H^0(\GG,\ \calq^\svee\otimes \calr^\svee) =
  L^{(1,1,0,\dots,0)}F = \Wedge^2 F$. This arises from $\alpha'
  \nearrow \beta'$.

  Assume $\lambda = (1,1,0, \dots, 0)$. Then $\gamma_+=0$ and the
  possibilities for $\gamma_-$ are $(0, \dots, 0)$, $(0, \dots,
  0,-1)$, and $(0, \dots, 0,-2)$.  The first and second cases lead to
  no cohomology, while the third possibility takes two swaps to give
  the zero partition, so $H^2(\GG,\ \Wedge^2\calq \otimes L^{(0,
    \dots, 0,-2)}\calr^\svee ) = K$. This $\gamma$ appears when
  $\alpha'$ is obtained by adding two boxes to $\beta'$, not in the
  same row.

  Lastly suppose $\lambda = (2,0,\dots, 0)$. Then again $\gamma_+=0$
  and now the possibilities for $\gamma_-$ are $(0, \dots, 0)$, $(0,
  \dots, 0,-1)$, and $(0, \dots, 0, -1,-1)$.  The first case gives no
  cohomology unless $m-l=1=l$, in which case $(0,2)$ flattens in one
  swap to $(1,1)$ and we get a contribution to $H^1$ but none to
  $H^0$.  The second case
  flattens in one step to $(0, \dots,0, 1,0, \dots,0)$, which gives no
  cohomology if $m-l > 1$ and $H^1(\GG,\ \Sym_2\calq \otimes L^{(0,
    \dots, 0,-1)}\calr^\svee) = F$ if $m-l=1$.  This occurs when
  $\alpha' \nearrow \beta'$. The third case flattens to the zero
  partition in two swaps, so gives $H^2(\GG,\ \Sym_2\calq\otimes
  L^{(0, \dots, 0, -1,-1)}\calr^\svee)=K$. This occurs when $\alpha'$
  is obtained by adding two boxes to $\beta'$, not in the same
  column.

  Analyzing the ways in which the $L^\gamma \calr^\svee$ appearing
  above can appear in $L^{\alpha'}\calr \otimes
  L^{\beta'}\calr^\svee$, we arrive at the final results.  If $m-l>1$, then
  $\Ext_A^2(S_\beta,\; S_\alpha)$ is given by
  \begin{equation}\label{eq:Arels}
  \begin{cases}
    \Sym_2 F & \text{if } \alpha \nearrow \nearrow \beta, \text{ two boxes in a column}\\
    \Wedge^2 F & \text{if } \alpha \nearrow \nearrow \beta, \text{ two
      boxes in a
      row} \\
    \Sym_2 F \oplus \Wedge^2 F \cong F \otimes F & \text{if } \alpha
    \nearrow \nearrow \beta, \text{ two
      disconnected boxes} \\
    F\otimes G^\svee & \begin{minipage}{2in} if $\alpha \neq \beta$
      and $\alpha \nearrow
      \delta$,\ $\beta \nearrow \delta$ \\ \phantom{asdf}for some $\delta \in B_{l,m-l}$ \end{minipage}\\
    (F\otimes G^\svee)^{\oplus(t(\alpha)-1)} & \text{if }\alpha =\beta\\
    \Sym_2 G^\svee & \text{if }\beta \nearrow \nearrow \alpha, \text{
      two
      boxes in a column} \\
    \Wedge^2 G^\svee & \text{if }\beta \nearrow \nearrow \alpha,
    \text{ two
      boxes in a row}\\
    \Sym_2 G^\svee \oplus \Wedge^2 G^\svee \cong G^\svee \otimes
    G^\svee & \text{if }\beta \nearrow \nearrow \alpha, \text{ two
      disconnected boxes.}
  \end{cases}
  \end{equation}
  Here $t(\alpha)$ is the number of ways to add a box to $\alpha$
  without passing out of the sides of the box $B_{l,m-l}$.  This is
  the case corresponding to $\gamma = (1,0, \dots, 0, -1)$ and $\alpha=\beta$.

  In the case of maximal minors, where $m-l=1$, some of the cases above
  do not occur and also we have some additional contributions to
  $\Ext^2$.  In that case we find
  \[
  \Ext_A^2(S_\beta,\; S_\alpha) =   
  \begin{cases}
    \Sym_2 F & \text{if }\alpha \nearrow \nearrow \beta, \text{ two boxes in a column}\\
    \Sym_2 G^\svee & \text{if }\beta \nearrow \nearrow \alpha, \text{ two
      boxes in a column}\\
    \Wedge^2F \otimes G^\svee & \text{ if } \alpha \nearrow \beta\\
    F \otimes \Wedge^2 G^\svee & \text{ if } \beta\nearrow \alpha\,.
  \end{cases}
  \]
\end{nsit}

\begin{remark}\label{rem:cubicrels}
  The computation of $\Ext^2(S_\beta,\; S_\alpha)$ when $m-l=1$ appears
  already in~\cite[Example 10.3]{Buchweitz-Leuschke-VandenBergh:2010},
  and the cubic relations between adjacent vertices in the last two
  lines above are reflected in the commutativity relations on the quiverized
  Clifford algebra in loc.\ cit., Remark~7.6.  See 
  Proposition~\ref{prop:Y-quadratic} for an explanation of their
  disappearance when $m-l>1$.
\end{remark}

\section{The Young quiver with Pieri
  relations}\label{sect:Young-Pieri}

We continue to assume $K$ is a field of characteristic zero.

Now we give an explicit description of the non-commutative
desingularization as a path algebra of a certain quiver with
relations. The vertices of the quiver are identified with partitions
$\alpha \in B_{l,m-l}$, or alternatively with the corresponding vector
bundles $\caln_\alpha ={p'}^*L^\alpha \calq$ on $\calz$, or again with
the corresponding MCM $R$-modules $N_\alpha$.  The arrows from
$\alpha$ to $\beta$ will, in accordance with Example~\ref{eg:smallt},
correspond to (a basis of) $F^\svee$ if $\alpha \nearrow \beta$, and
to (a basis of) $G$ if $\beta \nearrow \alpha$.  To define an explicit
action of the arrows on the modules or bundles, however, requires a
bookkeeping device.

Fix a $K$-vector space $V$ of dimension $d$.  For irreducible
(rational) representations $L^\alpha V$ and $L^\beta V$ of $\GL(V)$,
we know that the tensor product $L^\alpha V \otimes L^\beta V$ has a
canonical decomposition into irreducibles $\bigoplus_\gamma
{(L^\gamma V)}^{c^\gamma_{\alpha\beta}}$ with multiplicities
$c^\gamma_{\alpha\beta}$, but in general the decomposition projectors
are defined only up to some choices of bases for the vector spaces
$\Hom_{\GL(V)}(L^\alpha V\otimes L^\beta V,\ L^\gamma V)$.  To avoid
making these choices, we introduce the following notation.

\begin{definition}
  \label{def:LL}
  Let $\alpha_1, \dots, \alpha_r$ and $\beta_1, \dots, \beta_s$ be
  dominant weights for $\GL(V)$, and set
  \[
  \LL_{\alpha_1 \cdots\alpha_r}^{\beta_1\cdots\beta_s} =
  \Hom_{\GL(V)}(L^{\alpha_1}V\otimes \cdots \otimes L^{\alpha_r}V,\
  L^{\beta_1} V \otimes \cdots \otimes L^{\beta_s}V)\,.
  \]
\end{definition}

The spaces $\LL_{\alpha_1 \cdots\alpha_r}^{\beta_1\cdots\beta_s}$
satisfy various easily-verified properties. Denote by $\alpha^*$ the
dominant weight corresponding to the dual representation $(L^\alpha
V)^\svee = \Hom_{\GL(V)}(L^\alpha V,K)$.

\begin{prop}
  \label{prop:LL-props}
  Let $\alpha_1, \dots, \alpha_r, \beta_1, \dots, \beta_s$ be dominant weights and let
  $\sigma \in S_r$. We have canonical (basis-independent) isomorphisms
  \begin{enumerate}[\quad(i)]
  \item\label{item:LL-props-i} $\LL_{\alpha_1 \cdots\alpha_r}^{\beta_1\cdots\beta_s} =
    \LL_{\alpha_{\sigma(1)} \cdots\alpha_{\sigma(r)}}^{\beta_1\cdots\beta_s}$\,;
  \item\label{item:LL-props-ii} $\LL_{\alpha_1 \cdots\alpha_r}^{\beta_1\cdots\beta_s}
    = \LL_{\alpha_1 \cdots\alpha_{r-1}}^{\beta_1\cdots\beta_s\alpha_r^*}$\,;
  \item\label{item:LL-props-iii} $\LL_{\alpha_1
      \cdots\alpha_r}^{\beta_1\cdots\beta_s} = \bigoplus_\gamma
    \LL_{\gamma\alpha_{i+1} \cdots\alpha_r}^{\beta_1\cdots\beta_s}
    \otimes \LL^\gamma_{\alpha_{1}\cdots\alpha_i}$\,;
  \item \label{item:LL-props-iv} $\left(\LL_{\alpha_1\cdots
        \alpha_r}^{\beta_1\cdots\beta_s}\right)^\svee =
      \LL_{\beta_1\cdots\beta_s}^{\alpha_1\cdots \alpha_r}$\,.\qed
  \end{enumerate}
\end{prop}

The cases of (\ref{item:LL-props-iv}) we will use most often are the
identifications
\[
\left(\LL_{1\alpha}^\beta\right)^\svee = \LL_\beta^{1\alpha}
\qquad\text{and}\qquad
\left(\LL_{1^*\alpha}^\beta\right)^\svee = \LL_\beta^{1^*\alpha}\,.
\]
 Here 
\begin{align*}
\LL_{1\alpha}^{\beta}&\cong
\begin{cases}
K &\text{if $\alpha\nearrow \beta$, and }\\
0&\text{otherwise;}
\end{cases}\\
\LL_{1^\ast\alpha}^{\beta}&\cong
\begin{cases}
K &\text{if $\beta\nearrow \alpha$, and }\\
0&\text{otherwise.}
\end{cases}
\end{align*}
In particular, these properties yield a ``categorified Pieri rule''
yielding a canonical decomposition
\[
V \otimes L^\alpha V \cong \bigoplus_{\beta} L^\beta V\otimes
\left(\LL_{1\alpha}^\beta\right)^\svee\,,
\]
and similarly
\[
V^\svee \otimes L^\alpha V \cong \bigoplus_\beta L^\beta V \otimes
\left(\LL_{1^*\alpha}^\beta\right)^\svee\,,
\]
where the sum in each case is over \emph{all} partitions $\beta$. 
More generally, we have a ``categorified Littlewood-Richardson rule''
\[
L^\alpha V \otimes L^\beta V = \bigoplus_{\gamma} L^\gamma V \otimes
\left(\LL^\gamma_{\alpha\beta}\right)^\svee
\]
and the dimensions of the spaces $\LL^\gamma_{\alpha\beta}$ are given
by the usual Littlewood-Richardson numbers $c^\gamma_{\alpha\beta}$.
There is no canonical choice of bases for the spaces
$\LL^{\alpha\beta}_\gamma$, but see Section~\ref{sect:pierisystems}
below.

Now we are ready to define the (truncated) Young quiver and its action
on the tilting bundle. Return to the notation set in the Introduction,
so that $F^\svee$ and $G$ are $K$-vector spaces of ranks $m$ and $n$
respectively, and $l < \min\{m,n\}$.
\begin{definition}
  \label{def:Youngquiver}
  Let $\YY$ be the quiver having vertices labelled by dominant weights
  $\alpha$ for $\GL(l)$, and arrows $\alpha \to \beta$ indexed by
  \[
  \begin{cases}
    \LL_{1\alpha}^\beta \otimes F^\svee & \text{if } \alpha \nearrow
    \beta\,\text{ and}\\
    \LL_{1^*\alpha}^\beta \otimes G & \text{if } \beta \nearrow
    \alpha\,.
  \end{cases}
  \]
  Further let $\YY_{l,m-l}$ be the subquiver of $\YY$ obtained by
  deleting all vertices $\alpha$ having more than $l$ rows or more
  than $m-l$ columns, as well as all the arrows incident to them.
\end{definition}

To define a ring homomorphism from the path algebra $K[\YY_{l,m-l}]$
to the non-commutative desingularization $A = \End_\caloz{\caln}$ we
must define an action of the arrows on the summands $\caln_\alpha =
{p'}^*L^\alpha \calq$.

\begin{prop}
  \label{prop:map}
  There is a ring homomorphism $K[\YY_{l,m-l}] \to \End_{\caloz}(\caln)$.
\end{prop}

\begin{proof}
  As in the proof of Proposition~\ref{prop:image-sections}, let $\psi
  \colon {q'}^*(G \otimes S) \to {q'}^* (F \otimes S)$ be the pullback
  of the generic map of free $S$-modules to $\calz$, and let $(V,
  \theta)$ be a point of $\calz$.  The fiber of the dual $\psi^\svee$
  over $(V,\theta)$ factors as
  \[
  F^\svee \to V^\svee \to G^\svee
  \]
  so we have induced maps of bundles 
  \[
  {p'}^*\pi^* F^\svee \to {p'}^* \calq \to {p'}^*\pi^* G^\svee\,.
  \]
  Tensoring with $\caln_\alpha = {p'}^* L^\alpha \calq$ and an
  appropriate $\LL$ we obtain natural maps
  \begin{equation}\label{eq:action}
    \begin{split}
  \LL_{1\alpha}^{\beta} \otimes {p'}^* \pi^* F^\svee
  \otimes {p'}^* L^\alpha \calq 
  \to 
  \LL_{1\alpha}^{\beta} \otimes {p'}^* \calq
  \otimes {p'}^* L^\alpha \calq 
  \to 
  {p'}^* L^{\beta}\calq
  \\[2ex]
  \LL_{1^*\alpha}^{\beta} \otimes {p'}^* \pi^* G
  \otimes {p'}^* L^\alpha \calq 
  \to 
  \LL_{1^*\alpha}^{\beta} \otimes {p'}^* \calq^\svee
  \otimes {p'}^* L^\alpha \calq 
  \to 
  {p'}^* L^{\beta}\calq
  \end{split}
  \end{equation}
  for all $\beta$ such that $\alpha \nearrow \beta$, respectively
  $\beta \nearrow \alpha$.
  
  Thus $K[\YY]$ acts on $\caln$ and in fact $K[\YY_{l,m-l}]$ acts
  since $\caln$ contains only bundles $\caln_\alpha$ with $\alpha \in
  B_{l,m-l}$. 
\end{proof}

To identify the kernel of the homomorphism of
Proposition~\ref{prop:map}, observe that if $\gamma$ is obtained by
adding two boxes to $\alpha$, we have canonical decompositions into
one-dimensional spaces
\begin{align}
  \label{eq:F1}
  \LL_{11\alpha}^\gamma &= \LL_{[2]\alpha}^\gamma \oplus
  \LL_{[11]\alpha}^\gamma\\
  \label{eq:F2}
  \LL_{11\alpha}^\gamma &= \bigoplus_{\alpha\nearrow \beta \nearrow
    \gamma} \LL_{1\beta}^\gamma \otimes \LL_{1\alpha}^\beta\,.
\end{align}
If the two boxes are added in the same row, resp.\ column, then
$\LL_{[11]\alpha}^\gamma=0$, resp.\ $\LL_{[2]\alpha}^\gamma =0$, and
the sum in~\eqref{eq:F2} has only one summand.  If, however, the two
boxes are in different rows and columns, each of~\eqref{eq:F1}
and~\eqref{eq:F2} provides the two-dimensional space
$\LL_{11\alpha}^\gamma$ with a basis defined up to scalar multiples,
but these bases are not the same, even up to scalars.

Similarly we have the canonical decompositions
\begin{align}
  \label{eq:G1}
  \LL_{1^*1^*\gamma}^\alpha &= \LL_{[2]^*\gamma}^\alpha \otimes
  \LL_{[11]^*\gamma}^\alpha \\
  \label{eq:G2}
  \LL_{1^*1^*\gamma}^\alpha &= \bigoplus_{\alpha\nearrow \beta
    \nearrow \gamma} \LL_{1^*\beta}^\alpha \otimes
  \LL_{1^*\gamma}^\beta\,,
\end{align}
which again define two essentially different bases for
$\LL_{1^*1^*\gamma}^\alpha$.

If $\gamma$ is obtained by moving a box in $\alpha$ from row $i$ to
row $j$, then we have canonical isomorphisms
\begin{align}
  \label{eq:mix1}
  \LL_{11^*\alpha}^\gamma &= \bigoplus_\beta \LL_{1\beta}^\gamma
  \otimes \LL_{1^*\alpha}^\beta \\
  \label{eq:mix2}
  \LL_{1^*1\alpha}^\gamma &= \bigoplus_\beta \LL_{1^*\beta}^\gamma
  \otimes \LL_{1\alpha}^\beta\,.
\end{align}
As long as $i\neq j$, the space $\LL_{11^*\alpha}^\gamma$ is again
one-dimensional and acquires two different basis elements
from the sums~\eqref{eq:mix1} and~\eqref{eq:mix2}, each of which has
only one non-zero summand.

Finally, for each partition $\alpha$ the dimension of the space
$\LL_{11^*\alpha}^\alpha$ is equal to the number of addable boxes in
$\alpha$, or equivalently the number of ways to remove a box from
$\alpha$ and obtain a dominant weight. (We allow the removal of a
``phantom'' box below the lowest row of $\alpha$.)  Again this space
has a canonical decomposition into one-dimensional spaces
\begin{align}
  \label{eq:star1}
  \LL_{11^*\alpha}^\alpha &= \bigoplus_\beta \LL_{1\beta}^\alpha
  \otimes \LL_{1^*\alpha}^\beta\\
  \label{eq:star2}
  \LL_{1^*1\alpha}^\alpha &= \bigoplus_\beta \LL_{1^*\beta}^\alpha
    \otimes \LL_{1\alpha}^\beta\,.
\end{align}

We use these decompositions~\eqref{eq:F1}--\eqref{eq:star2} of the
universal coefficient spaces to define the relations on $\YY_{l,m-l}$.

\begin{definition}
  \label{def:Yrels}
  We impose relations $I$ on the Young quiver $\YY$ generated by the
  following subspaces of $K[\YY]_2$.
  \begin{enumerate}[\quad(i)]
  \item \label{item:reladd2}For $\gamma$ obtained by adding two boxes to $\alpha$, 
    \[
    \ker\left( 
      \bigoplus_{\alpha\nearrow \beta\nearrow \gamma}
      \LL_{1\beta}^\gamma \otimes F^\svee \otimes \LL_{1\alpha}^\beta
      \otimes F^\svee 
      \to     
      \mbox{\ensuremath{
        \begin{array}{c}
          \LL_{[2]\alpha}^\gamma \otimes \Sym_2 F^\svee \\ 
          {}\oplus{} \\ 
          \LL_{[11]\alpha}^\gamma \otimes \Wedge^2 F^\svee
        \end{array}
      }}
    = 
    \LL_{11\alpha}^\gamma \otimes F^\svee \otimes F^\svee
  \right)\,.
    \]
  \item \label{item:reldel2}For $\gamma$ obtained by deleting two
    boxes from $\alpha$,
    \[
    \ker\left( 
      \bigoplus_{\alpha\nearrow \beta\nearrow \gamma}
      \LL_{1^*\beta}^\alpha \otimes G \otimes \LL_{1^*\gamma}^\beta
      \otimes G
      \to     
      \mbox{\ensuremath{
        \begin{array}{c}
          \LL_{[2]^*\gamma}^\alpha \otimes \Sym_2 G \\ 
          {}\oplus{} \\ 
          \LL_{[11]^*\gamma}^\alpha \otimes \Wedge^2 G
        \end{array}
      }} 
    = \LL_{1^*1^*\gamma}^\alpha \otimes G \otimes G 
  \right)\,.
  \]
  \item \label{item:relmovebox}For $\gamma$ obtained by moving a box in $\alpha$ from row $i$
    to row $j\neq i$,
    \[
    \ker \left(
      \mbox{\ensuremath{
          \begin{array}{c}
            \LL_{1\alpha-\epsilon_i}^\gamma \otimes F^\svee \otimes
            \LL_{1^*\alpha}^{\alpha-\epsilon_i} \otimes G \\ 
            {}\oplus{} \\ 
            \LL_{1^*\alpha+\epsilon_j}^\gamma \otimes G \otimes
            \LL_{1\alpha}^{\alpha+\epsilon_j} \otimes F^\svee
          \end{array}
        }} 
      \to 
      \LL_{11^*\alpha}^\gamma \otimes F^\svee \otimes G
      \right)\,.
      \]
    \item \label{item:relstar} For each partition $\alpha$, 
    \[
    \ker \left(
      \mbox{\ensuremath{
          \begin{array}{c}
            \displaystyle\bigoplus_{\alpha\nearrow\beta}\LL_{1^*\beta}^\alpha
            \otimes G \otimes \LL_{1\alpha}^\beta \otimes F^\svee \\ 
            {}\oplus{} \\ 
            \displaystyle\bigoplus_{\beta\nearrow \alpha} \LL_{1\beta}^\alpha
            \otimes F^\svee \otimes \LL_{1^*\alpha}^\beta \otimes G 
          \end{array}
        }} 
      \to 
            \displaystyle\LL_{11^*\alpha}^\alpha \otimes F^\svee
            \otimes G
    \right)\,.
    \]
  \end{enumerate}
  In each case the indicated maps are defined by the canonical
  decompositions~\eqref{eq:F1}--\eqref{eq:star2}, together with the
  natural surjections $F^\svee \otimes F^\svee \to \Sym_2 F^\svee$,
  $F^\svee \otimes F^\svee \to \Wedge^2 F^\svee$, etc.

  We apply these relations to the truncated Young quiver $\YY_{l,m-l}$
  as well, keeping in mind that any path travelling outside
  $B_{l,m-l}$ is zero.
\end{definition}

\begin{prop}
  \label{prop:checkrels}
  The relations listed in Definition~\ref{def:Yrels} act trivially on
  $\caln$, thus induce a ring homomorphism
  $K[\YY_{l,m-l}]/(\text{relations}) \to \End_{\caloz}(\caln)$.
\end{prop}

\begin{proof}
  This amounts to checking in each case that the composition of two
  arrows in the quiver maps to the Hom-space by the obvious
  projection.  For example, in case~(\ref{item:reladd2}) the
  composition of maps $\caln_\alpha \to \caln_\beta \to \caln_\gamma$,
  where $\alpha \nearrow \beta \nearrow \gamma$, is given by the
  pullback of the evaluation
  \[
  \Hom(\calq\otimes L^\beta \calq,\; L^\gamma \calq) \otimes \calq
  \otimes \Hom(\calq \otimes L^\alpha \calq,\; L^\beta \calq )\otimes
  \calq \otimes L^\alpha \calq 
  \to 
  L^\gamma \calq\,.
  \]
  Shuffling the tensor products around and using the fixed splitting
  $\calq \otimes \calq = \Sym_2 \calq \oplus \Wedge^2\calq$, we can
  rewrite this as
  \[
  \Hom((\Sym_2\calq \oplus \Wedge^2\calq) \otimes L^\alpha \calq,
  L^\gamma \calq) \otimes (\Sym_2\calq \oplus \Wedge^2\calq) \otimes
  L^\alpha \calq \to L^\gamma \calq\,,
  \]
  so that the map is nothing but the natural projection.  Similar
  manipulations take care of the other cases.
\end{proof}

To show that the vector spaces of relations defined in
Definition~\ref{def:Yrels}, after restriction to $\YY_{l,m-l}$, have
the dimensions predicted by~\eqref{eq:Arels}, we must verify that
the maps
\begin{gather}
  \label{eq:relres1}
  \bigoplus_{\beta \in B_{l,m-l}} \LL_{1\beta}^\gamma \otimes
  \LL_{1\alpha}^\beta 
  \to 
  \LL_{[2]\alpha}^\gamma \oplus \LL_{[11]\alpha}^\gamma \\
  \label{eq:relres2}
  \bigoplus_{\beta \in B_{l,m-l}} \LL_{1^*\beta}^\alpha \otimes
  \LL_{1^*\gamma}^\beta 
  \to 
  \LL_{[2]^*\gamma}^\alpha \oplus \LL_{[11]^*\gamma}^\alpha \\
  \label{eq:relres3}
  \left(\LL_{1\alpha-\epsilon_i}^{\alpha-\epsilon_i+\epsilon_j} \otimes
    \LL_{1^*\alpha}^{\alpha-\epsilon_i}\right)
  \oplus
  \left(\LL_{1^*\alpha+\epsilon_j}^{\alpha-\epsilon_i+\epsilon_j} \otimes
    \LL_{1\alpha}^{\alpha+\epsilon_j}\right) 
  \to 
  \LL_{11^*\alpha}^{\alpha-\epsilon_i+\epsilon_j}\\
  \label{eq:relres4}
  \bigoplus_{\substack{\alpha\nearrow \beta \\ \beta \in
        B_{l,m-l}}} \LL_{1^*\beta}^\alpha \otimes
    \LL_{1\alpha}^\beta\ 
  \oplus
  \bigoplus_{\substack{ \beta \nearrow\alpha \\ \beta \in
        B_{l,m-l}}} \LL_{1\beta}^\alpha \otimes
    \LL_{1^*\alpha}^\beta 
  \to \LL_{11^*\alpha}^\alpha 
  \,,
\end{gather}
obtained by restricting all $\alpha,\; \beta,\; \gamma,\;
\alpha-\epsilon_i$ and $\alpha+\epsilon_j$ to lie in the box
$B_{l,m-l}$, remain surjective.
Our proof of this fact relies on an explicit computation 
relating two bases for $\LL_{11^*\alpha}^\alpha$.
In order not to disrupt the flow of the argument we postpone this 
computation to the next section.  See 
Corollary~\ref{cor:mix-matrix-nonzero}.

\begin{lemma}
  \label{lem:surjrels}
  Assume $m-l>1$.  The restricted
  maps~\eqref{eq:relres1}-\eqref{eq:relres4} are surjective.  The
  spaces of relations between two vertices $\alpha,\; \gamma \in
  B_{l,m-l}$ of $\YY_{l,m-l}$ are thus given by
  \begin{equation*}
    \begin{cases}
      \Sym_2 F^\svee & \text{if }\alpha \nearrow \nearrow \gamma, \text{ two boxes in a column}\\
      \Wedge^2 F^\svee & \text{if }\alpha \nearrow \nearrow \gamma, \text{ two boxes in a row} \\
      \Sym_2 F^\svee \oplus \Wedge^2 F^\svee \cong F^\svee \otimes
      F^\svee & \text{if }\alpha \nearrow \nearrow \gamma, \text{ two  disconnected boxes} \\
      F^\svee\otimes G & \begin{minipage}{2.5in} if $\alpha \neq
        \gamma$, and
        $\alpha \nearrow \beta$, $\gamma \nearrow \beta$ \\
        \phantom{asdf} for some
        $\beta$ with $\beta_1 \leq m-l$\end{minipage}\\
      (F^\svee\otimes G)^{\oplus (t(\alpha)-1)} & \text{if }\alpha =\gamma\\
      \Sym_2 G & \text{if }\gamma \nearrow \nearrow \alpha, \text{ two
        boxes in a column} \\
      \Wedge^2 G & \text{if }\gamma \nearrow \nearrow \alpha, \text{
        two
        boxes in a row}\\
      \Sym_2 G \oplus \Wedge^2 G \cong G \otimes G & \text{if }\gamma
      \nearrow \nearrow \alpha, \text{ two disconnected boxes}
    \end{cases}
  \end{equation*}
  where $t(\alpha)$ is the number of ways to add a box to $\alpha$
  without making any row longer than $m-l$.
\end{lemma}

\begin{proof}
  The statements about~\eqref{eq:relres1}, \eqref{eq:relres2},
  and~\eqref{eq:relres3} are clear, since if one of the intermediate
  partitions lies outside $B_{l,m-l}$, then so does $\gamma$ and the
  target of the map vanishes.

  Fix $\alpha \in B_{l,m-l}$.  There is exactly one dominant weight
  $\rho \notin B_{l,m-l}$ such that $\rho \nearrow \alpha$, namely the
  result of deleting the phantom box below the lowest row of
  $\alpha$.  Thus the sum 
  \[
  \bigoplus_{\substack{ \beta \nearrow\alpha \\ \beta \in
      B_{l,m-l}}} \LL_{1\beta}^\alpha \otimes
  \LL_{1^*\alpha}^\beta
  \]
  has $r(\alpha)-1$ summands, where $r(\alpha)$ is the total number of
  ways to add a box to $\alpha$.

  There are two cases, depending on whether the first row of $\alpha$
  has maximal length. If $\alpha_1 < m-l$, then there are no $\beta
  \notin B_{l,m-l}$ with $\alpha \nearrow \beta$, so that 
  \[
  \bigoplus_{\substack{\alpha\nearrow \beta \\ \beta \in
      B_{l,m-l}}} \LL_{1^*\beta}^\alpha \otimes
  \LL_{1\alpha}^\beta
  \to \LL_{11^*\alpha}^\alpha
  \]
  is an isomorphism, and~\eqref{eq:relres4} is surjective.  In this
  case we have $t(\alpha)=r(\alpha)$, and the kernel
  of~\eqref{eq:relres4} has dimension $r(\alpha)-1=t(\alpha)-1$.

  If on the other hand $\alpha_1=m-l$, then there is exactly one
  partition $\sigma \notin B_{l,m-l}$ with $\alpha \nearrow \sigma$.
  To show that~\eqref{eq:relres4} is onto, it suffices to see that the
  images of the one-dimensional spaces $\LL_{1^*\sigma}^\alpha \otimes
  \LL_{1\alpha}^\sigma$ and $\LL_{1\rho}^\alpha \otimes
  \LL_{1^*\alpha}^\rho$ do not coincide in $\LL_{11^*\alpha}^\alpha$.
  This follows from Corollary~\ref{cor:mix-matrix-nonzero} below; the
  matrix relating the two Pieri bases for $\LL_{11^*\alpha}^\alpha$
  has no non-zero entries, so no element of one basis is a scalar
  multiple of an element of the other basis.  Now
  $t(\alpha)=r(\alpha)-1$ in this case, so that the kernel
  of~\eqref{eq:relres4} has dimension
  $(r(\alpha)-1)+(r(\alpha)-1)-r(\alpha) = r(\alpha)-2 = t(\alpha)-1$.
\end{proof}

\begin{remark}
  \label{rem:surjfails}
  In the case $m-l=1$, Lemma~\ref{lem:surjrels} fails; there are cubic
  minimal relations in the quiver~\cite[Remark
  7.6]{Buchweitz-Leuschke-VandenBergh:2010}.  See
  Proposition~\ref{prop:Y-quadratic} for another point of view on
  their disappearance when $m-l>1$.
\end{remark}

\begin{theorem}
  \label{thm:quiveriso}
  Assume $m-l>1$.  The homomorphism $K[\YY_{l,m-l}]/(\text{relations})
  \to A = \End_{\caloz}(\caln)$ is an isomorphism. Thus $A$ is
  isomorphic to the bound path algebra of the Young quiver
  $\YY_{l,m-l}$ having vertices $\alpha \in B_{l,m-l}$ and arrows
  $\alpha \to \beta$ indexed by bases for
  \[
  \begin{cases}
    F^\svee & \text{if $\alpha\nearrow \beta$, and}\\
    G & \text{if $\beta \nearrow  \alpha$,}
  \end{cases}
  \]
  with Pieri relations as indicated in Lemma~\ref{lem:surjrels}.
\end{theorem}

\begin{proof}
  The computation of $\Ext_A^{0,1,2}(S_\beta,\;S_\alpha)$ for simple
  $A$-modules $S_\alpha$ and $S_\beta$ in Example~\ref{eg:smallt}
  shows that $A$ is a quotient of $K[\YY_{l,m-l}]$ with relations
  generated by $(\Ext_A^2(S_\gamma,\;S_\alpha)^\svee)_{\alpha,\gamma}$.
  We also have a surjection $K[\YY_{l,m-l}]/(\text{relations}) \to
  A$.  The induced endomorphism $K[\YY_{l,m-l}] \to K[\YY_{l,m-l}]$
  may not be the identity, but the map $K[\YY_{l,m-l}] \to A$ is
  $\GL(F) \times \GL(G)$-equivariant, and there is a unique such map
  up to scaling arrows.  We may therefore rescale to assume that the
  induced endomorphism of $K[\YY_{l,m-l}]$ is the identity.  

  Write $I$ for the ideal of relations.  Take graded pieces of degree
  $2$ to obtain the following commutative diagram of vector spaces.
\begin{center}
\includegraphics{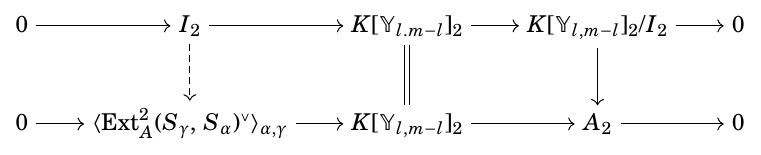}
\end{center}
Now, the dashed arrow is injective, whence an isomorphism since $I_2$
  has the same dimension as $\langle
  \Ext_A^2(S_\gamma,\;S_\alpha)^\svee\rangle_{\alpha,\gamma}$ by
  Example~\ref{eg:smallt}.  It follows that $K[\YY_{l,m-l}] \to A$ is
  an isomorphism.
\end{proof}

\section{Pieri systems}\label{sect:pierisystems}

To extract a really explicit description of the non-commutative
desingularization $A$ from Theorem~\ref{thm:quiveriso}, as well as to
finish the proof of Lemma~\ref{lem:surjrels} and thereby
Theorem~\ref{thm:quiveriso}, we must compute the non-diagonal
surjections $(F^\svee \otimes F^\svee)^{\oplus 2} \onto F^\svee \otimes
F^\svee$, $(G \otimes G)^{\oplus 2} \onto G \otimes G$, and $(F^\svee \otimes
G) \oplus (G \otimes F^\svee) \onto F^\svee \otimes G$ in
Definition~\ref{def:Yrels}.  Equivalently, we must choose bases for
the one-dimensional universal vector spaces appearing in the canonical
decompositions
\begin{equation}\label{eq:decomps}
\begin{tabular}{>{$\displaystyle}r<{$}>{$}c<{$}>{$}c<{$}>{$}c<{$}>{$\displaystyle}l<{$}}
    \LL_{[2]\alpha}^\gamma \oplus\LL_{[11]\alpha}^\gamma
    &=\ 
    &\LL_{11\alpha}^\gamma&
    \ =& \bigoplus_{\alpha\nearrow \beta \nearrow
      \gamma} \LL_{1\beta}^\gamma \otimes \LL_{1\alpha}^\beta\\
    \LL_{[2]^*\gamma}^\alpha \oplus \LL_{[11]^*\gamma}^\alpha
    &=\ 
    &\LL_{1^*1^*\gamma}^\alpha&
    \ =& \bigoplus_{\alpha\nearrow \beta
      \nearrow \gamma} \LL_{1^*\beta}^\alpha \otimes
    \LL_{1^*\gamma}^\beta\\
    \bigoplus_\beta \LL_{1\beta}^\gamma
    \otimes \LL_{1^*\alpha}^\beta
    &=\ 
    &\LL_{1^*1\alpha}^\gamma&
    \ =& \bigoplus_\beta \LL_{1^*\beta}^\gamma \otimes \LL_{1\alpha}^\beta\\
    \bigoplus_\beta \LL_{1\beta}^\alpha \otimes \LL_{1^*\alpha}^\beta
    &=\ 
    &\LL_{1^*1\alpha}^\alpha&
    \ =& \bigoplus_\beta \LL_{1^*\beta}^\alpha \otimes \LL_{1\alpha}^\beta
 \end{tabular}
\end{equation}
where in the first two equations $\gamma$ is obtained by adding two
boxes to $\alpha$ and in the third equation $\alpha$ and
$\gamma$ are related by moving a box from one row to another.

There is no canonical way to make these choices.
P.~Olver~\cite{Olver, Maliakas-Olver:1992} was the first to construct
a coherent set of choices, see also~\cite{Sam-Weyman:2011,
  Ottaviani-Rubei:2006}.  We do not use Olver's intricately
defined maps here, but instead characterize the choices one can
make and show how they determine the scalars in the quiver.

It is more convenient below to work with $\LL_{\beta}^{1\alpha}$
rather than the canonically isomorphic space $\LL_{1^*\beta}^\alpha$.
This replacement gives isomorphic maps to those in
Definition~\ref{def:Yrels}, so makes no difference for the purpose of
identifying the relations.

Throughout this section, $K$ is a field of
characteristic zero and $V$ is a vector space of dimension $d$.  Let
$\epsilon_i$ be the vector $(0, \dots, 0,1,0,\dots, 0)$ with $1$ at
the $i^\text{th}$ position.

Pieri's theorem tells us
\[
V\otimes L^\alpha V\cong \bigoplus_i L^{\alpha+\epsilon_i}V \otimes
\LL_{\alpha+\epsilon_i}^{1\alpha}\,,
\]
where $\LL_{\alpha+\epsilon_i}^{1\alpha}$ is one-dimensional if
$\alpha+\epsilon_i$ is still a partition, and zero otherwise. A
\emph{Pieri system} is a family of non-zero $\GL(V)$-equivariant
linear maps
\[
\chi_{\alpha,i}\colon L^{\alpha+\epsilon_i} V\to V \otimes L^\alpha
V\,.
\]
These maps are unique up to non-zero scalars.  One easily deduces that
for $i<j$ such that $\alpha+\epsilon_i$ and $\alpha+\epsilon_j$ are
partitions one has that
\begin{equation*}
  \Hom_{\GL(V)}(L^{\alpha+\epsilon_i+\epsilon_j}V,\ V \otimes V
  \otimes L^\alpha V)
\end{equation*}
is two-dimensional with basis
\begin{align*}
  \chi_{\alpha,i,j}&=(1\otimes \chi_{\alpha,i})\circ \chi_{\alpha+\epsilon_i,j}\\
  \chi_{\alpha,j,i}&=(1\otimes\chi_{\alpha,j})\circ \chi_{\alpha+\epsilon_j,i}\,.
\end{align*}
\[
\begin{tikzcd}
  {} &  V \otimes L^{\alpha+\epsilon_i}V \arrow{dl}[swap]{1\otimes \chi_{\alpha,i}}\\
V \otimes V \otimes L^\alpha V && L^{\alpha+\epsilon_i+\epsilon_j} \arrow{ul}[swap]{\chi_{\alpha+\epsilon_i,j}} \arrow{dl}{\chi_{\alpha+\epsilon_j,i}} \\
{}  & V \otimes L^{\alpha+\epsilon_j}V \arrow{ul}{1 \otimes \chi_{\alpha,j}}
\end{tikzcd}
\]
Let $\chi^+_{\alpha,i,j}$ and $\chi^-_{\alpha,i,j}$ be obtained by
postcomposing $\chi_{\alpha,i,j}$ respectively with the
symmetrization map $V\otimes V\to \Sym^2 V$ and the
anti-symmetrization map $V\otimes V\to \Wedge^2 V$. By Pieri's theorem
for symmetric and exterior powers we also know that both
$\Hom_{\GL(V)}(L^{\alpha+\epsilon_i+\epsilon_j}V,\ \Sym^2 V\otimes
L^\alpha V)$ and
$\Hom_{\GL(V)}(L^{\alpha+\epsilon_i+\epsilon_j}V,\ \Wedge^2 V \otimes
L^\alpha V)$ are one-dimensional. Furthermore these spaces are clearly
spanned by $\{\chi^+_{\alpha,i,j},\;\chi^+_{\alpha,j,i}\}$ and
$\{\chi^-_{\alpha,i,j},\;\chi^-_{\alpha,j,i}\}$ respectively. This
means we can define scalars (well-defined but not a priori finite or
non-zero at this stage)
\[
\gamma^+_{\alpha,i,j}=\frac{\chi^+_{\alpha,j,i}}{\chi^+_{\alpha,i,j}}\,,
\qquad
\gamma^-_{\alpha,i,j}=\frac{\chi^-_{\alpha,j,i}}{\chi^-_{\alpha,i,j}}\,.
\]
We call $(\gamma^+_{\alpha,i,j})$, $(\gamma^-_{\alpha,i,j})$ the
\emph{(symmetric, exterior) characteristic ratios} of the Pieri system
$(\chi_{\alpha,i})$.

We say that two Pieri systems $\chi,\chi'$ are \emph{equivalent}
(notation: $\chi\sim \chi'$) if there are $(c_\alpha)_\alpha\in
K^\ast$, with $\alpha$ running through the partitions, such that
\[
\chi'_{\alpha,i}=\frac{c_{\alpha+\epsilon_i}}{c_\alpha}\chi_{\alpha,i}\,. 
\]
Clearly two equivalent Pieri systems have the same characteristic ratios.

The following summarizes what we know about Pieri systems.
\begin{proposition}
  \label{ref-2.1-1}
  Let $(\chi_{\alpha,i})_{\alpha,i}$ be a Pieri system with
  characteristic ratios $(\gamma^+_{\alpha,i,j})_{\alpha,i,j}$,
  $(\gamma^-_{\alpha,i,j})_{\alpha,i,j}$.
  \begin{enumerate}[\quad(i)]
  \item \label{item:olver} The characteristic ratios are finite and non-zero.
  \item \label{ref-2-2} We have
    \[
    \frac{\gamma^+_{\alpha,i,j}}{\gamma^-_{\alpha,i,j}}=\frac{u-1}{u+1}
    \]
    where
    \begin{equation}
      \label{ueq}
      u=\frac{1}{ (i-\alpha_i-1) - (j-\alpha_j-1)}\,.
    \end{equation}
    We have written $u$ in this peculiar way to emphasise how it depends on the added
    boxes $(i,\alpha_i+1)$, $(j,\alpha_j+1)$. 
  \item Assume that $\alpha$ is a partition and $i<j<k$ are such that $\alpha+\epsilon_i$,
    $\alpha+\epsilon_j$, $\alpha+\epsilon_k$ are partitions. 
    Then we have 
    \begin{equation}
      \label{ref-2.2-3}
      \begin{aligned}
        \gamma^+_{\alpha+\epsilon_k,ij}\gamma^+_{\alpha,i k}\gamma^+_{\alpha+\epsilon_i,j k}&=
        \gamma^+_{\alpha,j k}\gamma^+_{\alpha+\epsilon_j,i k}\gamma^+_{\alpha,ij}\\
        \gamma^-_{\alpha+\epsilon_k,ij}\gamma^-_{\alpha,i k}\gamma^-_{\alpha+\epsilon_i,j k}&=
        \gamma^-_{\alpha,j k}\gamma^-_{\alpha+\epsilon_j,i k}\gamma^-_{\alpha,ij}\,.
      \end{aligned}
    \end{equation}
  \item 
    \label{ref-4-4}
    Two Pieri systems with the same characteristic ratios are equivalent.
  \item 
    \label{ref-5-5} 
    We can fix either the symmetric or the exterior characteristic ratios
    of a Pieri system arbitrarily provided they satisfy \eqref{ref-2.2-3}.
  \end{enumerate}
\end{proposition}

\begin{remark} 
  \label{remolver} 
  Olver constructs an explicit Pieri system, which we call the
  \emph{classical} system, from the combinatorics of Young
  tableaux. Part (\ref{item:olver}) of the theorem appears in
  \cite[Lemma 8.3]{Olver} and in~\cite[Section
  3]{Maliakas-Olver:1992}, where it is stated for the inverse maps
  $\phi_{\alpha+\epsilon_i,i} \colon V \otimes L^{\alpha} V \to
  L^{\alpha+\epsilon_i} V$ (see Definition~\ref{def:Pieripair} below).
  A detailed proof of the non-vanishing of $\chi_{\alpha,i,j}^+$
  appears in~\cite[Lemma 1.6]{Sam-Weyman:2011}, and their proof is
  easily modified to apply as well to $\chi_{\alpha,i,j}^-$.

  Sam and Weyman also compute~\cite[Cor. 1.8]{Sam-Weyman:2011} the
  scalar multipliers $\gamma^{\pm}_{\alpha,i,j}$ for the classical
  system (though the expression in loc.\ cit.\ for $\gamma^-$ should
  be preceded by a minus sign), and Sam's ``PieriMaps'' package
  implements the calculation of $\chi^+$ in Macaulay2~\cite{Sam:2009}.

  It follows from part (\ref{ref-5-5}) that we may set $\gamma^+=1$ or
  $\gamma^-=1$, but not both.  Indeed, the canonical (basis-free)
  isomorphisms
  \[
  \LL_{\alpha+\epsilon_i + \epsilon_j}^{[2]\alpha} 
  \oplus
  \LL_{\alpha+\epsilon_i+\epsilon_j}^{[11]\alpha} 
  =
  \LL_{\alpha+\epsilon_i+\epsilon_j}^{11\alpha} 
  = 
  (\LL_{\alpha+\epsilon_i+\epsilon_j}^{1\alpha+\epsilon_i} 
  \otimes
  \LL_{\alpha+\epsilon_i}^{1\alpha}) 
  \oplus 
  (\LL_{\alpha+\epsilon_i+\epsilon_j}^{1\alpha+\epsilon_j} 
  \otimes
  \LL_{\alpha+\epsilon_j}^{1\alpha})
  \]
  define four one-dimensional subspaces of the two-dimensional space
  $\LL_{\alpha+\epsilon_i+\epsilon_j}^{11\alpha}$.  Such a
  configuration is essentially classified by a single invariant, the
  cross-ratio, which is independent (up to sign) of all choices. This
  is the origin of the constant in part~(\ref{ref-2-2}) of the theorem.
  In~\cite[\S 8]{Olver}, Olver shows how to renormalize the classical
  system so that $\gamma^+=1$.

  Note also that \eqref{ref-2.2-3} is automatically satisfied if
  $\gamma^{\pm}_{\alpha,i,j}$ depends only on the added boxes
  $(i,\alpha_i+1)$, $(j,\alpha_j+1)$. In other words we may put
  \[
  \gamma^{+}_{\alpha,i,j}=1-u\,,
  \qquad
  \gamma^{-}_{\alpha,i,j}=-(1+u)
  \]
  with $u$ as in \eqref{ueq}. These happen to be the characteristic
  ratios for the classical system.  See Lemma~\ref{ref-5.4-14} and
  Remark~\ref{rem:dontneed}. 
\end{remark}

\subsection*{Schur-Weyl duality}
For $\alpha$ a partition with $|\alpha|=n$, let $H_\alpha$ be the
corresponding irreducible representation of $S_n$.  Consider the
contravariant functor
\( \DD\colon \operatorname{Rep}(S_n)\to \operatorname{Rep}(\GL(V)) \) which sends
$H$ to $\Hom_{S_n}(H, V^{\otimes n})$.  Let $\cals$ be the full
subcategory of $\operatorname{Rep}(S_n)$ spanned by the $H_\alpha$ such that
$\alpha$ has $>d$ parts. Then $\DD$ defines a duality between
$\operatorname{Rep}(S_n)\big/\cals$ and the full subcategory of
$\operatorname{Rep}(\GL(V))$ consisting of polynomial representations.  We
also have
\[
\DD\left(\operatorname{Ind}^{S_{a+b}}_{S_a\times S_b}(H_1\otimes H_2)\right)
=\DD(H_1)\otimes \DD(H_2)
\]
for $H_1$, $H_2$ representations of $S_a$, $S_b$ respectively. 

We may take
\(
L^\alpha V=\DD(H_\alpha)
\).
For partitions $\lambda^1, \dots, \lambda^k,\alpha$, put
\[
{}^\backprime\LL_\alpha^{\lambda^1\cdots\lambda^k}=\Hom_{S_{c_1}\times\cdots \times S_{c_k}}(
H_{\lambda^1}\otimes\cdots\otimes  H_{\lambda^k},\Res^{S_a}_{S_{c_1}\times\cdots
\times S_{c_k}}H_\alpha)
\]
where $|\lambda^i|=c_i$, $|\alpha|=a$. Then
\begin{align*}
  \LL_{\alpha}^{\lambda^1\cdots
    \lambda^k}&=\Hom_{\GL(V)}(\DD(H_\alpha),\; 
  \DD(H_{\lambda^1})\otimes\cdots\otimes  \DD(H_{\lambda^k}))\\
  &=\Hom_{S_{a}}(\Ind^{S_a}_{S_{c_1}\times \cdots \times S_{c_k}}
  (H_{\lambda^1}\otimes \cdots\otimes H_{\lambda^k}),\; H_\alpha)\\
  &=\Hom_{S_{c_1}\times\cdots \times
    S_{c_k}}(H_{\lambda^1}\otimes\cdots \otimes
  H_{\lambda^k},\; \Res^{S_a}_{S_{c_1}\times \cdots \times S_{c_k}}H_\alpha)\\
  &={}^\backprime\LL^{\lambda^1\cdots\lambda^k}_{\alpha}\,.
\end{align*}
We will denote the so obtained canonical isomorphism
${}^\backprime\LL_{\alpha}^{\lambda^1\cdots\lambda^k} \cong
\LL_\alpha^{\lambda^1\cdots\lambda^k}$ also by~$\DD$.  As in
Proposition~\ref{prop:LL-props} we have canonical isomorphisms:
\[
\bigoplus_{\lambda} \LL_{\alpha}^{\beta\lambda} \otimes \LL_{\lambda}^{\delta\epsilon} 
 \to
\LL_{\alpha}^{\beta\delta\epsilon}\,,
\qquad 
\phi_1\otimes \phi_2\mapsto (1\otimes \phi_2)\circ \phi_1\,.
\]
Likewise we have canonical isomorphisms
\[
\bigoplus_{\lambda} {}^\backprime\LL_{\alpha}^{\beta\lambda} \otimes
{}^\backprime\LL_{\lambda}^{\delta\epsilon} 
 \to
{}^\backprime\LL_{\alpha}^{\beta\delta\epsilon}\,,
\qquad 
\theta_1\otimes \theta_2\mapsto (1\otimes \theta_1)  \circ \theta_2\,.
\]
One easily checks that these decompositions are compatible, that is,
\[
\DD((1\otimes \theta_1)  \circ \theta_2)= (1\otimes \DD(\theta_2))
\circ \DD(\theta_1)\,. 
\]
In particular we see that the canonical decomposition
\begin{equation}
\label{ref-3.1-6}
\LL^{1\cdots 1}_{\alpha}=\bigoplus_{\lambda^1,\ldots,\lambda^n=\alpha}
\LL^1_{\lambda^1}\otimes \LL^{1\lambda^1}_{\lambda^2}\otimes \cdots \otimes 
\LL^{1\lambda^{n-1}}_{\lambda^n}
\end{equation}
is the image under $\DD$ of the corresponding canonical decomposition
\begin{equation}
\label{ref-3.2-7}
H_\alpha={}^\backprime\LL^{1\cdots 1}_{\alpha}=\bigoplus_{\lambda^1,\ldots,\lambda^n=\alpha}
{}^\backprime\LL^1_{\lambda^1}\otimes {}^\backprime\LL^{1\lambda^1}_{\lambda^2}\otimes \cdots \otimes 
{}^\backprime\LL^{1\lambda^{n-1}}_{\lambda^n}\,.
\end{equation}
The righthand side of \eqref{ref-3.2-7} is precisely the decomposition into 
one-dimensional subspaces of $H_\alpha$ given by a Young basis. This
observation is due to Jucys \cite{Jucys:1966,Jucys:1971} and is the basis
for the new approach to the representation theory of the
symmetric group in \cite{Okounkov-Vershik:1996} (see equation (1.2) in loc.\ cit.).

Below we follow the setup of \cite{Okounkov-Vershik:1996} but we
formulate the results directly in terms of the decomposition
\eqref{ref-3.1-6}.

\subsection*{The Pieri complex}
The claims (\ref{ref-4-4}) and (\ref{ref-5-5}) in
Proposition~\ref{ref-2.1-1} can be proved directly, but this is
notationally somewhat cumbersome. Therefore we prefer to deduce them
from some topological considerations. This is based on the fact that a
certain cubical complex is contractible.

We define the \emph{Pieri complex} $\PP$ as the cubical set whose
non-degenerate $h$-cubes are given by tuples
$(\alpha,i_1,\ldots,i_h)$ such that $1\le i_1<\dots<i_h\le d$ and
such that $\alpha$ is a partition with at most $d$ parts with the
property that for all $1\le u\le h$ we have that
$\alpha+\epsilon_{i_u}$ is also a partition. Thus the vertices of
$\PP$ are simply the partitions with at most $d$ rows. We say that
$(\alpha',i'_1,\ldots,i'_{h'})$ is a face of
$(\alpha,i_1,\ldots,i_h)$ if either $\alpha'=\alpha$ and
$\{i'_1,\ldots,i'_{h'}\}\subset \{i_1,\ldots,i_h\}$, or 
$\alpha'=\alpha+\epsilon_{i_j}$ for some $j\in \{1,\ldots,u\}$ and
$\{i'_1,\ldots,i'_{h'}\}\subset \{i_1,\ldots,\hat{i}_j,\ldots,i_h\}$.
(The reader should have no difficulty visualizing
$(\alpha,i_1,\ldots,i_h)$ as an $h$-dimensional hypercube; see the
Figure~\ref{fig:Young} for inspiration.) The
following is our basic result about $\PP$.
\begin{proposition} 
\label{ref-4.1-8}
The geometric realization $|\PP|$ of $\PP$ is contractible.
\end{proposition}
\begin{proof} By construction $|\PP|$ is a CW complex.  For $s\ge 0$
  let $\PP_{\le s}\subset \PP$ be the subcomplex of faces that contain
  only vertices $\alpha$ with $|\alpha|\le s$.  We first claim that
  $|\PP_{\le s-1}|$ is a deformation retract of $|\PP_{\le s}|$.

  If $\alpha$ is a vertex in $\PP_{\le s}$ but not in $\PP_{\le s-1}$
  then it belongs to a unique maximal face
  $\square=(\alpha',i_1,\ldots,i_h)$ in $\PP_{\le s}$ and all other
  vertices of $\square$ lie in $\PP_{\le s-1}$.  Thus two different
  such maximal faces intersect each other in $\PP_{\le s-1}$.

  Therefore it is sufficient to retract each such maximal face
  individually to its intersection with $|\PP_{\le s-1}|$.  The
  following picture shows this schematically for a $2$-cube.
  \[
  \includegraphics[height=2cm]{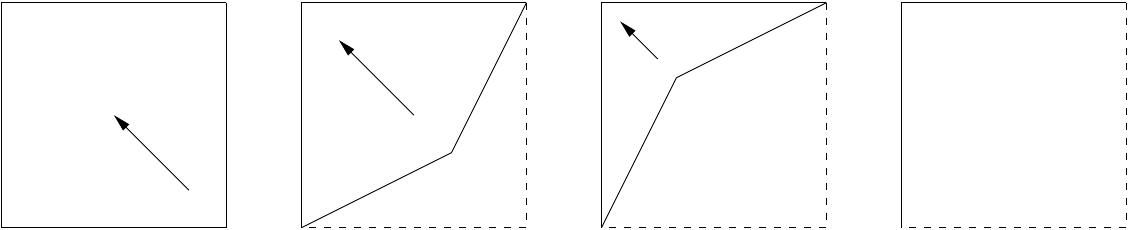}
  \]
  Hence each $|\PP_{\le s}|$ is contractible. So $|\PP|$ is
  contractible as well (see e.g.\
  \cite[Thm. 5.1.35]{Aguilar-Gitler-Prieto:2002}).
\end{proof}

\subsection*{Proof of Proposition~\ref{ref-2.1-1}}
We start by constructing a particular Pieri system using the results
of \cite{Okounkov-Vershik:1996}.  For simplicity we encode a chain of
partitions
\[
\emptyset \nearrow \alpha^1\nearrow \alpha^2\nearrow\cdots\nearrow
\alpha^n=\alpha
\]
by a standard tableau $T$ of shape $\alpha$ where $\alpha^i$ is the
shape of $T_{\le i}$ which is by definition the subtableau of $T$
containing only the letters $1,\ldots,i$.  For a partition~$\alpha$ we
denote by $\diag(\alpha)$ the set of standard tableaux of shape
$\alpha$. The symmetric group $S_n$ partially acts on
$\diag(\alpha)$ by permuting the entries of the tableaux.  If $T\in
\diag(\alpha)$ then we write $\alpha=|T|$. We put
\[
\LL_T=\LL^{1\alpha^1}_{\alpha^2}\otimes \cdots \otimes 
\LL^{1\alpha^{n-1}}_{\alpha^n}
\]
so that \eqref{ref-3.1-6} becomes
\[
\LL^{1\cdots 1}_{\alpha}=\bigoplus_{T\in \diag(\alpha)} \LL_T\,.
\]
Let $T_\alpha$ be the tableau with $1,\ldots,\alpha_i$ in the first
row, $\alpha_i+1,\ldots,\alpha_1+\alpha_2$ in the second row and so
on. We write $T=w_T T_\alpha$ for $w_T\in S_n$. We put $l(T)=l(w_T)$
(see \cite[Remark 6.3]{Okounkov-Vershik:1996}).  If $T\in
\diag(\alpha)$ then a transposition $s=(i,i+1)$ is \emph{admissible}
with respect to $T$ if $i$ and $i+1$ are neither in the same row nor
in the same column.  We say that an admissible transposition is
strongly admissible if it increases $l(T)$.  This happens if and only
if it moves the $i+1$ box upward.

Following \cite{Okounkov-Vershik:1996} we fix a non-zero vector
$v_{T_\alpha}$ in $\LL_{T_\alpha}$ for every partition
$\alpha$. For $T\in \diag(\alpha)$ we define $v_{T}\in \LL_T$ as the
projection of $w_T v_{T_\alpha}\in \LL^{1\cdots 1}_\alpha$ on $\LL_T$.

\begin{proposition} [{see \cite[Prop 5, eq. (7.3)(7.4)]{Okounkov-Vershik:1996}}]
  \label{ref-5.1-9} 
  Let $T\in \diag(\alpha)$ and let $s=(i,i+1)$ be a transposition.
  Then the following hold.
  \begin{enumerate}[\quad(i)]
  \item If $i$ and $i+1$ are in the same row in $T$ then
    \begin{equation*}
      s v_T=v_T\,.
    \end{equation*}
  \item If $i$ and $i+1$ are in the same column in $T$ then
    \begin{equation*}
      s v_T=-v_T\,.
    \end{equation*}
  \item If $s$ is strongly admissible with respect to $T$ and $T'=s T$
    then
    \begin{equation}
      \label{ref-5.3-12}
      \begin{aligned}
        s v_T&=v_{T'}+uv_T\\
        s v_{T'}&=-uv_{T'}+(1-u^2)v_T
      \end{aligned}
    \end{equation}
    with
    \[
    u=\frac{1}{ (k-\alpha_k-1)-(l-\alpha_l-1)}
    \]
  \end{enumerate}
  with $k,l$ being the rows of $i$ and $i+1$ respectively.\qed
\end{proposition}
Note that the case where $s$ is admissible but not strongly admissible
follows by exchanging $T$ and $T'$.

\medskip

The following lemma is a slight extension of
\cite[eq. (7.2)]{Okounkov-Vershik:1996}.
\begin{lemma} \
  \label{ref-5.2-13}
  \begin{enumerate}[\quad(i)]
  \item\label{item:OV1} Let $w\in S_n$ and $T\in \diag(\alpha)$. Then
    \[
    w v_T=\sum_{R\in \diag(\alpha),\;l(R)\le l(T)+l(w)}\gamma_R v_R
    \]
    for some $\gamma_R\in \QQ$.
  \item \label{item:OV2} Assume in addition that $w$ is a product of strongly
    admissible transpositions. Then
    \[
    w v_T=v_{w T}+\sum_{R\in \diag(\alpha),\;l(R)<l(w T)} \gamma_R v_R
    \]
    for some $\gamma_R\in \QQ$.
  \end{enumerate}
\end{lemma}

\begin{proof} 
  Assertion~(\ref{item:OV1}) follows easily from Proposition~\ref{ref-5.1-9} by
  writing $w$ as a composition of transpositions.

  For the second statement, write $w=s w'$ where $s$ is a strongly
  admissible transposition and $w'$ is a product of strongly
  admissible transpositions. By induction we have
  \[
  w' v_T=v_{w' T}+\sum_{R\in \diag(\alpha),\;l(R')<l(w' T)} \gamma'_{R'}
  v_{R'}
  \]
  so that we obtain
  \begin{align*}
    w v_T&=s v_{w' T}+\sum_{R'\in \diag(\alpha),\;l(R')<l(w' T)} \gamma'_{R'} s v_{R'}\\
    &=v_{w T}+uv_{w' T}+\sum_{R'\in \diag(\alpha),\;l(R')<l(w' T)} \gamma'_{R'} s v_{R'}\\
    &=v_{w T}+\sum_{R\in \diag(\alpha),\;l(R)<l(w T)} \gamma_R v_R
  \end{align*}
  where in the second line we have used \eqref{ref-5.3-12} and in the
  third line we have invoked the first part of the lemma.
\end{proof}
Assume now that $T\in \diag(\alpha)$ and that
$\beta=\alpha+\epsilon_i$ is a partition.  Let $T'$ be obtained from
$T$ by adjoining a box labeled $n+1$ at the end of row $i$. Thus
$T'\in \diag(\beta)$.

We now have $v_T\in \LL_T$, $v_{T'}\in \LL_{T'}$. Since
$\LL_{T'}=\LL_{T}\otimes \LL^{1\alpha}_{\beta}$ we may choose
$\chi^c_{T,i}\in \LL^{1\alpha}_{\beta}$ such that $v_{T}\otimes
\chi^c_{T,i}$ and $v_{T'}$ correspond to each other. The following key
result makes everything work.

\begin{lemma} 
  The map $\chi^c_{T,i}$ is independent of the choice of $T\in
  \diag(\alpha)$.
\end{lemma}

\begin{proof}
  If is sufficient to prove that for any $T$ we have
  $\chi^c_{T,i}=\chi^c_{T_\alpha,i}$.  Consider $w_T\in S_n$ as an
  element of $S_{n+1}$.  Let $T'_{\alpha}$ be obtained from
  $T_\alpha$ by adjoining a box labeled $n+1$ at the end of row
  $i$. If we write $w_T\in S_n$ as a product of strongly admissible
  transpositions, then it remains a product of strongly admissible
  transpositions with respect to $T_{\alpha}'$, when considered as an
  element of $S_{n+1}$. Furthermore we have
  $v_{T'}=w_T v_{T'_\alpha}$.

  Let $i\colon \LL_T\to \LL^{1\cdots 1}_\alpha$, $p\colon \LL^{1\cdots
    1}_\alpha \to \LL_T$ be respectively the injection and the
  projection and let $c_{S,S'}\colon \LL_S\to \LL_{S'}$ be the linear
  morphism which sends $v_S$ to $v_{S'}$.  We have the following
  diagram.
\begin{center}
\includegraphics{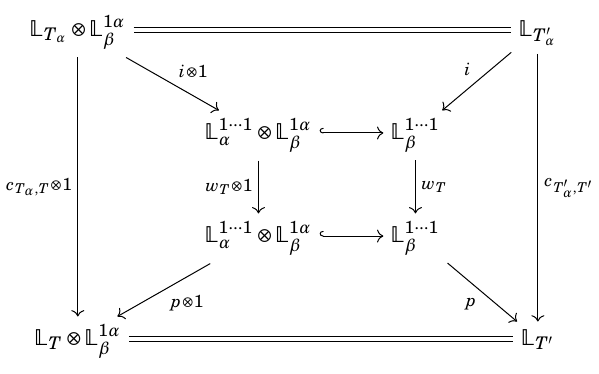}
\end{center}
The commutativity of the leftmost trapezoid is by definition. The
  commutativity of the middle square is clear. The commutativity of
  the rightmost trapezoid follows from Lemma~\ref{ref-5.2-13}(\ref{item:OV2}). The
  commutativity of the upper and lower trapezoid is again by
  construction. From this it is easy to see that the outer square is
  commutative which proves the lemma.
\end{proof}
We now write $\chi_{\alpha,i}^c=\chi^c_{T,i}$ for $T\in \diag(\alpha)$ chosen arbitrarily.
Thus $(\chi^c_{\alpha,i})_{\alpha,i}$ is a particular Pieri system. 

\begin{lemma}
  \label{ref-5.4-14}
  The symmetric and exterior characteristic ratios of
  $(\chi^c_{\alpha,i})_{\alpha,i}$ are respectively given by
  \begin{equation}
    \label{ref-5.4-15}
    \begin{aligned}
      \gamma^{c+}_{\alpha,i,j}&=1-u\\
      \gamma^{c-}_{\alpha,i,j}&=-1-u
    \end{aligned}
  \end{equation}
  with $u$ as in \eqref{ref-5.3-12}.
\end{lemma}

\begin{proof} Assume that $\alpha$, $\alpha+\epsilon_i$,
  $\alpha+\epsilon_j$ are partitions and $i<j$. We have the
  decomposition
  \[
  \LL^{1\alpha+\epsilon_i}_{\alpha+\epsilon_i+\epsilon_j}\otimes
  \LL^{1\alpha}_{\alpha+\epsilon_i} \oplus
  \LL^{1\alpha+\epsilon_j}_{\alpha+\epsilon_i+\epsilon_j}\otimes
  \LL^{1\alpha}_{\alpha+\epsilon_j} \cong
  \LL^{11\alpha}_{\alpha+\epsilon_i+\epsilon_j}\,.
  \]
  To determine the characteristic ratios we have to compose this with
  the canonical maps
  \begin{align*}
    \LL^{11\alpha}_{\alpha+\epsilon_i+\epsilon_j}&\to
    \LL^{[2]\alpha}_{\alpha+\epsilon_i+\epsilon_j}\\
    \LL^{11\alpha}_{\alpha+\epsilon_i+\epsilon_j}&\to
    \LL^{[11]\alpha}_{\alpha+\epsilon_i+\epsilon_j}\,.
  \end{align*}
  After left-multiplying everything with an arbitrary $\LL_T$, $T\in
  \diag(\alpha)$, we may then use the equations \eqref{ref-5.3-12} to
  compute the characteristic ratios, taking into account that $s$ acts
  by $\pm 1$ after projecting to the symmetric, respectively exterior
  square.  It is easy to see that we obtain indeed \eqref{ref-5.4-15}.
\end{proof}

\begin{proof}[Proof of Proposition~\ref{ref-2.1-1}]\
  \begin{enumerate}[\quad(i)]
  \item It is sufficient to prove this for
    $(\chi^c_{\alpha,i})_{\alpha,i}$ where it follows directly from
    Lemma~\ref{ref-5.4-14}.
  \item One easily checks that the ratio $\gamma^+/\gamma^-$ is the
    same for every Pieri system. Again the conclusion follows from
    Lemma~\ref{ref-5.4-14}.
  \item This follows by writing down the $6$ possible maps $L^\alpha
    V\to V\otimes V\otimes V\otimes L^\alpha V$ that can arise as
    compositions of maps in the Pieri system and applying the
    symmetrization $V\otimes V\otimes V \to \Sym^3 V$ and
    anti-symmetrization $V\otimes V\otimes V \to \Wedge^3 V$ to them.
  \item Assume that $\left(\chi^1_{\alpha,i}\right)_{\alpha,i}$,
    $\left(\chi^2_{\alpha,i}\right)_{\alpha,i}$ are Pieri systems with the same
    characteristic ratios. Put
    $\mu_{\alpha,i}=\chi^2_{\alpha,i}\big/\chi^1_{\alpha,i}$. Then
    \begin{equation}
      \label{ref-5.5-16}
      \frac{
        \mu_{\alpha+\epsilon_i,j}\cdot\mu_{\alpha,i}
      }
      {
        \mu_{\alpha,j}\cdot \mu_{\alpha+\epsilon_j,i} 
      }
      =
      1\,.
    \end{equation}
    We have to find $c_\alpha\in K^\ast$ such that
    \begin{equation}
      \label{ref-5.6-17}
      \mu_{\alpha,i}=\frac{c_{\alpha+\epsilon_i}}{c_\alpha}\,.
    \end{equation}
    The condition~\eqref{ref-5.5-16} implies that $\mu$ represents a
    cocycle in the cochain complex $\calc^\bullet(\PP,K^\ast)$. Since
    $\PP$ is contractible by Proposition~\ref{ref-4.1-8}, $\mu$
    must be a coboundary.  This amounts precisely to $\mu$ being writable
    in the form \eqref{ref-5.6-17}.
  \item We will only discuss the symmetric characteristic ratios. The
    exterior characteristic ratios are entirely similar. Assume we
    want to construct a Pieri system
    $\left(\chi_{\alpha,i}\right)_{\alpha,i}$ with prescribed
    $\gamma^+_{\alpha,i,j}$ satisfying \eqref{ref-2.2-3}. Put
    $\delta_{\alpha,i,j}=\gamma^+_{\alpha,i,j}/
    \gamma^{c+}_{\alpha,i,j}$. Then $\delta$ satisfies the equation
    \begin{equation}
      \label{ref-5.7-18}
      \begin{aligned}
        \frac{ \delta_{\alpha+\epsilon_i,j k} \delta_{\alpha,i k}
          \delta_{\alpha+\epsilon_k,ij} }{
          \delta_{\alpha,j k}\delta_{\alpha+\epsilon_j,i k}\delta_{\alpha,ij}}=1\,. 
      \end{aligned}
    \end{equation}
    We put
    $\chi_{\alpha,i}=\mu_{\alpha,i}\chi^c_{\alpha,i}$. It follows
    that $\mu_{\alpha,i}$ must satisfy
    \begin{equation}
      \label{ref-5.8-19}
      \frac{
        \mu_{\alpha+\epsilon_i,j}\cdot\mu_{\alpha,i}
      }
      {
        \mu_{\alpha,j} \cdot \mu_{\alpha+\epsilon_j,i} 
      }
      =
      \delta_{\alpha,i,j}\,.
\end{equation}
The condition \eqref{ref-5.7-18} implies that $\delta$ represents a
cocycle in the cochain complex $\calc^\bullet(\PP,K^\ast)$. Since
$\PP$ is contractible by Proposition~\ref{ref-4.1-8}, $\delta$ must be
a coboundary.  This amounts precisely to $\delta$ being writable in
the form \eqref{ref-5.8-19}.
\end{enumerate}
\end{proof}

\begin{remark} \label{rem:dontneed} If we combine
  Proposition~\ref{ref-2.1-1}(\ref{ref-4-4}), Remark~\ref{remolver},
  and Lemma~\ref{ref-5.4-14} we see that the classical Pieri system
  constructed by Olver is equivalent to
  $(\chi^c_{\alpha,i})_{\alpha,i}$. Recall that the construction of
  $(\chi^c_{\alpha,i})_{\alpha,i}$ depends on the choice of a basis
  element in $\LL_{T_\alpha}$ for each partition. Since we don't need
  it we have not verified which basis element one should take to
  obtain equality rather than equivalence.
\end{remark}

\begin{remark}
  \label{rem:define-gamma}
  We extend the definitions of $\gamma_{\alpha,i,j}^\pm$ to include
  the possibilities $i=j$ or $\alpha_i=\alpha_j$ by
  \[
  \gamma_{\alpha,i,i}^+=1
  \qquad\text{and}\qquad
  \gamma_{\alpha,i,i}^- =0\,,
  \]
  while 
  \[
  \gamma_{\alpha,i,j}^+=0
  \qquad\text{and}\qquad
  \gamma_{\alpha,i,j}^-=1\
  \]  
  if $\alpha_i=\alpha_j$.
\end{remark}

\bigskip

We also require basis vectors for the one-dimensional spaces
$\LL_{1\alpha-\epsilon_i}^{\alpha}$.  

\begin{definition}\label{def:Pieripair}
  A \emph{compatible pair of Pieri systems} consists of two families
  of non-zero equivariant maps
  \[
  \chi_{\alpha,i} \colon  L^{\alpha+\epsilon_i} V \to
  V \otimes L^\alpha V
  \]
  \[
  \phi_{\alpha,i} \colon V \otimes L^{\alpha-\epsilon_i} V\to L^\alpha
  V
  \]
  such that for each $\alpha$ the composition
  \begin{equation}\label{eq:compat}
  L^{\alpha+\epsilon_i} V \xto {\chi_{\alpha,i}}
  V \otimes L^\alpha V \xto{\phi_{\alpha+\epsilon_i,i}} 
  L^{\alpha+\epsilon_i} V 
  \end{equation}
  is the identity on 
  $L^{\alpha+\epsilon_i} V$.
\end{definition}

One can of course make other choices of normalization for the
compatibility condition in Definition~\ref{def:Pieripair}. One natural
choice is to require that~\eqref{eq:compat} is given by multiplication
by the scalar $\dim_K L^\alpha V$.  This complicates the formulas
below only slightly.

The relations among the maps in a dual Pieri system are completely
determined by the compatibility condition~\eqref{eq:compat} and the
relations in Proposition~\ref{ref-2.1-1}.  Let $\alpha$ be a partition
and $i<j$ such that $\alpha+\epsilon_i$ and $\alpha+\epsilon_j$ are
both partitions, so we have the picture below.
\[
\begin{tikzcd}[column sep = small]
  {} &  V \otimes L^{\alpha+\epsilon_i}V\arrow{dr}{\phi_{\alpha+\epsilon_i+\epsilon_j,j}}\\
V \otimes V \otimes L^\alpha V  \arrow{ur}{1\otimes \phi_{\alpha+\epsilon_i,i}} \arrow{dr}[swap]{1 \otimes \phi_{\alpha+\epsilon_j,j}} && L^{\alpha+\epsilon_i+\epsilon_j}   \\
{}  & V \otimes L^{\alpha+\epsilon_j}V \arrow{ur}[swap]{\phi_{\alpha+\epsilon_i+\epsilon_j,i}}
\end{tikzcd}
\]
Set
\begin{align*}
\phi_{\alpha,i,j} 
&= \phi_{\alpha+\epsilon_i +\epsilon_j,i}\circ(1\otimes \phi_{\alpha+\epsilon_j,j})\\
\phi_{\alpha,j,i} 
&= \phi_{\alpha+\epsilon_i +\epsilon_j,j}\circ(1\otimes \phi_{\alpha+\epsilon_i,i})\,.
\end{align*}
Let $\phi_{\alpha,i,j}^+ \in \LL_{[2]\alpha}^{\alpha+\epsilon_i
  +\epsilon_j}$ and $\phi_{\alpha,i,j}^- \in
\LL_{[11]\alpha}^{\alpha+\epsilon_i +\epsilon_j}$ be obtained by
symmetrizing, resp.\ anti-symmetrizing the input, and define
characteristic ratios
\[
    \delta_{\alpha,i,j}^+ =
  \frac{\phi_{\alpha,j,i}^+}{\phi_{\alpha,i,j}^+}\,,
  \qquad
  \delta_{\alpha,i,j}^- =
  \frac{\phi_{\alpha,j,i}^-}{\phi_{\alpha,i,j}^-}\,. 
\]

\begin{prop}
  \label{prop:phi-props}
  Let $(\chi_{\alpha,i})$ and $(\phi_{\alpha,i})$ be a compatible pair
  of Pieri systems and let $(\gamma_{\alpha,i,j}^+)$,
  $(\gamma_{\alpha,i,j}^-)$ be the characteristic ratios for
  $(\chi_{\alpha,i})$.  Then
  \begin{enumerate}[\quad(i)]
  \item The characteristic ratios $\delta_{\alpha,i,j}^\pm$ are finite
    and non-zero.
  \item We have  
    \[ 
    \delta_{\alpha,i,j}^+ =
    -
    \gamma_{\alpha,i,j}^-
    \qquad\text{and}\qquad
    \delta_{\alpha,i,j}^- =
    -
    \gamma_{\alpha,i,j}^+
    \,.
    \]
    In particular
    \[ 
    \frac{\delta_{\alpha,i,j}^+}{\delta_{\alpha,i,j}^-} =
    \frac{\gamma_{\alpha,i,j}^-}{\gamma_{\alpha,i,j}^+} =
    \frac{u+1}{u-1}\,,
    \]
    where $u$ is as in~\eqref{ueq}.
  \end{enumerate}
\end{prop}

Proposition~\ref{prop:phi-props} follows immediately from the next
Lemma, which will also be used in results below.    Observe that
\begin{align*}
  \phi_{\alpha,i,j} \chi_{\alpha,j,i} 
  &=
  \phi_{\alpha+\epsilon_i+\epsilon_j,i}
  (1\otimes \phi_{\alpha+\epsilon_j,j})
  (1\otimes \chi_{\alpha,j})
  \chi_{\alpha+\epsilon_j,i}\\
  &= 1 
\end{align*}
as a map $L^{\alpha+\epsilon_i+\epsilon_j}V \to V \otimes V \otimes
L^{\alpha} V \to L^{\alpha+\epsilon_i+\epsilon_j}V$.  We wish to
compute $\phi_{\alpha,i,j}^\pm \chi_{\alpha,j,i}^\pm$, which amounts
to understanding the effect of inserting the projectors $V \otimes V
\to \Sym_2V \to V \otimes V$ and $V \otimes V \to \Wedge^2 V \to V
\otimes V$.

Note on the other hand that $\phi_{\alpha,j,i} \chi_{\alpha,j,i}=0$.  Indeed,
\[
\phi_{\alpha,j,i} \chi_{\alpha,j,i} =
\phi_{\alpha+\epsilon_i+\epsilon_j,j}(1\otimes
\phi_{\alpha+\epsilon_i,i}) (1\otimes \chi_{\alpha,j})
\chi_{\alpha+\epsilon_j,i}
\]
and the middle two maps $(1\otimes \phi_{\alpha+\epsilon_i,i})
(1\otimes \chi_{\alpha,j})$ comprise
\[
1 \otimes \phi_{\alpha+\epsilon_i,i}\chi_{\alpha,j}\colon V \otimes
L^{\alpha+\epsilon_j}V \to V \otimes V\otimes L^{\alpha}V \to V
\otimes L^{\alpha+\epsilon_i}V\,,
\]
but there are no non-zero maps $L^{\alpha+\epsilon_j}V \to
L^{\alpha+\epsilon_i}V$.

\begin{lemma}\label{lem:portlandia}
 We have
 \begin{gather}
   \label{eq:fred}
   \phi_{\alpha,i,j}^+ \chi_{\alpha,j,i}^+ 
   \quad=\quad 
   \frac{-\gamma^+}{\gamma^- - \gamma^+}
   \quad=\quad 
   \phi_{\alpha,j,i}^- \chi_{\alpha,i,j}^- \\
   \label{eq:carrie}
   \phi_{\alpha,i,j}^- \chi_{\alpha,j,i}^- 
   \quad=\quad 
   \frac{\gamma^-}{\gamma^- - \gamma^+}
   \quad=\quad 
   \phi_{\alpha,j,i}^- \chi_{\alpha,i,j}^+
 \end{gather}
  where $\gamma^\pm = \gamma_{\alpha,i,j}^\pm$.
\end{lemma}

\begin{proof}
  Suppress $\alpha$ from the notation, writing simply $\phi_{ij}$,
  etc. Applying $\Hom(L^{\alpha+\epsilon_i+\epsilon_j} V,-)$ to the
  composition $L^{\alpha+\epsilon_i+\epsilon_j}V \to V \otimes V
  \otimes L^{\alpha} V \to L^{\alpha+\epsilon_i+\epsilon_j}V$, we see
  that computing $\phi_{ij}^\pm \chi_{ji}^\pm$ is the same as finding
  the image of $\chi_{ji}^\pm \in
  \LL_{\alpha+\epsilon_i+\epsilon_j}^{11\alpha}$ under the sequence of
  maps
  \[
  \begin{tikzcd}[column sep = huge]
    \LL_{\alpha+\epsilon_i+\epsilon_j}^{11\alpha} 
    \arrow{r}{(1\otimes \phi_{\alpha+\epsilon_j,j})\circ-}
    & \LL_{\alpha+\epsilon_i+\epsilon_j}^{1,\alpha+\epsilon_j}
    \arrow{r}{\phi_{\alpha+\epsilon_i+\epsilon_j,i}\circ-} 
    & \LL_{\alpha+\epsilon_i+\epsilon_j}^{\alpha+\epsilon_i+\epsilon_j}\,.
  \end{tikzcd}
  \]
  We know that $\chi_{ji} \mapsto 1$ and that $\chi_{ij} \mapsto 0$,
  so we have only to rewrite the basis 
  $\{\chi_{ji}^+, \chi_{ji}^- \}$ in terms of the basis $\{\chi_{ij},
  \chi_{ji}\}$.  We have
  \begin{align*}
    \chi_{ij} &= \chi_{ij}^+ + \chi_{ij}^-
    =\frac{1}{\gamma^+}\chi_{ji}^+ +
    \frac{1}{\gamma^-}\chi_{ji}^- \\
    \chi_{ji} &= \chi_{ji}^+ + \chi_{ji}^-\,.
  \end{align*}
  Inverting the $2\times 2$ matrix $\left[\begin{smallmatrix}
      1/\gamma^+ & 1/\gamma^-\\ 1 & 1 \end{smallmatrix}\right]$
  gives      
  \begin{equation}
  \begin{split}\label{eq:chi-ji-pm}
    \chi_{ji}^+ &=
    \frac{\gamma^+\gamma^-}{\gamma^--\gamma^+}\left(\chi_{ij} -
    \frac{1}{\gamma^-}\chi_{ji}\right)\\
    \chi_{ji}^- &=
    \frac{\gamma^+\gamma^-}{\gamma^--\gamma^+}\left(-\chi_{ij} +
    \frac{1}{\gamma^+}\chi_{ji}\right)\,.
  \end{split}
  \end{equation}
  Composing with $\phi_{ij}$ gives the first equality in each of~\eqref{eq:fred}
  and~\eqref{eq:carrie}.  A similar argument establishes the other;
  alternatively, note that interchanging $i$ and $j$ just amounts to
  replacing $\gamma^\pm$ by $1/\gamma^\pm$.
\end{proof}

\medskip

Given a pair of partitions $\alpha,\; \gamma$ such that $\alpha
+\epsilon_j = \gamma +\epsilon_i$ we obtain two compositions of Pieri
maps
\begin{equation}\label{eq:mixpent}
\begin{tikzcd}[row sep=small, column sep = small]
  {} & V\otimes L^\alpha V \arrow{dl}[swap]{1\otimes\chi_{\lambda,i}} \arrow{ddr}{\phi_{\beta,j}} \\
  V\otimes V \otimes L^\lambda V \arrow{dd}[swap]{\tau\otimes1} \\
  {} & {} & L^\beta V \arrow{ddl}{\chi_{\gamma,i}}\\
  V\otimes V \otimes L^\lambda V \arrow{dr}[swap]{1\otimes \phi_{\gamma,j}} \\
  {} & V\otimes L^\gamma V
\end{tikzcd}
\end{equation}
where $\tau\colon V \otimes V \to V \otimes V$ denotes the swap,
$\beta = \alpha +\epsilon_j = \gamma +\epsilon_i$ and $\lambda =
\alpha -\epsilon_i = \gamma -\epsilon_j$.  The
diagram~\eqref{eq:mixpent} is not commutative; there are non-trivial
quadratic relations on the Young quiver relating the two paths.

There are two cases to consider, according to whether $\alpha =
\gamma$.

Assume first that $\alpha \neq \gamma$. Then
$\LL_{1\alpha}^{1\gamma}$ is one-dimensional,
so we may define another scalar
\begin{equation}\label{eq:def-eta}
m_{\alpha,i,j} = \frac{(1\otimes
  \phi_{\gamma,j})(\tau\otimes1)(1\otimes \chi_{\lambda,i})}
{\chi_{\gamma,i}\phi_{\beta,j}}
\end{equation}
equal to the ratio of the two paths around the diagram above.  

In the other case $\alpha=\gamma$, the space $\LL_{1\alpha}^{1\alpha}$
is no longer one-dimensional, rather, has dimension equal to the
number of ways to add a box to $\alpha$ to obtain a partition. This is
equal to the number of ways to remove a box from $\alpha$ leaving a
dominant weight. Denote this number $r(\alpha)$, let $\Delta_\alpha$
be the set of indices $i$ such that $\alpha+\epsilon_i$ is a
partition, and let $\nabla_\alpha$ be the set of indices $j$ such that
$\alpha-\epsilon_j$ is a dominant weight.

The canonical decompositions
\begin{align*}
  \LL^{1\alpha}_{1\alpha} &= \bigoplus_{i \in \Delta_\alpha}
  \LL_{\alpha+\epsilon_i}^{1\alpha} \otimes \LL_{1\alpha}^{\alpha+\epsilon_i}\\
  &= \bigoplus_{j \in \nabla_\alpha}
  \LL_{11\alpha-\epsilon_j}^{1\alpha} \otimes
  \LL_{1\alpha}^{11\alpha-\epsilon_j}
\end{align*}
equip the $r(\alpha)$-dimensional space $\LL_{1\alpha}^{1\alpha}$ with
two bases, $\left(\chi_{\alpha,i}\phi_{\alpha+\epsilon_i,i}\right)_{i
  \in \Delta_\alpha}$ and $\left((1\otimes\phi_{\alpha,j})(\tau\otimes
  1)(1\otimes \chi_{\alpha-\epsilon_j,j})\right)_{j \in
  \nabla_\alpha}$.  We adopt the convention that the former,
corresponding to adding and then removing boxes, is the ``natural''
basis.  Then for each $j \in \nabla_\alpha$, there are uniquely defined
scalars $c_{\alpha,i,j}$ such that
\begin{equation}\label{eq:cij}
(1\otimes \phi_{\alpha,j})(\tau\otimes 1)(1\otimes
\chi_{\alpha-\epsilon_j,j})
= 
\sum_{i \in \Delta_\alpha} 
c_{\alpha,i,j}
\chi_{\alpha,i}\phi_{\alpha+\epsilon_i,i}\,.
\end{equation}

To compute the scalars $m_{\alpha,i,j}$ and $c_{\alpha,i,j}$, we
need the following lemma.
\begin{lemma}
  \label{lem:swap-compose}
  We have
  \begin{align*}
    \phi_{\alpha,j,i}(\tau \otimes 1)\chi_{\alpha,i,j} 
    &=
    \frac{\gamma^- + \gamma^+}{\gamma^- - \gamma^+}
    \\
    \phi_{\alpha,i,j}(\tau \otimes 1)\chi_{\alpha,i,j} &
    = \frac{-2}{\gamma^- - \gamma^+}
  \end{align*}
  where as before $\gamma^\pm = \gamma_{\alpha,i,j}$.
\end{lemma}

\begin{proof}
  As in the proof of Lemma~\ref{lem:portlandia}, we abbreviate
  $\chi_{\alpha,i,j}$ as $\chi_{ij}$ and so on. Also as in that proof,
  write $\chi_{ij}^\pm$ in terms of $\chi_{ij}$ and $\chi_{ji}$:
  \begin{align*}
    \chi_{ij}^+ 
    &=
    \frac{1}{\gamma^--\gamma^+} \left(\gamma^- \chi_{ij} -
      \chi_{ji}\right)\\
    \chi_{ij}^-
    &=
    \frac{1}{\gamma^--\gamma^+} \left(-\gamma^+ \chi_{ij} +
      \chi_{ji}\right)\,.
  \end{align*}
  Since $\tau$ acts as $+1$ on $\Sym_2 V$ and $-1$ on $\Wedge^2 V$, we have
  \begin{align*}
    (\tau \otimes 1)\chi_{ij}
    &=
    \chi_{ij}^+ - \chi_{ij}^- \\
    &= 
    \frac{1}{\gamma^--\gamma^+} \left(
      \left(\gamma^-+\gamma^+\right) \chi_{ij} 
      -2\chi_{ji}\right)\,,
  \end{align*}
  and the desired formulas follow since  $\phi_{ji}\chi_{ij} =1$ and
  $\phi_{ij}\chi_{ij}=0$. 
\end{proof}

\begin{prop}
  \label{prop:mixed-i-neq-j}
  Let $\alpha,\; \gamma$ be partitions such that $\alpha+\epsilon_j =
  \gamma+\epsilon_i$ for some $i \neq j$.  Set $\beta =
  \alpha+\epsilon_j =\gamma+\epsilon_i$ and $\lambda =
  \alpha-\epsilon_i = \gamma-\epsilon_j$ as in~\eqref{eq:mixpent}.  
  Then
  \[
  m_{\alpha,i,j} = \frac{-2}{\gamma_{\lambda,i,j}^- -
    \gamma_{\lambda,i,j}^+}\,.
  \]
\end{prop}

\begin{proof}
  Apply $\Hom(L^{\beta} V,-)$ to the diagram~\eqref{eq:mixpent} to
  obtain the pentagon below.
  \begin{equation}\label{eq:LLpent}
    \begin{tikzcd}[row sep=tiny, column sep = normal]
      {} & \LL_\beta^{1\alpha} \arrow{dl}[swap]{(1\otimes\chi_{\lambda,i})\circ-} \arrow{ddr}{\phi_{\beta,j}\circ-} \\
      \LL_\beta^{11\lambda}\arrow{dd}[swap]{\tau\otimes1} \\
      {} & {} & \LL_\beta^\beta \arrow{ddl}{\chi_{\gamma,i}\circ-}\\
      \LL_\beta^{11\lambda} \arrow{dr}[swap]{(1\otimes \phi_{\gamma,j})\circ-} \\
      {} & \LL_\beta^{1\gamma}
    \end{tikzcd}
\end{equation} 
At the top of \eqref{eq:LLpent} we have the basis element $\chi_{\alpha,j}
\in \LL_\beta^{1\alpha}$. Following this vector down the
right-hand side of the diagram, we find at the bottom
\[
\chi_{\gamma,i}\phi_{\beta,j} \chi_{\alpha,j} =
 \chi_{\gamma,i} \in
\LL_\beta^{1\gamma}\,.
\]
On the other hand, $\chi_{\alpha,j}$ maps leftward to 
\[
 (1\otimes \chi_{\lambda,i})\chi_{\alpha,j} = \chi_{\lambda,i,j} \in
\LL_\beta^{11\lambda}\,.
\]
By the definition of $m_{\alpha,i,j}$, we have 
\[
(1\otimes \phi_{\gamma,j})(\tau\otimes 1)\chi_{\lambda,i,j} = m_{\alpha,i,j}
\chi_{\gamma,i} \in \LL_\beta^{1\gamma}\,.
 \]
Then composing with $\phi_{\beta,i}$ gives
\begin{align*}
  \phi_{\lambda,i,j}(\tau\otimes 1)\chi_{\lambda,i,j} &=
  \phi_{\beta,i} (1\otimes \phi_{\gamma,j})(\tau\otimes
  1)\chi_{\lambda,i,j} \\
  &= m_{\alpha,i,j} \phi_{\beta,i}\chi_{\gamma,i}\\
  &= m_{\alpha,i,j}\,.
\end{align*}
Now Lemma~\ref{lem:swap-compose} finishes the proof.
\end{proof}

\begin{prop}
  \label{prop:mix-matrix-2}
  Let $i,\; j$ be such that $\alpha+\epsilon_i$ is a partition and
  $\alpha-\epsilon_j$ is a dominant weight.  Then
  \[
  c_{\alpha,i,j} = 
  \frac{\gamma_{\alpha-\epsilon_j,i,j}^+ +
    \gamma_{\alpha-\epsilon_j,i,j}^-}
  {\gamma_{\alpha-\epsilon_j,i,j}^+ -
    \gamma_{\alpha-\epsilon_j,i,j}^-}\,. 
  \]
\end{prop}

\begin{proof}
  Fix $k \in \Delta_\alpha$, and pre-compose the
  equation~\eqref{eq:cij} with $\chi_{\alpha,k}$ while post-composing
  with $\phi_{\alpha+\epsilon_k,k}$.  On the right-hand side, the
  result is
  \[
  \sum_{i \in \Delta_\alpha} 
  c_{ij}\ 
  \phi_{\alpha+\epsilon_k,k}
  \chi_{\alpha,i}\phi_{\alpha+\epsilon_i,i}
  \chi_{\alpha,k}\,.
  \]
  For $i \neq k$, note that $\phi_{\alpha+\epsilon_k,k}
  \chi_{\alpha,i} \colon L^{\alpha+\epsilon_i}V \to V \otimes L^\alpha
  V \to L^{\alpha+\epsilon_k}$ is the zero map.  Hence the entirety of the
  right-hand side is
  \[
  c_{kj}\ 
  \phi_{\alpha+\epsilon_k,k}
  \chi_{\alpha,k}\phi_{\alpha+\epsilon_k,k}
  \chi_{\alpha,k}
  = 
  c_{kj}
  \,.
  \]

  On the other side, we obtain
  \begin{align*}
  \phi_{\alpha+\epsilon_k,k} 
  (1\otimes \phi_{\alpha,j})
  (\tau\otimes 1)
  (1\otimes \chi_{\alpha-\epsilon_j,j})
  \chi_{\alpha,k} 
  &=
  \phi_{\alpha-\epsilon_j,k,j}
  (\tau\otimes 1)
  \chi_{\alpha-\epsilon_j,j,k}\\
  &=
  \frac{\gamma_{\alpha-\epsilon_j,j,k}^- +
    \gamma_{\alpha-\epsilon_j,j,k}^+}
  {\gamma_{\alpha-\epsilon_j,j,k}^- - \gamma_{\alpha-\epsilon_j,j,k}^+}
  \end{align*}
  by Lemma~\ref{lem:swap-compose}. To get the result in terms of
  $\gamma_{\alpha-\epsilon_j,k,j}^\pm$, replace each $\gamma$ appearing
  by its reciprocal.
\end{proof}

\begin{cor}
  \label{cor:mix-matrix-nonzero}
  For any Pieri system, any $\alpha$, and any $i,j$, we have
  $c_{\alpha,i,j} \neq 0$.
\end{cor}

\begin{proof}
  If $\gamma_{\alpha-\epsilon_j,i,j} =
  -\gamma_{\alpha-\epsilon_j,i,j}$ then $1-u=1+u$, so that $u=0$,
  which is impossible by the definition of $u$.
\end{proof}

This finishes the proof of Lemma~\ref{lem:surjrels} and therefore
Theorem~\ref{thm:quiveriso}.\qed 

\begin{remark}
  \label{rem:classical}
  If the given Pieri system $(\chi_{\alpha,i})$ is equivalent to the
  classical system, so that $\gamma_{\alpha,i,j}^+ = 1-u$ and
  $\gamma_{\alpha,i,j}^- = -(1+u)$ with 
  \[
  u = \frac{1}{(i-\alpha_i-1)-(j-\alpha_j-1)}\,,
  \]
  then the other scalars can also be written in terms of $u$:
  \[
  \delta_{\alpha,i,j}^+ = u-1\,;
  \qquad
  \delta_{\alpha,i,j}^-=u+1\,;
  \]
  \[
  m_{\alpha,i,j} = 1\,;
  \]
  and
  \[
  c_{\alpha,i,j} = \frac{u}{u+1}\,.
  \]
\end{remark}

We finish the section by making explicit the relations on the Young
quiver (Definition~\ref{def:Yrels}).

\begin{theorem}
  \label{thm:explicitrels}
  Let $(\chi_{\alpha,i})$, $(\phi_{\alpha,i})$ be a choice of a
  compatible pair of Pieri systems, and let
  $\gamma_{\alpha,i,j}^\pm,\; \delta_{\alpha,i,j}^\pm$ be the
  characteristic ratios for $(\chi_{\alpha,i})$, $(\phi_{\alpha,i})$
  respectively.  Let $\alpha,\; \gamma \in B_{l,m-l}$.  The relations
  on the truncated Young quiver between the vertices labeled $\alpha$
  and $\gamma$ are the kernels of the following linear maps.
  \begin{enumerate}[\quad(i)]
  \item If $\gamma$ is obtained by adding $2$ boxes to $\alpha$ in
    rows $i<j$, the map $(F^\svee \otimes F^\svee)^{\oplus2} \to F^\svee
    \otimes F^\svee$ defined by 
    \begin{align*}
    &(\lambda_1\otimes \lambda_2, \lambda_1'\otimes \lambda_2')\\
    &\mapsto
    \lambda_1\otimes \lambda_2 + \frac12 \left[
      \left(\gamma_{\alpha,i,j}^+ + \gamma_{\alpha,i,j}^-\right)
      \lambda_1'\otimes \lambda_2'
      + \left(\gamma_{\alpha,i,j}^+ - \gamma_{\alpha,i,j}^-\right)
      \lambda_2' \otimes \lambda_1'\right]\,.
    \end{align*}
  \item If $\gamma$ is obtained by removing two boxes from $\alpha$
      in rows $i<j$, the map $(G\otimes G )^{\oplus2} \to G \otimes G$
      defined by 
      \begin{align*}
      &(g_1\otimes g_2, g_1'\otimes g_2')\\
      &\mapsto 
      g_1\otimes g_2 +  \frac12 \left[
      \left(\delta_{\alpha,i,j}^+ + \delta_{\alpha,i,j}^-\right)
      g_1'\otimes g_2'
      + \left(\delta_{\alpha,i,j}^+ - \delta_{\alpha,i,j}^-\right)
      g_2' \otimes g_1'\right]\\
      &=
      g_1\otimes g_2 -  \frac12 \left[
      \left(\gamma_{\alpha,i,j}^+ + \gamma_{\alpha,i,j}^-\right)
      g_1'\otimes g_2'
      + \left(\gamma_{\alpha,i,j}^+ - \gamma_{\alpha,i,j}^-\right)
      g_2' \otimes g_1'\right]\,.
    \end{align*}
  \item If $\gamma$ is obtained by moving a box in $\alpha$ from row
    $i$ to row $j>i$, the map $(F^\svee \otimes G)^{\oplus2} \to F^\svee
    \otimes G$ defined by 
    \begin{align*}
      (\lambda\otimes g, \lambda'\otimes g') 
      &\mapsto 
      \lambda \otimes g + m_{\alpha,i,j} \lambda'\otimes g'\\
      &=
      \lambda \otimes g + \frac{2}{\gamma_{\alpha,i,j}^+ -
        \gamma_{\alpha,i,j}^-}
      \lambda'\otimes g'\,.
    \end{align*}
  \item If $\gamma=\alpha$, the map $(F^\svee \otimes G)^{\oplus(t(\alpha)+r(\alpha)-1)}
    \to (F^\svee \otimes G)^{\oplus(r(\alpha))}$ defined by
    \begin{align*}
      \left(\left(\lambda_i\otimes g_i\right)_{i \in \Delta_\alpha'}, 
        \left(\lambda_j' \otimes g_j'\right)_{j \in
          \nabla_\alpha'}\right)
      &\mapsto 
        \left(
        \lambda_i \otimes g_i + \sum_{j \in \nabla_\alpha'}
        c_{\alpha,i,j}\lambda_j' \otimes g_j'\right)_{i\in
        \Delta_\alpha'}
    \\
    &=
        \left(
        \lambda_i \otimes g_i + \sum_{j \in \nabla_\alpha'}
        \frac{\gamma_{\alpha-\epsilon_j,i,j}^+ +
          \gamma_{\alpha-\epsilon_j,i,j}^-}
        {\gamma_{\alpha-\epsilon_j,i,j}^+ -
          \gamma_{\alpha-\epsilon_j,i,j}^-}
          \lambda_j' \otimes g_j'\right)_{i\in
        \Delta_\alpha'}
    \,,
    \end{align*}
    where $\Delta_\alpha'$ is the set of indices $i$ such that
    $\alpha+\epsilon_i\in B_{l,m-l}$, $\nabla_\alpha'$ is similarly the
    set of indices $j$ with $\alpha-\epsilon_j \in B_{l,m-l}$,
    $t(\alpha)$ is the number of ways to add a box to $\alpha$ without
    making any row longer than $m-l$, and $r(\alpha)$ is the total
    number of ways to add a box to $\alpha$.\qed
  \end{enumerate}
\end{theorem}

\section{The case of $4\times 4$ matrices of rank 2}
\label{sect:G24}

Let us compute the quiver and some of the relations for the first
non-trivial example, $(m,n,l) = (4,4,2)$.  As a matter of notational
convenience we denote the vertices $\caln_\alpha = {p'}^*L^{\alpha}
\calq$ of the quiver by the corresponding Young diagrams.  We live
inside the box $B_{2,2}$, and therefore have the quiver below.
\begin{center}
\includegraphics{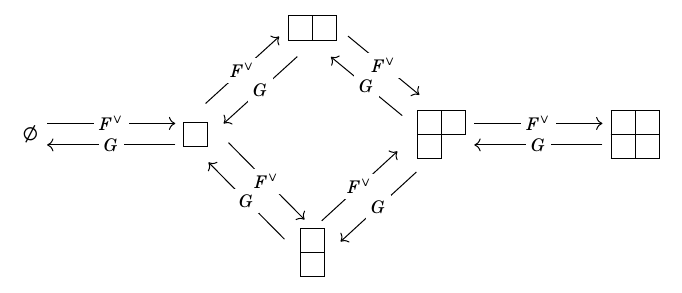}
\end{center}
In this picture, each arrow $F^\svee \colon \alpha \to
\alpha+\epsilon_i$ represents $\LL_{1\alpha}^{\alpha+\epsilon_i}
\otimes F^\svee$, while each $G\colon \alpha+\epsilon_i \to \alpha$
represents $\LL_{1^*\alpha+\epsilon_i}^\alpha\otimes G$.  The action
of the linear maps on the bundles $\caln_\alpha$ is via
the natural maps~\eqref{eq:action}. 

Even more
explicitly, if we fix bases $\{\lambda_1, \dots, \lambda_4\}$ and
$\{g_1, \dots, g_4\}$ for $F^\svee$ and $G$, then each such arrow
stands for four arrows labeled by $\phi_{\alpha+\epsilon_i,i}\otimes
\lambda_k$, respectively $\chi_{\alpha,i} \otimes g_k$, where
$(\chi_{\alpha,i})$ and $(\phi_{\alpha,i})$ is a chosen pair of
compatible Pieri systems.

Let us write down a particular compatible pair of Pieri systems. In
fact it is just as easy to write down a pair of Pieri systems for all
partitions $\alpha=(p,q)$ with at most two rows.  The
corresponding Schur functor $L^{(p,q)}V$ is a quotient of $\left(\Wedge^2
  V\right)^{\otimes q} \otimes \Sym_{p-q}V$, modulo certain
exchange-type relations.  For example, in the case of $\Yaab$, we have 
\[
u \wedge v \otimes w + v \wedge w \otimes u + w \wedge u \otimes v =
0\,.
\]

We denote a general element of $L^{(p,q)}V$ by
\[
\prod_{k=1}^q (u_k \wedge v_k) \otimes \x\,,
\]
where $\x = x_1 \cdots x_{p-q} \in \Sym_{p-q}V$.  Further denote by
$\x_\ihat$ the product $x_1\cdots \hat{x_i}\cdots x_{p-q}$ with $x_i$
deleted. 

Define $ \chi_{(p,q),1}\colon  L^{(p+1,q)}V \to V \otimes L^{(p,q)}V$ by
\begin{multline*}
  \prod_{k=1}^q (u_k \wedge v_k) \otimes \x 
  \mapsto 
  {\sum_{i=1}^{p-q+1} x_i \otimes \prod_{k=1}^q (u_k \wedge
  v_k) \otimes \x_\ihat} \\
  + \frac{1}{p-q+2}\sum_{j=1}^q \left(u_j \otimes x_i \wedge v_j \otimes
  \prod_{k\neq j} (u_k \wedge v_k) \otimes \x_\ihat \right. \\
  + \left. v_j \otimes u_j
  \wedge x_i \otimes \prod_{k\neq j} (u_k \wedge v_k) \otimes
  \x_\ihat \right)  
\end{multline*}
and $\chi_{(p,q),2} \colon  L^{(p,q+1)}V  \to V \otimes L^{(p,q)}V$ by
\begin{equation*}
  \prod_{k=1}^{q+1} (u_k \wedge v_k) \otimes \x 
  \mapsto 
  \sum_{i=1}^{q+1} \left(u_i \otimes \prod_{k\neq i} (u_k \wedge
  v_k) \otimes v_i \x 
  - v_i \otimes \prod_{k\neq i} (u_k \wedge v_k)
  \otimes u_i \x\right)\,.
\end{equation*}

We also define the dual Pieri maps $\phi_{(p+1,q),1}\colon V\otimes
L^{(p,q)}V \to L^{(p+1,q)}$ by
\[
w\otimes\prod_{k=1}^q (u_k \wedge v_k) \otimes \x
\mapsto
\frac{(p-q+2)}{(p+2)(p-q+1)}\prod_{k=1}^q (u_k \wedge v_k)
\otimes  w\x
\]
and $\phi_{(p,q+1),2} \colon  V\otimes L^{(p,q)}V  \to L^{(p,q+1)}V$ by
\[
w\otimes\prod_{k=1}^q (u_k \wedge v_k) \otimes \x 
\mapsto \frac{1}{(q+1)(p-q+1)}\sum_{i=1}^{p-q} w \wedge x_i \otimes
\prod_{k=1}^q (u_k \wedge v_k) \otimes \x_\ihat\,.
\]

It is a soothing combinatorial exercise to prove that each of these
maps is well-defined and that $\phi_{(p,q)+\epsilon_i,i}$ is a left inverse
for $\chi_{\alpha,i}$.

We point out that these are essentially the classical Pieri systems of
Olver, as we shall confirm below (at least up to equivalence) by
computing the characteristic ratios.

For the six partitions of interest, the formulas simplify:

\noindent\begin{tabular}{>{$}r<{$}>{$}l<{$}>{$}l<{$}}
  \chi_{\emptyset,1}\colon &\Ya \to V \otimes  \emptyset\,,
  &  u \mapsto u \otimes 1 \\
  \chi_{\Ya,1} \colon & \Yaa \to V \otimes    \Ya\,,
  & uv \mapsto u \otimes v + v \otimes u\\
  \chi_{\Ya,2}\colon &\Yab \to  V \otimes   \Ya\,, 
  &  u \wedge v \mapsto 
  u \otimes v - v \otimes u\\
  \chi_{\Yab,1}\colon & \Yaab \to  V \otimes   \Yab\,, 
  &  u \wedge v \otimes w 
  \mapsto w \otimes u \wedge v + \frac12 \left(u \otimes w \wedge v +
    v \otimes u \wedge w\right)\\
  \chi_{\Yaa,2}\colon &\Yaab \to  V \otimes   \Yaa\,, 
  &  u \wedge v \otimes w \mapsto 
   u \otimes vw - v \otimes uw\\
  \chi_{\Yaab,2}\colon &\Yaabb \to  V \otimes  \Yaab\,, 
  &  t\wedge u \otimes v \wedge w 
  \mapsto  t \otimes v \wedge w \otimes u - u \otimes v \wedge w
  \otimes t  \\
  & & \hfill +
  v \otimes t \wedge u \otimes w - w \otimes t \wedge u
  \otimes v
\end{tabular}

\noindent\begin{tabular}{>{$}r<{$}>{$}l<{$}>{$}l<{$}}
  \phi_{\Ya,1} \colon &V \otimes\emptyset \to  \Ya\,,
  & u \otimes 1 \mapsto
  u\\
  \phi_{\Yaa,1}\colon &V \otimes  \Ya \to  \Yaa\,, 
  & u \otimes v \mapsto \frac12 uv\\
  \phi_{\Yab,2}\colon &V \otimes \Ya \to  \Yab\,, 
  & u \otimes v \mapsto \frac 12 u \wedge v \\
  \phi_{\Yaab,1}\colon &V \otimes \Yab \to  \Yaab\,, 
  & u \otimes v \wedge w \mapsto \frac 23 v \wedge w \otimes u \\
  \phi_{\Yaab,2} \colon &V \otimes \Yaa \to \Yaab\,,
  & u \otimes v w \mapsto \frac13 \left(u \wedge v \otimes w + u
    \wedge w \otimes v\right)\\
  \phi_{\Yaabb,2} \colon &V \otimes \Yaab \to \Yaabb\,,
  & u \otimes v \wedge w \otimes t \mapsto 
  \frac 14 u \wedge t \otimes v \wedge w
\end{tabular}

\medskip

Let us verify the relations across the central diamond.
\[
\begin{tikzcd}[column sep = small]
{} &  V \otimes \Yaa \arrow{dl}[swap]{1\otimes\chi_{\Ya,1}}\\
V \otimes V \otimes \Ya   & && \Yaab \arrow{ull}[swap]{
  \chi_{\Yaa,2}} \arrow{dll}{\chi_{\Yab,1}} \\
{}  & V \otimes \Yab \arrow{ul}{1 \otimes \chi_{\Ya,2}}
\end{tikzcd}
\]
One computes the characteristic ratios
\[
\gamma_{\Ya,1,2}^+ = \frac{\chi_{\Ya,2,1}^+}{\chi_{\Ya,1,2}^+} = \frac 32
\qquad\text{and}\qquad
\gamma_{\Ya,1,2}^- = \frac{\chi_{\Ya,2,1}^-}{\chi_{\Ya,1,2}^-} = -\frac12\,.
\]
Observe that
\[
\gamma_{\Ya,1,2}^+ = 1-u
\qquad\text{and}\qquad
\gamma_{\Ya,1,2}^- = -1-u\,,
\]
where 
\[
u = \frac{1}{(1-p-1)-(2-q-1)} = 
\frac{-1}{p-q+1} = -\frac12\,,
\]
are the characteristic ratios of Olver's classical Pieri system, cf.\
Lemma~\ref{ref-5.4-14} and Remark~\ref{rem:dontneed}.  One checks
laboriously that the same holds true for all the $(\chi_{(p,q),i})$
defined above.  In particular we verify
\[
\frac{\gamma_{\Ya,1,2}^+}{\gamma_{\Ya,1,2}^-} = \frac{3/2}{-1/2} = - 3
= \frac{u-1}{u+1}\,.
\]


The relation in the reverse direction across the central diamond is
also easy to compute.
\[
\begin{tikzcd}[column sep = small]
{} & V \otimes \Yaa \arrow{drr}{\phi_{\Yaab,2}}\\
V \otimes V \otimes \Ya  \arrow{ur}{1 \otimes \phi_{\Yaa,1}}
\arrow{dr}[swap]{1\otimes \phi_{\Yab,2}} & & &  \Yaab  \\
{} & V \otimes \Yab \arrow{urr}[swap]{\phi_{\Yaab,1}}
\end{tikzcd}
\]
One finds
\[
\delta_{\Ya,1,2}^+ = \frac{\phi_{\Ya,2,1}^+}{\phi_{\Ya,1,2}^+} = 
\frac 12
\]
and 
\[
\delta_{\Ya,1,2}^- = \frac{\phi_{\Ya,2,1}^-}{\phi_{\Ya,1,2}^-} = 
-\frac 32
\]
in accordance with Proposition~\ref{prop:phi-props}.


We can also compute the relation corresponding to moving a box
downward in $\Yaa$ to obtain $\Yab$, finding
\[
m_{\Yaa,1,2} = 
\frac{(1\otimes
  \phi_{\Yab,2})(\tau\otimes1)(1\otimes\chi_{\Ya,1})}
{\chi_{\Yab,1}\phi_{\Yaab,2}}
=
\frac{1/2}{1/2} = 1\,.
\]
Of course this matches Proposition~\ref{prop:mixed-i-neq-j} and
Remark~\ref{rem:classical}:
\[
m_{\Yaa,1,2} = \frac{-2}{\gamma_{\Ya,1,2}^- - \gamma_{\Ya,1,2}^+}
= \frac{-2}{-1/2 - 3/2} = 1\,.
\]


Finally, in order to compute the relation at a single vertex, say
$\alpha = \Ya$, we write all of the $2$-cycles leaving $\alpha$ via
$\LL_{1\alpha}^{11\alpha-\epsilon_j}\otimes G$ (removing a box) in
terms of the basis of $\LL_{1\alpha}^{1\alpha}$ given by those cycles leaving
via $\LL_{1\alpha}^{\alpha+\epsilon_i}$
 (adding a box).
We have $\Delta_\alpha = \{1,2\}$ and
\[
\begin{tikzcd}[row sep = scriptsize]
  {} &  \Yaa \arrow{dl}{\chi_{\Ya,1}}\\
  V \otimes \Ya \arrow[yshift=1ex]{ur}{\phi_{\Yaa,1}}
  \arrow{dr}[yshift=-0.75ex]{\phi_{\Yab,2}} \\
  {} & \Yab  \arrow[yshift=-1ex]{ul}{\chi_{\Ya,2}}
\end{tikzcd}
\]
\begin{align*}
  \chi_{\Ya,1} \phi_{\Yaa,1} \colon \quad 
  & u \otimes v \mapsto \frac12 uv \mapsto \frac12\left(u\otimes v + v \otimes u\right)\\
  \chi_{\Ya,2}  \phi_{\Yab,2} \colon \quad 
  & u \otimes v \mapsto \frac 12 u \wedge v \mapsto \frac 12 \left(u \otimes v
    - v \otimes u\right)
\end{align*}
In the other direction, we have 
\[
(1\otimes \phi_{\Ya,1})(\tau\otimes1)(1\otimes
\chi_{\emptyset,1})\colon \quad u\otimes v \mapsto u \otimes v \otimes
1 \mapsto v \otimes u \otimes 1 \mapsto v \otimes u\,.
\]
Thus 
\[
(1\otimes \phi_{\Ya,1})(\tau\otimes1)(1\otimes
\chi_{\emptyset,1})
=
 \left(\chi_{\Ya,1} \phi_{\Yaa,1}\right) 
- \left(\chi_{\Ya,2} \circ \phi_{\Yab,2}\right)
\]
and
\[
c_{\Ya,1,1} = 1
\quad\text{while}\quad
c_{\Ya,2,1} = -1\,.
\]
This is a somewhat trivial example, coming down to
$\gamma_{\emptyset,1,1}^+=1$, $\gamma_{\emptyset,1,1}^-=0$, 
 $\gamma_{\emptyset,1,2}^+=0$, and $\gamma_{\emptyset,1,2}^-=1$.

\bigskip

The action of the quiver on the bundles $\caln_\alpha$ is defined in
terms of the adjoints $\chi_{\alpha,i}^\#\colon V^\svee \otimes
L^{\alpha+\epsilon_i} V \to L^\alpha V$ of the Pieri maps
$\chi_{\alpha,i}\colon L^{\alpha+\epsilon_i}V \to V \otimes L^\alpha
V$ defined above.  We denote the trace pairing $\Tr\colon V^\svee
\otimes V \to K$ by $\lambda \otimes v \mapsto \lambda(v)$.

\noindent\begin{tabular}{>{$}r<{$}>{$}l<{$}>{$}l<{$}}
  \chi_{\emptyset,1}^\#\colon & V^\svee \otimes \Ya \to \emptyset\,,
  &  \lambda \otimes u \mapsto \lambda(u)  \\
  \chi_{\Ya,1}^\# \colon &  V^\svee \otimes \Yaa \to  \Ya\,,
  & \lambda \otimes uv \mapsto \lambda(u) v + \lambda(v) u\\
  \chi_{\Ya,2}^\# \colon & V^\svee \otimes \Yab \to  \Ya\,, 
  &  \lambda \otimes u \wedge v \mapsto 
  \lambda(u) v - \lambda(v) u\\
  \chi_{\Yab,1}^\# \colon &  V^\svee \otimes \Yaab \to  \Yab\,, 
  & \lambda\otimes  u \wedge v \otimes w 
  \mapsto \lambda(w) u \wedge v \\
  && \hfill+ \frac12 \left(\lambda(u) w \wedge v +
    \lambda(v) u \wedge w\right)\\
  \chi_{\Yaa,2}^\# \colon & V^\svee \otimes \Yaab \to  \Yaa\,, 
  & \lambda \otimes u \wedge v \otimes w \mapsto 
  \lambda(u) vw - \lambda(v) uw\\
  \chi_{\Yaab,2}^\# \colon & V^\svee \otimes \Yaabb \to  \Yaab\,, 
  & \lambda \otimes  t\wedge u \otimes v \wedge w 
  \mapsto  \lambda(t) v \wedge w \otimes u - \lambda(u) v \wedge w
  \otimes t  \\
  & & \hfill +
  \lambda(v) t \wedge u \otimes w - \lambda(w) t \wedge u
  \otimes v
\end{tabular}
The characteristic ratios of these adjoint maps are equal to those of
the originals. 

Now the relations on the quiver are clear.  For instance, between
$\emptyset$ and $\Yaa$ we have $\Wedge^2 F^\svee=0$, i.e.\ 
\[
(\phi_{\Yaa,1}\otimes \lambda_k)(\phi_{\Ya,1}\otimes \lambda_l)
-
(\phi_{\Yaa,1}\otimes \lambda_l)(\phi_{\Ya,1}\otimes \lambda_k)
=0
\]
for all $k,l= 1,\dots, 4$, or more compactly
$\lambda_k\lambda_l=\lambda_l\lambda_k$.

Across the central diamond, we have relations defined by the kernel of
\begin{align*}
  &(\lambda_r\otimes \lambda_s, \lambda_t\otimes \lambda_u)\\
  &\mapsto \lambda_r\otimes \lambda_s + \frac12 \left[
    \left(\gamma_{\alpha,i,j}^+ + \gamma_{\alpha,i,j}^-\right)
    \lambda_t\otimes \lambda_u + \left(\gamma_{\alpha,i,j}^+ -
      \gamma_{\alpha,i,j}^-\right)
    \lambda_u \otimes \lambda_t\right]\\
  &= \lambda_r\otimes \lambda_s + \frac12 \left[ \left(\frac32 -
      \frac12\right) \lambda_t\otimes \lambda_u + \left(\frac32
      +\frac12\right) \lambda_u \otimes \lambda_t\right]\\
  &= \lambda_r\otimes \lambda_s + \frac12 \lambda_t\otimes
  \lambda_u + \lambda_u \otimes \lambda_t\,.
\end{align*}
This kernel is of course isomorphic to $F^\svee \otimes F^\svee$.  Similarly,
from $\Yaab$ to $\Ya$ we have relations defined by the kernel of 
\[
(g_r\otimes g_s, g_t\otimes g_u) \mapsto g_r\otimes g_s - \frac12
g_t\otimes g_u -g_u\otimes g_t\,.
\]

Since $m_{\Yaa,1,2}=1$ (see Remark~\ref{rem:classical}), the vertical
relation across the central diamond is just the commutativity
relation.

Finally, at the vertex $\Ya$ we have relations defined by the kernel
of 
\begin{align*}
  &\left(\lambda_a\otimes g_b,\, \lambda_c\otimes g_d,\, \lambda_e\otimes
  g_f\right) \mapsto\\
  &\left(\lambda_a\otimes g_b + \lambda_e\otimes g_f,\, 
\lambda_c\otimes g_d - \lambda_e\otimes g_f,\, 
\lambda_a\otimes g_b + 3 \lambda_c\otimes g_d\right)\,. 
\end{align*}

\appendix

\section{The quiverized Clifford algebra}
\label{sect:adv-quiverization}

We offer here an alternative approach to the proof of
Theorem~\ref{thm:quiverintro}, which is conceptually closer to the
spirit of~\cite{Buchweitz-Leuschke-VandenBergh:2010}, but is a bit too
cumbersome for explicit examples due to the multiple identifications
involved.

\subsection*{Quiverization}
Let $\Gamma$ be a linearly reductive algebraic group over an arbitrary
field $K$ and let $\check{\Gamma}$ be the set of characters of
$\Gamma$.  If $\alpha \in \check{\Gamma}$ then we denote its
corresponding irreducible representation by $\SSS^\alpha$.  The
character belonging to the dual representation $(\SSS^\alpha)^\svee =
\Hom_\Gamma(\SSS^\alpha,K)$ will be denoted $\alpha^*$.  Write
$\emptyset$ for the character of the trivial representation.

Let $\mod(\Gamma)$ be the category of rational representations of
$\Gamma$, and let $\mod^\circ(\Gamma)$ be the category of collections
of vector spaces $V = (V_\alpha)_{\alpha \in \check{\Gamma}}$.  We have
functors
\[
Q^\circ \colon \mod(\Gamma) \to \mod^\circ(\Gamma)\,, \qquad V \mapsto
(\Hom_\Gamma(\SSS^\alpha,V))_{\alpha \in \check{\Gamma}}
\]
\[
R^\circ \colon \mod^\circ(\Gamma) \to \mod(\Gamma)\,, \qquad V \mapsto
\bigoplus_{\beta \in \check{\Gamma}} V_\beta \otimes \SSS^\beta\,.
\]

The following lemma just expresses the fact that $\mod(\Gamma)$ is a
semisimple category.

\begin{lemma}
  \label{lem:mod-modcirc}
  The functors $Q^\circ$ and $R^\circ$ define inverse equivalences of
  categories. \qed
\end{lemma}

Unfortunately it is not immediately obvious what $Q^\circ$ does to the
monoidal structure on $\mod(\Gamma)$.  Therefore we introduce another
monoidal category $\mod^1(\Gamma)$ which consists of collections of
vector spaces $\VV = (V_\beta^\alpha)_{\alpha,\beta\in
  \check{\Gamma}}$ with tensor product defined as in matrix
multiplication:
\[
(\VV \otimes \WW)_\gamma^\alpha = \bigoplus_{\beta \in \check{\Gamma}} V^\alpha_\beta
\otimes W^\beta_\gamma\,.
\]
Furthermore
$\mod^1(\Gamma)$ acts on $\mod^\circ(\Gamma)$ by
\[
(\VV \otimes W)_\alpha = \bigoplus_{\beta \in \check{\Gamma}} V^\beta_\alpha \otimes
W_\beta\,.
\]
\begin{lemma}\label{lem:quiv-functor-ff}
  There is a fully faithful monoidal functor
  \[
  Q \colon \mod(\Gamma) \to \mod^1(\Gamma)\,, \qquad 
  V \mapsto (\Hom_\Gamma(\SSS^\beta,\, \SSS^\alpha\otimes
  V))^\alpha_\beta
  \]
  which is also compatible with the left actions of $\mod(\Gamma)$ on
  itself and of $\mod^1(\Gamma)$ on $\mod^\circ(\Gamma)$.
\end{lemma}

\begin{proof}
  That $Q$ is fully faithful follows from the fact that it has a left
  inverse
  \[
  R \colon \mod^1(\Gamma) \to \mod(\Gamma)\,, \qquad \VV\mapsto
  \bigoplus_{\beta \in \check{\Gamma}} \VV^\emptyset_\beta \otimes
  \SSS^\beta\,.
  \]
  That $Q$ is compatible with tensor product is a straightforward
  verification.
\end{proof}

From this we easily obtain the following.

\begin{lemma}
  \label{lem:algebra-object}
  If $C$ is an algebra object in $\mod(\Gamma)$ then $Q(C)$ is an
  algebra object in $\mod^1(\Gamma)$, and if $C$ is given by
  generators and relations as a quotient of a tensor algebra, say,
  $C = TV/I$ for $\Gamma$-representations $V$ and $I$, then
  \[
  Q(C) = T(Q(V))/(Q(I))\,.
  \]
  Furthermore $Q^\circ$ defines an equivalence between the category
  $\mod_\Gamma(C)$ of left $\Gamma$-equivariant $C$-modules and the
  category $\mod^\circ(Q(C))$ of left $Q(C)$-modules in
  $\mod^\circ(\Gamma)$. \qed
\end{lemma}
Here we understand $T(Q(V))$ to be the tensor algebra defined in terms
of the natural monoidal structure on $\mod^1(\Gamma)$.

If $D$ is a subset of $\check{\Gamma}$ then we denote by $\mod_D(C)$
the $\Gamma$-equivariant $C$-modules whose characters lie in $D$.
Also write
\[
Q_D(C) = Q(C)\big/\left(e_\alpha\right)_{\alpha \notin D}
\]
for the quotient of $Q(C)$ by the idempotents $e_\alpha$ corresponding
to characters $\alpha$ not in $D$.
\begin{lemma}
  \label{lem:restrict-quiver}
  Let $C$ be an algebra object in $\mod(\Gamma)$.  The equivalence
  $Q^\circ\colon \mod_\Gamma(C) \to \mod^\circ(Q(C))$ restricts to an
  equivalence between $\mod_D(C)$ and $\mod(Q_D(C))$.\qed
\end{lemma}

We define the indicator spaces $\LL$ in this more general setting
analogously to Definition~\ref{def:LL}.  

\begin{definition}
  \label{def:gen-LL}
  Let $\alpha_1, \dots, \alpha_n,\beta \in \check \Gamma$, and set
  \[
  \LL_{\beta}^{\alpha_1\cdots\alpha_n} = \Hom_\Gamma(\SSS^\beta,\
  \SSS^{\alpha_1}\otimes \cdots \otimes \SSS^{\alpha_n})\,.
  \]
\end{definition}
Obvious analogs of the properties in Proposition~\ref{prop:LL-props}
hold in this setting.  

\begin{prop}
  \label{prop:LL-V-props}
  Let $V=(V_\alpha)_\alpha$ and $W = (W_\alpha)_\alpha \in \mod^\circ(\Gamma)$.  Then
  \[
  Q(R^\circ(V))^\beta_\gamma = Q\left(\bigoplus_\alpha V_\alpha
    \otimes \SSS^\alpha\right)^\beta_\gamma \cong \bigoplus_\alpha
  V_\alpha \otimes \LL_\gamma^{\alpha\beta}
  \]
  and
  \[
  Q(V\otimes W)_\gamma^\beta = \bigoplus_{\alpha_1, \alpha_2}
  V_{\alpha_1} \otimes W_{\alpha_2} \otimes
  \LL_{\gamma}^{\alpha_1\alpha_2\beta}\,.\qed
  \]
\end{prop}

The canonical isomorphism $Q(V\otimes W) \cong Q(V) \otimes Q(W)$ is
given by
\begin{align*}
\bigoplus_{\alpha_1,\alpha_2} V_{\alpha_1}\otimes W_{\alpha_2} \otimes
\LL_\gamma^{\alpha_1\alpha_2\beta} 
&\cong Q(V\otimes W)^\beta_\gamma\\
&\cong \bigoplus_\delta Q(V)^\beta_\delta \otimes Q(W)^\delta_\gamma\\
&\cong \bigoplus_{\delta,\alpha_1,\alpha_2} 
V_{\alpha_1}\otimes\LL^{\alpha_1\beta}_\delta \otimes W_{\alpha_2}\otimes\LL^{\alpha_2\delta}_\gamma
\end{align*}
combined with the isomorphism 
\[
\LL_\gamma^{\alpha_1\alpha_2\beta} 
\cong \bigoplus_{\delta} 
\LL^{\alpha_2\delta}_\gamma \otimes \LL^{\alpha_1\beta}_\delta
\]
from Proposition~\ref{prop:LL-props}(\ref{item:LL-props-iii}).

\subsection*{The Clifford algebra}
We want to use the quiverization recipe above
applied to the general linear group, so from now on we assume that
\emph{$K$ is a field of characteristic zero}.

We fix an arbitrary $(m-l)$-dimensional vector space $U$ and set
$\tilde F = F\otimes U^\svee$, $\tilde G = G \otimes U^\svee$.  There
is a natural pairing 
\[
\langle-,-\rangle \colon \tilde F^\svee \times \tilde G \to S
\]
which is just the inclusion $F^\svee \otimes G \to S$ combined with
the canonical pairing $U \otimes U^\svee \to K$.  We extend this
pairing to a symmetric bilinear form on $\left(\tilde F^\svee \oplus
  \tilde G\right) \times\left( \tilde F^\svee \oplus \tilde G\right)$
and thence to a quadratic form $b\colon \tilde F^\svee \oplus \tilde G
\to S$.

We let $C$ be the associated Clifford algebra of $b$ over $S$.  For a
concrete description, choose ordered bases $\{\lambda_1, \dots,
\lambda_m\}$, $\{g_1, \dots, g_n\}$, and $\{u_1, \dots, u_{m-l}\}$ for
$F^\svee$, $G$, and $U$, respectively, and let $\left\{u_1^*, \dots,
  u_{m-l}^*\right\}$ denote the dual basis for $U^\svee$.  Then $C$ is
the $S$-algebra generated by $\{\lambda_i\otimes u_a\}_{i,a}$ and
$\{g_j\otimes u_b^*\}_{j,b}$ subject to the relations
\begin{multline*}
  (\lambda_i \otimes u_a)(\lambda_j\otimes u_b) + (\lambda_j \otimes
  u_b)( \lambda_i \otimes u_a) = 0 = (\lambda_i \otimes u_a)^2 
\\
\shoveright{\text{for $i,j = 1, \dots, m$;}}\\
\shoveleft{(g_i\otimes u_a^*)( g_j \otimes u_b^*)+ (g_j \otimes u_b^*)( g_i
\otimes u_a^*) = 0 = (g_i\otimes u_a^*)^2}
\\
\shoveright{\text{for $i,j = 1, \dots, n$; and}}\\
\shoveleft{(\lambda_i \otimes u_a)(g_j\otimes u_b^*) + (g_j\otimes u_b^*)(
\lambda_i \otimes u_a) = \delta_{ab} x_{ij}}
\\
\text{for $i=1, \dots, m$, $j=1, \dots, n$}
\end{multline*}
for all $a,b = 1, \dots, m-l$.

Recall that $B_{l,m-l}$ denotes the set of partitions having at most
$l$ rows and at most $m-l$ columns, which we now think of as
representing characters for $\GL(U) \cong \GL(m-l)$ via the
identification $\alpha \leftrightarrow L^{\alpha'}U$ (note the
transpose!), where $L^{\alpha'}$ is the Schur functor for the weight
$\alpha'$.

\begin{definition}
  \label{def:QC}
  The \emph{quiverized Clifford algebra} is 
  \[
  Q_{B_{l,m-l}}(C) =
  Q(C)/(e_\alpha)_{\alpha \notin B_{l,m-l}}\,,
  \]
  where $e_\alpha$ denotes the idempotent corresponding to $\alpha$.
\end{definition}

To show that the quiverized Clifford algebra is isomorphic to the
non-commutative desingularization, we define a left action on the
tilting bundle $\caln = \bigoplus_{\alpha\in B_{l,m-l}}{p'}^*L^\alpha
\calq$.

\begin{prop}
  \label{prop:map-app}
  There is a ring homomorphism $\Theta\colon Q_{B_{l,m-l}}(C) \to
  A=\End_\caloz(\caln)$.
\end{prop}

\begin{proof}
  Pulling back the tautological quotient map $\pi^* F^\svee \to \calq$
  from $\GG$ to $\calz$ and tensoring with $U$ we obtain a map
  \[
  \Phi^U\colon {q'}^* (F\otimes S)^\svee \otimes U \to {p'}^*\calq\otimes
  U\,.
  \]
  Similarly the fact that $\calz =
  \underline{\Spec}(\Sym_{\calog}(\calq \otimes G))$ yields a
  tautological map ${p'}^* \calq \otimes {q'}^*(G \otimes S) \to \caloz$ which
  we transform into a map
  \[
  \Psi^U \colon {q'}^* (G\otimes S) \otimes U^\svee \to {p'}^*\calq^\svee
  \otimes U^\svee\,.
  \]
  Now $\tilde F^\svee$ maps to the global sections of ${q'}^* (F\otimes S)^\svee
  \otimes U$ and similarly $\tilde G$ maps to the global sections of
  ${q'}^* (G\otimes S) \otimes U^\svee$. Thus $\tilde F^\svee$ acts via the
  map $\Phi^U$ on $\Wedge_{\caloz}({p'}^*\calq \otimes U)$ by left
  exterior multiplication, and $\tilde G$ acts via the map $\Psi^U$ by
  contraction.  It is easy to see that these two actions satisfy the
  Clifford relations.

  Thus $C$ acts on $\Wedge_{\caloz}({p'}^*\calq \otimes U)$ and hence
  $Q(C)$ acts on $Q^\circ(\Wedge_{\caloz}({p'}^*\calq \otimes U))$.
  By the Cauchy formula we have
  \[
  \Wedge_\caloz({p'}^*\calq \otimes U) = \bigoplus_{\alpha\in
    B_{l,m-l}} L^\alpha \calq \otimes L^{\alpha'}U =
  \bigoplus_{\alpha\in B_{l,m-l}} \caln_\alpha \otimes
  L^{\alpha'}U\,,
  \]
  and hence
  \[
  Q^\circ (\Wedge_\caloz({p'}^*\calq \otimes U)) = \bigoplus_{\alpha
    \in B_{l,m-l}} \caln_\alpha = \caln\,.
  \]
  Thus $Q(C)$ acts on $\caln$ and in fact $Q_{B_{l,m-l}}(C)$ acts
  since $\Wedge_\caloz({p'}^*\calq \otimes U)$ contains only
  representations $L^\alpha \calq$ with weight in $B_{l,m-l}$.
\end{proof}

To prove that $\Theta$ is an isomorphism, we must understand
$Q_{B_{l,m-l}}(C)$ more concretely.  The presentation of $C$ over $S$
yields a presentation of $Q(C)$ by Lemma~\ref{lem:algebra-object}, and
hence of $Q_{B_{l,m-l}}(C)$.  The generators are easily identified.
\begin{prop}
  \label{prop:gens-QC}
  The quiver for $Q(C)$ has vertices indexed by the transposes
  $\alpha'$ of partitions corresponding to representations $L^\alpha
  U$, and has arrows $\alpha' \to \beta'$ indexed by (a basis of) 
  \[
  \begin{cases}
    F^\svee & \text{if $\alpha\nearrow \beta$, and}\\
    G & \text{if $\beta \nearrow  \alpha$.}
  \end{cases}
  \]
\end{prop}

\begin{proof}
  The Clifford algebra $C$ is generated by $\tilde F^\svee = F^\svee
  \otimes U$ and $\tilde G = G\otimes U^\svee$. We therefore compute
  the generators of $Q(C)$ as
  \[
  Q(F^\svee \otimes U)_{\beta'}^{\alpha'} =\Hom_{\GL(U)}(L^{\beta} U,\
  L^{\alpha} U \otimes F^\svee \otimes U) = F^\svee \otimes
  \LL_{\beta}^{1\alpha}
  \]
  and
  \[
  Q(G \otimes U^\svee)_{\beta'}^{\alpha'} =\Hom_{\GL(U)}(L^{\beta} U,\
  L^{\alpha} U \otimes G \otimes U^\svee) = G \otimes
  \LL_{\beta}^{1^*\alpha}
  \]
  for two partitions $\alpha', \beta'$, where the transposes arise
  because of our identification $\alpha \leftrightarrow L^{\alpha'}U$.
  These are the natural generators.  To have them solely in terms of
  $F^\svee$ and $G$, one can choose basis elements for the
  one-dimensional spaces $\LL^{1\alpha}_\beta$ and
  $\LL^{1^*\alpha}_\beta$.
\end{proof}

The presentation of $C$ over $S$ can be translated into a presentation
over the ground field $K$.  In the case of maximal minors we
saw~\cite[Remark 7.6]{Buchweitz-Leuschke-VandenBergh:2010} that this
presentation involves cubic relations of the form $\lambda_k(\lambda_i
g_j + g_j\lambda_i) = (\lambda_i
g_j + g_j\lambda_i)\lambda_k$ and $g_k(\lambda_i
g_j + g_j\lambda_i)=(\lambda_i
g_j + g_j\lambda_i)g_k$ expressing the fact that the polynomial ring
$S$ lies in the center of the algebra.  We observe that this
phenomenon disappears for smaller minors.

\begin{prop}
  \label{prop:Y-quadratic}
  If $m-l>1$, then the Clifford algebra $C$ is defined by quadratic
  relations over $K$, whence $Y$ is quadratic as well.
\end{prop}

\begin{proof}
  We have to show that the generators $x_{ij} = \lambda_i \otimes g_j$
  of the polynomial ring are central in $C$, using only the quadratic
  relations. To show that this element commutes with the generators
  $\lambda_k\otimes u_a$ and $g_k \otimes u_a^*$, fix $k$ and $a$ and
  observe that $\lambda_k\otimes u_a$ and $g_k \otimes u_a^*$ each
  anticommute with any $\lambda_i \otimes u_b$ and $g_j\otimes u_b^*$
  for any $b \neq a$.  Since $m-l>1$ we may choose $b \neq a$, and then
  \[
  \lambda_i \otimes g_j = (\lambda_i \otimes u_b)(g_j\otimes u_b^*) + 
  (g_j\otimes u_b^*)(\lambda_i \otimes u_b)
  \]
  commutes with $\lambda_k\otimes u_a$ and $g_k \otimes u_a^*$.  The
  consequence that $Y$ is quadratic follows from
  Lemma~\ref{lem:algebra-object}.
\end{proof}

We can obtain the relations in $Q(C)$ by quiverization as well, giving
an alternative to Lemma~\ref{lem:surjrels}.

\begin{proposition}
  \label{prop:QC-rels}  
  Assume $m-l>1$. The spaces of relations in $Q(C)$ between two
  vertices $\alpha'$ and $\gamma'$ are given below.
  \begin{equation*}
  \begin{cases}
    \Sym_2 F^\svee & \text{if }\gamma \nearrow \nearrow \alpha, \text{ two boxes in a column}\\
    \Wedge^2 F^\svee & \text{if }\gamma \nearrow \nearrow \alpha, \text{ two boxes in a row} \\
    \Sym_2 F^\svee \oplus \Wedge^2 F^\svee \cong F^\svee \otimes
      F^\svee & \text{if }\gamma \nearrow \nearrow \alpha, \text{ two  disconnected boxes} \\
    F^\svee\otimes G & \text{if }\alpha \neq \gamma \text{, and } \alpha \nearrow
      \beta,\ \gamma \nearrow \beta, \text{ some } \beta\\
    (F^\svee\otimes G)^{\oplus(r(\alpha)-1)} & \text{if }\alpha =\gamma\\
    \Sym_2 G & \text{if }\alpha \nearrow \nearrow \gamma, \text{ two
      boxes in a column} \\
    \Wedge^2 G & \text{if }\alpha \nearrow \nearrow \gamma, \text{ two
      boxes in a row}\\
    \Sym_2 G \oplus \Wedge^2 G \cong G \otimes G & \text{if }\alpha \nearrow \nearrow \gamma, \text{
      two disconnected boxes.}
  \end{cases}
  \end{equation*}
  Here $r(\alpha)$ denotes the number of rows in which a box can be
  added to $\alpha$ to obtain a partition.
\end{proposition}

Note that as in Definition~\ref{def:Yrels}, the embedding in each case
is not the obvious diagonal one, but relies on the canonical
decompositions~\ref{eq:decomps}. 

We prove the proposition by considering in turn the quiverizations of the
three kinds of relations on $C$. These are defined by subspaces of the
degree-two part of the tensor algebra $T_S((\tilde F^\svee \oplus
\tilde G)\otimes S)$, which decomposes
\[
(\tilde F^\svee \oplus \tilde G) \otimes (\tilde F^\svee \oplus \tilde
G) 
=
(\tilde F^\svee \otimes \tilde F^\svee) \oplus
(\tilde G \otimes  \tilde G) \oplus
(\tilde F^\svee \otimes \tilde G) \oplus 
(\tilde F^\svee \otimes \tilde G)\,.
\]

\begin{nsit}
  {Relations coming from $\tilde F^\svee$.}\label{sit:F-rels}
  In $C$ the elements of $\tilde F^\svee$ anticommute; equivalently,
  the relations defining $C$ include the representation
  $\Sym_2(F^\svee \otimes U)$.  Now
  \[
  \Sym_2(F^\svee \otimes U) = \left(\Sym_2 F^\svee \otimes \Sym_2
    U\right) \oplus \left(
  \Wedge^2 F^\svee \otimes \Wedge^2 U\right)
  \]
  naturally (for definiteness we take the splitting $\Sym_2 F^\svee
  \to F^\svee \otimes F^\svee$ sending $\lambda \mu$ to $\frac12
  (\lambda\otimes \mu + \mu \otimes \lambda)$).  So in fact we have two
  types of relations $\Sym_2 F^\svee \otimes \Sym_2 U$ and $\Wedge^2
  F^\svee \otimes \Wedge^2 U$.  We discuss these individually.

  For the first case we need to describe the map $Q(\Sym_2 F^\svee
  \otimes \Sym_2 U) \to Q(\Sym_2 F^\svee) \otimes Q(\Sym_2 U)$.
  Specializing to two vertices $\alpha',\ \gamma'$, we need to describe the
  induced map
  \begin{align*}
    \Sym_2 F^\svee \otimes \LL_\gamma^{[2]\alpha} 
    &= Q(\Sym_2 F^\svee \otimes \Sym_2 U)^{\alpha'}_{\gamma'}  \\
    & \to \bigoplus_{\beta'} Q(F^\svee \otimes U)^{\beta'}_{\gamma'}
    \otimes Q(F^\svee \otimes U)^{\gamma'}_{\alpha'}\\
    &= \bigoplus_{\beta'} F^\svee \otimes \LL_{\gamma}^{1\beta} \otimes
    F^\svee \otimes \LL_\beta^{1\alpha} \\
    &= F^\svee \otimes F^\svee
    \otimes \LL^{11\alpha}_\gamma\,.
  \end{align*}
  The map on the $F^\svee$ factors is the natural one $\Sym_2 F^\svee
  \to F ^\svee \otimes F^\svee$, as we have not really touched
  $F^\svee$.  The inclusion map $\LL_\gamma^{[2]\alpha} \to
  \LL_\gamma^{11\alpha}$ is obtained from the canonical decomposition
  \[
  \LL_\gamma^{11\alpha} = \left(\LL_{[2]}^{11} \otimes \LL_\gamma
    ^{[2]\alpha} \right) \oplus \left(\LL_{[11]}^{11} \otimes
    \LL_\gamma^{[11]\alpha}\right)\,.
  \]

  There are three essentially different possibilities for
  $\alpha'$, $\gamma'$.
  \begin{enumerate}[\quad(i)]
  \item \emph{$\gamma'$ is obtained from $\alpha'$ by adding $2$ boxes
      to a row.} In this case there is a unique $\beta'$ such that
    $\alpha'\nearrow\beta'\nearrow\gamma'$.  By the
    Littlewood-Richardson rule we have $\LL^{[11]\alpha}_{\gamma}=0$
    and hence
    \[
    \LL^{11\alpha}_\gamma=\LL^{[2]\alpha}_\gamma= \LL^{1\beta}_\gamma\otimes \LL^{1\alpha}_\beta\,.
    \]
    The corresponding relations are given by 
    \[
    \Sym_2 F^{\svee}\otimes \LL^{1\beta}_\gamma\otimes \LL^{1\alpha}_\beta\hookrightarrow
    \left(F^\svee\otimes \LL^{1\beta}_\gamma\right) \otimes \left(F^\svee \otimes \LL^{1\alpha}_\beta\right)\,.
    \]
    Thus for $\alpha'\nearrow\beta'\nearrow\gamma'$ with the boxes being
    added in the same row the relations are the \emph{anti-commutation
      relations}.
  \item \emph{$\gamma'$ is obtained from $\alpha'$ by adding $2$ boxes to a
      column.}  In this case $\LL^{[2]\alpha}_\gamma=0$ and hence
    there are no such relations.
  \item \emph{$\gamma'$ is obtained from $\alpha'$ by adding $2$ boxes
      not in the same row or column.}
    In this case there are distinct $\beta'_1$, $\beta'_2$ such that
    $\alpha'\nearrow\beta'_1\nearrow\gamma'$,
    $\alpha'\nearrow\beta'_2\nearrow\gamma'$. The corresponding relations
    are now relations between paths going
    $\alpha'\to\beta_1'\to \gamma'$ and
    $\alpha'\to\beta_2'\to \gamma'$:
    \begin{equation*}
      \label{eq:sss1}
      \Sym_2 F^\svee\otimes \LL^{[2]\alpha}_{\gamma}
      \hookrightarrow 
      (F^\svee\otimes \LL^{1\beta_1}_\gamma) \otimes (F^\svee \otimes \LL^{1\alpha}_{\beta_1})\oplus
      (F^\svee\otimes \LL^{1\beta_2}_\gamma) \otimes (F^\svee \otimes \LL^{1\alpha}_{\beta_2})\,.
    \end{equation*}
  \end{enumerate}

  Now we describe the relations on $Q(C)$ derived from the inclusion
  \[
  \Wedge^2 F^\svee\otimes \Wedge^2 U \to (F^\svee\otimes U)\otimes
  (F^\svee\otimes U)\,. 
  \]
  Applying $Q(-)^{\alpha'}_{\gamma'}$ to both sides yields
  \begin{align*}
    \Wedge^2 F^\svee\otimes \LL^{[11]\alpha}_{\gamma}
    &= Q(\Wedge^2 F^\svee\otimes \Wedge^2 U)^{\alpha'}_{\gamma'} \\
    &\to \bigoplus_{\beta'} Q(F^\svee\otimes U)^{\beta'}_{\gamma'}\otimes
    Q(F^\svee\otimes U)_{\beta'}^{\alpha'}\\
    &= \bigoplus_{\beta'} F^\svee\otimes \LL^{1\beta}_\gamma \otimes
    F^\svee \otimes \LL^{1\alpha}_\beta \\
    &=F^\svee\otimes
    F^\svee\otimes\LL^{11\alpha}_\gamma\,.
  \end{align*}
  We discuss again the possible cases.
  \begin{enumerate}[\quad(i)]
  \item \emph{$\gamma'$ is obtained from $\alpha'$ by adding $2$ boxes to a
      row.}  In this case $\LL^{[11]\alpha}_{\gamma}=0$ and hence
    there are no such relations.
  \item \emph{$\gamma'$ is obtained from $\alpha'$ by adding $2$ boxes to a column.}
    In this case there is again a unique $\beta'$ such that
    $\alpha'\nearrow\beta'\nearrow\gamma'$. The corresponding relations are
    \[
    \Wedge^2 F^{\svee}\otimes \LL^{1\beta}_\gamma\otimes
    \LL^{1\alpha}_\beta
    \hookrightarrow
    (F^\svee\otimes \LL^{1\beta}_\gamma) \otimes (F^\svee \otimes \LL^{1\alpha}_\beta)\,.
    \]
    Thus for $\alpha'\nearrow\beta'\nearrow\gamma'$ with the boxes being added in
    the same column the relations are the \emph{commutation relations}.
  \item\emph{$\gamma'$ is obtained from $\alpha'$ by adding $2$ boxes not
      in the same row or column.}  In this case there are distinct
    $\beta'_1$, $\beta'_2$ such that $\alpha'\nearrow\beta'_1\nearrow\gamma'$,
    $\alpha'\nearrow\beta'_2\nearrow\gamma'$. The corresponding relations are now
    relations between paths going $\alpha'\to\beta'_1\to \gamma'$ and
    $\alpha'\to\beta'_2\to \gamma'$:
    \begin{equation*}
      \label{sss3}
      \Wedge^2 F^\svee\otimes \LL^{[11]\alpha}_{\gamma}
      \to (F^\svee\otimes \LL^{1\beta_1}_\gamma) \otimes (F^\svee \otimes \LL^{1\alpha}_{\beta_1})\oplus
      (F^\svee\otimes \LL^{1\beta_2}_\gamma) \otimes (F^\svee \otimes \LL^{1\alpha}_{\beta_2})\,.
    \end{equation*}
  \end{enumerate}
\end{nsit}

\begin{nsit}
  {Relations coming from $\tilde G$.}\label{sit:G-rels}
  Next we discuss the relations on $Q(C)$ coming from the inclusion
  \[
  \Sym_2 (G\otimes U^\svee) \subseteq (G \otimes U^\svee) \otimes
  (G\otimes U^\svee)\,.
  \]
  A discussion exactly parallel to the one above, using the identity
  $Q(G\otimes U^\svee)^{\beta'}_{\gamma'} = G \otimes
  \LL_{\gamma}^{1^*\beta}$, leads to the following cases.
  \begin{enumerate}[\quad(i)]
  \item \emph{$\gamma'$ is obtained from $\alpha'$ by deleting $2$ boxes
      from a row.} Here there is a unique $\beta'$ such that $\gamma'
    \nearrow \beta' \nearrow \alpha'$. We find
    $\LL^{[11]^*\alpha}_{\gamma}=0$ and hence
    $\LL^{[2]^*\alpha}_\gamma  = \LL_\gamma^{1^*\beta} \otimes
    \LL_{\beta}^{1^*\alpha}$. This leads to the inclusion 
    \[
    \Sym_2 G \otimes \LL_\gamma^{1^*\beta} \otimes
    \LL_{\beta}^{1^*\alpha} \hookrightarrow \left(G \otimes
      \LL_\gamma^{1^*\beta}\right) \otimes \left( G \otimes
      \LL_{\beta}^{1^*\alpha}\right)\,,
    \]
    so we obtain the \emph{anti-commutation relations.}

  \item \emph{$\gamma'$ is obtained from $\alpha'$ by deleting $2$
      boxes from a column.}  In this case there is again a unique
    $\beta'$ such that $\alpha'\nearrow\beta'\nearrow\gamma'$. We find
    the corresponding relations
    \[
    \Wedge^2 G\otimes \LL^{1^*\beta}_\gamma\otimes
    \LL^{1^*\alpha}_\beta \hookrightarrow (G\otimes
    \LL^{1^*\beta}_\gamma) \otimes (G \otimes \LL^{1^*\alpha}_\beta)\,,
    \]
    that is, the \emph{commutation relations}.

  \item\emph{$\gamma'$ is obtained from $\alpha'$ by deleting $2$ boxes
      not in the same row or column.} There are now two  distinct
    $\beta'_1$, $\beta'_2$ such that $\alpha'\nearrow\beta'_1\nearrow\gamma'$,
    $\alpha'\nearrow\beta'_2\nearrow\gamma'$. The corresponding relations are now
    relations between paths going $\alpha'\to\beta'_1\to \gamma'$ and
    $\alpha'\to\beta'_2\to \gamma'$:
    \[
      \Sym_2 G\otimes \LL^{[2]^*\alpha}_{\gamma}
      \hookrightarrow 
      (G\otimes \LL^{1^*\beta_1}_\gamma) \otimes (G \otimes \LL^{1^*\alpha}_{\beta_1})\oplus
      (G\otimes \LL^{1^*\beta_2}_\gamma) \otimes (G \otimes \LL^{1^*\alpha}_{\beta_2})
    \]
    and
    \[
      \Wedge^2 G\otimes \LL^{[11]^*\alpha}_{\gamma}
      \to (G\otimes \LL^{1^*\beta_1}_\gamma) \otimes (G \otimes \LL^{1^*\alpha}_{\beta_1})\oplus
      (G\otimes \LL^{1^*\beta_2}_\gamma) \otimes (G \otimes \LL^{1^*\alpha}_{\beta_2})\,.
    \]
\end{enumerate}
\end{nsit}

\begin{nsit}
  {Mixed relations.}\label{sit:mix-rels}
  Finally we discuss the anti-commutativity relations between $F^\svee\otimes U$
  and $G\otimes U^\svee$.  They are defined by the image of the map
  defined by the identity, the swap, and the trace:
  \begin{equation}
    \label{threecomponents}
    (F^\svee \otimes U) \otimes (G \otimes U^\svee)
    \xrightarrow{\left[\begin{smallmatrix}\id \\ \tau  \\ -\Tr
        \end{smallmatrix}\right]} 
    \mbox{\ensuremath{
        \begin{array}{c}
          (F^\svee\otimes U)\otimes (G\otimes U^\svee) \\ 
          {}\oplus{} \\ 
          (G\otimes U^\svee)\otimes (F^\svee\otimes U)\,.\\
          {}\oplus{} \\ 
          (F^\svee\otimes G)
        \end{array}
      }}
  \end{equation}
  The summands on the right-hand side are living in the obvious places
  in the tensor algebra $T_{S} ((\tilde F^\svee\oplus
  \tilde G )\otimes S)$; in particular the third summand sits
  inside the degree zero part of the tensor algebra, which is $S$.  

  We apply  $Q(-)^{\alpha'}_{\gamma'}$ to the  components of
  \eqref{threecomponents}, using the canonical isomorphisms
  \begin{align}
    \label{eq:mix1-app}\LL^{11^\ast \alpha}_\gamma&\cong \bigoplus_{\beta'}  \LL^{1\beta}_\gamma
    \otimes \LL^{1^\ast\alpha}_{\beta}\\
    \label{eq:mix2-app}\LL^{11^\ast \alpha}_\gamma&\cong \bigoplus_{\beta'}  \LL^{1^\ast\beta}_\gamma
    \otimes \LL^{1\alpha}_{\beta}\,.
  \end{align}
  We see first that if $\alpha \neq \gamma$ then the third component
  of the target vanishes:
  \begin{equation*}
    Q(F^\svee\otimes G)^{\alpha'}_{\gamma'}\\
    =F^\svee\otimes G \otimes \LL^{\alpha}_\gamma =0
  \end{equation*}
  since $\LL^{\alpha}_\gamma=\delta_{\alpha,\gamma} K$.  Therefore
  when $\alpha \neq \gamma$ the direct sums appearing in the
  quiverizations of the first two components
  \begin{align*}
    F^\svee\otimes G\otimes \LL^{11^\ast \alpha}_\gamma 
    &= Q(F^\svee\otimes U\otimes G\otimes U^\svee)^{\alpha'}_{\gamma'} \\
    &\to  \bigoplus_{\beta'}Q(F^\svee\otimes U)^{\beta'}_{\gamma'}\otimes Q(G\otimes
    U^\svee)^{\alpha'}_{\beta'} \\
    &= \bigoplus_{\beta'} (F^\svee \otimes
    \LL^{1\beta}_\gamma ) \otimes (G \otimes
    \LL^{1^\ast\alpha}_{\beta} )
  \end{align*}
and 
  \begin{align*}
    F^\svee\otimes G\otimes \LL^{11^\ast \alpha}_\gamma
    &= Q(F^\svee\otimes U\otimes G\otimes U^\svee)^{\alpha'}_{\gamma'}\\
    &\to \bigoplus_{\beta'}Q(G\otimes U^\svee)^{\beta'}_{\gamma'}\otimes Q(F^\svee\otimes
    U)^{\alpha'}_{\beta'} \\
    &= \bigoplus_{\beta'} (G \otimes
    \LL^{1^\ast\beta}_\gamma ) \otimes (F^\svee \otimes \LL^{1\alpha}_{\beta}
    )
  \end{align*}
  have exactly one summand each, and are thus of the form 
  \[
  F^\svee \otimes G \otimes \LL^{11^*\alpha}_\gamma \to (F^\svee
  \otimes \LL^{1\beta_1}_\gamma) \otimes (G \otimes
  \LL^{1^*\alpha}_{\beta_1})
  \]
  and
  \[
  F^\svee \otimes G \otimes \LL^{11^*\alpha}_\gamma \to  (G \otimes
  \LL^{1\alpha}_{\beta_2})\otimes (F^\svee
  \otimes \LL^{1^*\beta_2}_\gamma)
  \]  
  for some partitions $\beta_1,\ \beta_2$ with $\beta_1'\nearrow \alpha'
  \nearrow \beta'_2$ and $\beta'_1 \nearrow \gamma' \nearrow \beta'_2$.
  The image in this case is thus $F^\svee \otimes G$.

  If $\alpha = \gamma$, then we discard the degree-zero relations
  expressing the orthogonality of the idempotents corresponding to the
  vertices and need only consider the image of $F^\svee \otimes G
  \otimes \Tr_0 U$ in $(F^\svee \otimes U) \otimes (G \otimes U^\svee)
  \oplus (G \otimes U^\svee) \otimes (F^\svee \otimes U)$, where
  $\Tr_0U = \ker(\Tr\colon U^\svee \otimes U \to K)$.  The direct sums
  appearing in \eqref{eq:mix1-app} and \eqref{eq:mix2-app} have one non-zero
  summand for each partition $\beta'$ such that $\alpha' \nearrow
  \beta'$ and $\beta'$ has at most $m-l$ rows (so that $L^{\beta'}U
  \neq 0$).  That is, they have $t(\alpha)$ direct summands.  Since
  $\LL^{11^*\alpha}_\alpha = K \oplus \Tr_0U$, the image of $F^\svee
  \otimes G \otimes \Tr_0U$ is $(F^\svee \otimes G)^{\oplus(t(\alpha)-1)}$.
\end{nsit}

Arguments parallel to those in Lemma~\ref{lem:surjrels} and
Theorem~\ref{thm:quiveriso} now prove the following.

\begin{theorem}
  \label{thm:quiverisoQC}
  The homomorphism $Q_{B_{l,m-l}}(C) \to A = \End_{\caloz}(\caln)$ is
  an isomorphism. \qed
\end{theorem}

\begin{remark}
  \label{rem:char-free-QC}
  The description of the non-commutative desingularization as a
  quiverized Clifford algebra depends essentially on characteristic
  zero, relying as it does on the canonical direct-sum decompositions
  of representations of $\GL(U)$ into irreducibles. In retrospect, it
  was the fact that the torus $\GL(1)$ is linearly reductive in all
  characteristics that allowed us to prove the analogous result for
  the case of maximal minors in a characteristic-free manner. 
\end{remark}

\begin{remark}
  \label{rem:modulispace}
  Using the description above of the non-commutative desingularization
  as a quiverized Clifford algebra, one can prove an analogue
  of~\cite[Theorem D]{Buchweitz-Leuschke-VandenBergh:2010}, to the
  effect that $\calz$ is the fine moduli space for certain
  representations of the truncated Young quiver. The details are
  essentially identical to those in~\cite[section
  8]{Buchweitz-Leuschke-VandenBergh:2010}, so we don't pursue this
  direction further.
\end{remark}

\newcommand{\arxiv}[2][AC]{\mbox{\href{http://arxiv.org/abs/#2}{\textsf{arXiv:#2}}}}
\newcommand{\oldarxiv}[2][AC]{\mbox{\href{http://arxiv.org/abs/math/#2}{\textsf{arXiv:math/#2[math.#1]}}}}
\renewcommand{\MR}[1]{%
  {\href{http://www.ams.org/mathscinet-getitem?mr=#1}{MR #1}.}}
\providecommand{\bysame}{\leavevmode\hbox to3em{\hrulefill}\thinspace}
\newcommand{\arXiv}[1]{%
  \relax\ifhmode\unskip\space\fi\href{http://arxiv.org/abs/#1}{arXiv:#1}}


\begin{thebibliography}{HVdB07}

\bibitem[AGP02]{Aguilar-Gitler-Prieto:2002}
Marcelo Aguilar, Samuel Gitler, and Carlos Prieto, \emph{Algebraic topology
  from a homotopical viewpoint}, Universitext, Springer-Verlag, New York, 2002,
  Translated from the Spanish by Stephen Bruce Sontz. \MR{1908260}

\bibitem[BFK12]{Ballard-Favero-Katzarkov:2013}
Matthew {Ballard}, David {Favero}, and Ludmil {Katzarkov}, \emph{{Variation of
  geometric invariant theory quotients and derived categories}}, preprint
  (2012), 63 pp., \url{http://arxiv.org/abs/1203.6643}.

\bibitem[BLV10]{Buchweitz-Leuschke-VandenBergh:2010}
Ragnar-Olaf Buchweitz, Graham~J. Leuschke, and Michel {Van den Bergh},
  \emph{Non-commutative desingularization of determinantal varieties, {I}},
  Invent. Math. \textbf{182} (2010), no.~1, 47--115. \MR{2672281}

\bibitem[BLV12]{Buchweitz-Leuschke-VandenBergh:Grass}
\bysame, \emph{On the derived category of {G}rassmannians in arbitrary
  characteristic}, preprint (2012), 24 pages \url{http://arxiv.org/abs/1006.1633}.

\bibitem[BO02]{Bondal-Orlov:2002}
Alexei~I. Bondal and Dmitri Orlov, \emph{Derived categories of coherent
  sheaves}, Proceedings of the International Congress of Mathematicians, Vol.
  II (Beijing, 2002), Higher Ed. Press (Beijing), 2002, pp.~47--56.
  \MR{1957019}

\bibitem[Bof88]{Boffi:1988}
Giandomenico Boffi, \emph{The universal form of the {L}ittlewood-{R}ichardson
  rule}, Adv.\ in Math. \textbf{68} (1988), no.~1, 40--63. \MR{931171}

\bibitem[DRS74]{Doubilet-Rota-Stein:1974}
Peter Doubilet, Gian-Carlo Rota, and Joel Stein, \emph{On the foundations of
  combinatorial theory. {IX}. {C}ombinatorial methods in invariant theory},
  Studies in Appl. Math. \textbf{53} (1974), 185--216. \MR{0498650}

\bibitem[Don11]{Donovan:2011}
Will {Donovan}, \emph{{Grassmannian twists on the derived category via
  spherical functors}}, preprint (2011), 44 pp.,
\url{http://arxiv.org/abs/1111.3774}.

\bibitem[DS12]{Donovan-Segal:2012}
Will {Donovan} and Ed~{Segal}, \emph{{Window shifts, flop equivalences and
  Grassmannian twists}}, preprint (2012), 38 pp.,
\url{http://arxiv.org/abs/1206.0219}. 

\bibitem[Ful97]{Fulton:1997}
William Fulton, \emph{Young tableaux}, London Mathematical Society Student
  Texts, vol.~35, Cambridge University Press, Cambridge, 1997, With
  applications to representation theory and geometry. \MR{1464693}

\bibitem[Har77]{Hartshorne}
Robin Hartshorne, \emph{Algebraic geometry}, Springer-Verlag, New York, 1977,
  Graduate Texts in Mathematics, No. 52. \MR{0463157}

\bibitem[HVdB07]{Hille-VdB:2007}
Lutz Hille and Michel Van~den Bergh, \emph{Fourier-{M}ukai transforms},
  Handbook of tilting theory, London Math. Soc. Lecture Note Ser., vol. 332,
  Cambridge Univ. Press, Cambridge, 2007, pp.~147--177. \MR{2384610}

\bibitem[Juc66]{Jucys:1966}
A.~A. Jucys, \emph{On the {Y}oung operators of symmetric groups}, Litovsk. Fiz.
  Sb. \textbf{6} (1966), 163--180. \MR{0202866}

\bibitem[Juc71]{Jucys:1971}
\bysame, \emph{Factorization of {Y}oung's projection operators for symmetric
  groups.}, Litovsk. Fiz. Sb. \textbf{11} (1971), 1--10. \MR{0290671}

\bibitem[Kap88]{Kapranov:1988}
Mikhail~M. Kapranov, \emph{On the derived categories of coherent sheaves on
  some homogeneous spaces}, Invent. Math. \textbf{92} (1988), no.~3, 479--508.
  \MR{939472}

\bibitem[MO92]{Maliakas-Olver:1992}
Mihalis Maliakas and Peter~J. Olver, \emph{Explicit generalized {P}ieri maps},
  J. Algebra \textbf{148} (1992), no.~1, 68--85. \MR{1161566}

\bibitem[Olv]{Olver}
Peter~J. Olver, \emph{{D}ifferential {H}yperforms {I}}, available from
  \url{http://www.math.umn.edu/\textasciitilde olver/}.

\bibitem[OR06]{Ottaviani-Rubei:2006}
Giorgio Ottaviani and Elena Rubei, \emph{Quivers and the cohomology of
  homogeneous vector bundles}, Duke Math. J. \textbf{132} (2006), no.~3,
  459--508. \MR{2219264}

\bibitem[OV96]{Okounkov-Vershik:1996}
Andrei Okounkov and Anatoly Vershik, \emph{A new approach to representation
  theory of symmetric groups}, Selecta Math. (N.S.) \textbf{2} (1996), no.~4,
  581--605. \MR{1443185}

\bibitem[Sam09]{Sam:2009}
Steven~V Sam, \emph{Computing inclusions of {S}chur modules}, J. Softw. Algebra
  Geom. \textbf{1} (2009), 5--10. \MR{2878669}

\bibitem[SW11]{Sam-Weyman:2011}
Steven~V Sam and Jerzy Weyman, \emph{Pieri resolutions for classical groups},
  J. Algebra \textbf{329} (2011), 222--259, see the revised and augmented
  version at \url{http://arxiv.org/abs/0907.4505v5}. \MR{2769324}

\bibitem[VdB04a]{VandenBergh:crepant} Michel Van~den Bergh,
  \emph{Non-commutative crepant resolutions (with some corrections)},
  The legacy of Niels Henrik Abel, Springer, Berlin, 2004,
  pp.~749--770. The updated [2009] version on arXiv has some minor
  corrections over the published version:
  \url{http://arxiv.org/abs/math/0211064}.  \MR{2077594}

\bibitem[VdB04b]{VandenBergh:flops}
\bysame, \emph{Three-dimensional flops and noncommutative rings}, Duke Math. J.
  \textbf{122} (2004), no.~3, 423--455. \MR{2057015}

\bibitem[Wey03]{Weyman:book}
Jerzy Weyman, \emph{Cohomology of vector bundles and syzygies}, Cambridge
  Tracts in Mathematics, vol. 149, Cambridge University Press, Cambridge, 2003.
  \MR{1988690}

\bibitem[WZ12]{Weyman-Zhao}
Jerzy {Weyman} and Gufang {Zhao}, \emph{{Noncommutative desingularization of
  orbit closures for some representations of $GL_n$}}, preprint
  (2012), 40 pp., \url{http://arxiv.org/abs/1204.0488v1}.

\end{thebibliography}

\def\cprime{$'$} \def\polhk#1{\setbox0=\hbox{#1}{\ooalign{\hidewidth
  \lower1.5ex\hbox{`}\hidewidth\crcr\unhbox0}}}
  \def\polhk#1{\setbox0=\hbox{#1}{\ooalign{\hidewidth
  \lower1.5ex\hbox{`}\hidewidth\crcr\unhbox0}}}
\providecommand{\bysame}{\leavevmode\hbox to3em{\hrulefill}\thinspace}
\providecommand{\MR}{\relax\ifhmode\unskip\space\fi MR }
\providecommand{\MRhref}[2]{%
  \href{http://www.ams.org/mathscinet-getitem?mr=#1}{#2}
}
\providecommand{\href}[2]{#2}

\end{document}